\pdfoutput=1
\documentclass[a4paper,11pt]{amsart}
\usepackage[hmarginratio={1:1},vmarginratio={1:1},lmargin=60.0pt,tmargin=60.0pt]{geometry}

\allowdisplaybreaks

\usepackage[numbers]{natbib}
\usepackage{latexsym,mathtools,amssymb,stmaryrd}
\usepackage{enumitem}
\usepackage{xcolor}
\usepackage{diagbox,makecell,caption}
\usepackage{tikz}
\usetikzlibrary{cd,decorations.markings,arrows.meta,positioning,calc,tqft}
\usepackage{aliascnt,etoolbox}
\usepackage[hypertexnames=false]{hyperref}
\usepackage{bookmark}
\usepackage{listings}

\setlist[enumerate]{itemsep=0.15cm,label=\emph{\upshape(\alph*)}}
\setlist[enumerate,2]{itemsep=0.15cm,label=\emph{\upshape(\roman*)}}

\definecolor{spinach}{RGB}{46,139,87}
\definecolor{tomato}{RGB}{255,99,71}
\definecolor{orchid}{RGB}{143,40,194}
\definecolor{lava}{RGB}{207,16,32}
\definecolor{cream}{RGB}{255,253,208}
\definecolor{verdigris}{RGB}{67,179,174}
\definecolor{mydarkblue}{RGB}{10,10,170}

\let\emph\relax
\DeclareTextFontCommand{\emph}{\bfseries\em}
\renewcommand{\dots}{\text{...}}

\newcommand{\eg}{\text{e.g.}}
\newcommand{\ie}{\text{i.e.}}
\newcommand{\Z}{\mathbb{Z}}
\newcommand{\K}{\mathbb{K}}
\newcommand{\C}{\mathbb{C}}
\newcommand{\End}{\mathrm{End}}
\newcommand{\Hom}{\mathrm{Hom}}
\newcommand{\id}{id}

\newcommand{\paraa}{\alpha}
\newcommand{\parab}{\beta}
\newcommand{\parac}{\gamma}

\newcommand{\cob}{\mathbf{2Cob}}
\newcommand{\mcob}[1]{\mathbf{MCob}_{#1}}
\newcommand{\mucob}[3]{\mathbf{MCob}_{#1, #2, #3}}

\tikzset{
anchorbase/.style={baseline={([yshift=#1]current bounding box.center)}},
anchorbase/.default={-0.5ex},
tinynodes/.style={font=\tiny,text height=0.25ex,text depth=0.05ex},
mor/.style={line width=0.75,color=black,fill=cream},
usual/.style={line width=1.2,color=black},
match/.style={line width=1.2, densely dotted, color=verdigris},
ribbon/.style={double, double distance=4pt, line width=0.8pt, black},
ribbon boundary/.style={line width=0.8pt, black},
vertex/.style={circle, draw, fill=white, minimum size=14pt, line width=0.8pt},
hol/.style={
decoration={markings,
post length=0.25mm,
pre length=0.25mm,
mark=at position #1 with {\node[circle,radius=0.15cm,inner sep=-1.2pt,draw,color=black,fill=white]{};}
},
postaction={decorate}
},
mob/.style={
decoration={markings,
post length=0.25mm,
pre length=0.25mm,
mark=at position #1 with {\node[circle,radius=0.15cm,inner sep=-1.2pt,draw, color=tomato,fill=tomato]{};}
},
postaction={decorate}
},
dot/.style={
decoration={markings,
post length=0.25mm,
pre length=0.25mm,
mark=at position #1 with {\node[circle,radius=0.5cm,inner sep=-1.5pt,color=black,fill=black]{};}
},
postaction={decorate}
},
dot/.default=1
}

\let\oldlightning\lightning
\renewcommand{\lightning}{\textcolor{tomato}{\pmb{\oldlightning}}}

\def\NewTheorem#1{%
\newaliascnt{#1}{equation}%
\newtheorem{#1}[#1]{#1}%
\aliascntresetthe{#1}%
\expandafter\def\csname #1autorefname\endcsname{#1}%
}
\def\equationautorefname~#1\null{(#1)\null}

\numberwithin{equation}{subsection}

\NewTheorem{Proposition}
\NewTheorem{Theorem}
\NewTheorem{Lemma}
\theoremstyle{definition}
\NewTheorem{Definition}
\AtEndEnvironment{Definition}{\null\hfill$\Diamond$}%
\NewTheorem{Notation}
\AtEndEnvironment{Notation}{\null\hfill$\Diamond$}%
\NewTheorem{Example}
\AtEndEnvironment{Example}{\null\hfill$\Diamond$}%
\theoremstyle{remark}
\NewTheorem{Remark}
\AtEndEnvironment{Remark}{\null\hfill$\Diamond$}%

\hypersetup{
pdftoolbar=true,
pdfmenubar=true,
pdffitwindow=false,
pdfstartview={FitH},
pdftitle={Categorification of some Penrose polynomials},
pdfauthor={Daniel W. Collison and Daniel Tubbenhauer},
pdfsubject={},
pdfcreator={Daniel W. Collison and Daniel Tubbenhauer},
pdfproducer={Daniel W. Collison and Daniel Tubbenhauer},
pdfkeywords={},
pdfnewwindow=true,
colorlinks=true,
linkcolor=mydarkblue,
citecolor=teal,
filecolor=magenta,
urlcolor=orchid,
linkbordercolor=lava,
citebordercolor=teal,
urlbordercolor=orchid,
linktocpage=true
}

\newcommand{\nnfootnote}[1]{%
\begin{NoHyper}
\renewcommand\thefootnote{}\footnote{#1}%
\addtocounter{footnote}{-1}%
\end{NoHyper}
}

\definecolor{codegreen}{rgb}{0,0.6,0}
\definecolor{codepurple}{rgb}{0.58,0,0.82}
\definecolor{backcolor}{rgb}{0.95,0.95,0.92}

\lstdefinestyle{mystyle}{
backgroundcolor=\color{backcolor},
commentstyle=\color{codegreen},
keywordstyle=\color{magenta},
stringstyle=\color{codepurple},
basicstyle=\ttfamily\footnotesize,
breakatwhitespace=false,
breaklines=true,
captionpos=b,
keepspaces=true,
showspaces=false,
showstringspaces=false,
showtabs=false,
tabsize=2
}
\def\makeautorefname#1#2{\csdef{#1autorefname}{#2}}
\makeautorefname{section}{Section}%
\makeautorefname{subsection}{Section}%
\makeautorefname{subsubsection}{Section}%

\begin{document}
\title[Categorification of some Penrose polynomials]{Categorification of some Penrose polynomials}
\author[D. W. Collison and D. Tubbenhauer]{Daniel W. Collison and Daniel Tubbenhauer}

\address{D.W.C.: The University of Sydney, School of Mathematics and Statistics F07, NSW 2006, Australia}
\email{daniel.collison@sydney.edu.au}

\address{D.T.: The University of Sydney, School of Mathematics and Statistics F07, Office Carslaw 827, NSW 2006, Australia, \href{http://www.dtubbenhauer.com}{www.dtubbenhauer.com}, https://orcid.org/0000-0001-7265-5047}
\email{daniel.tubbenhauer@sydney.edu.au}

\begin{abstract}
We construct doubly- and triply-graded Penrose-type homologies for ribbon graphs.
The construction is a TQFT-valued cube of resolutions built from
two-dimensional cobordisms, which may be nonorientable. Their Euler
characteristics recover specializations of some Penrose polynomials; in
particular, the four color case comes with a
refinement of the classical Penrose criterion.
\end{abstract}

\nnfootnote{\textit{Mathematics Subject Classification 2020.} Primary: 05C31, 57R56; Secondary: 05C15, 57K18, 57M15.}
\nnfootnote{\textit{Keywords.} Penrose polynomial, graph coloring, four color theorem, topological quantum field theories, categorification, nonorientable cobordisms.}

\addtocontents{toc}{\protect\setcounter{tocdepth}{1}}

\maketitle

\tableofcontents

\section{Introduction}\label{S:Introduction}

In this paper, we categorify some Penrose polynomials.

\subsection{Some history}

The Penrose polynomial is one of the classical meeting points of graph
coloring, topology, and diagrammatics. Its origin is Penrose's graphical
calculus for tensor webs \cite{Pen-negative}, which detects Tait colorings in the planar trivalent case: the evaluation at 3 counts
proper 3-edge colorings, or equivalently, 4-face colorings after planar duality.
Thus the Penrose polynomial is closely tied to the four color theorem, but in a
local, diagrammatic form. Aigner developed the corresponding Penrose polynomial
for plane graphs \cite{Aig-penrose}, which Ellis-Monaghan--Moffatt later extended to embedded graphs and related to twisted duality, transition
polynomials, and ribbon-graph structure \cite{ElMo-penrose}; see also
\cite{ElKaMo-edgecolour} for a useful account of how Penrose's coloring
calculus fits into the world of topological graph polynomials.

The feature that makes the Penrose polynomial so adaptable is that it is a
local state sum. Similarly to lattice models in statistical mechanics, the construction involves resolving local pieces, attaching weights, and summing over
the resulting states. 
The resulting polynomial is thus closely related to Jaeger's transition polynomial
\cite{Jaeger-transition} and the topological transition polynomial viewpoint
for ribbon graphs, as well as the Bollob\'as--Riordan polynomial,
one of the basic extensions of the Tutte polynomial to graphs on surfaces
\cite{BoRi-surfaces}, and also Chmutov's partial-duality framework
\cite{Chmutov-duality}; in all these instances, graph polynomials are not
isolated gadgets: they are different shadows of local operations on embedded
graphs.

The second relevant history is homological. Khovanov homology showed that a
local state sum can be the Euler characteristic of a more refined invariant
\cite{Kho-jones} (such a refinement is often called categorification), 
and Bar-Natan's formulation made the underlying
(1+1 or 2D) TQFT structure explicit \cite{BN-cobordisms}. Soon after,
similar ideas were applied to graph polynomials: for example,
Helme-Guizon--Rong categorified the chromatic polynomial
\cite{HeGuRo-chromatic}, Jasso-Hernandez--Rong categorified a version of the
Tutte polynomial \cite{JaRo-tutte}, and Loebl--Moffatt categorified the
chromatic polynomial of fatgraphs and the Bollob\'as--Riordan polynomial
\cite{LoMo-fatgraphs}, to name a few; while not technically used below, these papers
explain the guiding philosophy: a cube of local resolutions can remember more
than its Euler characteristic.

The closest relatives of the present paper come from two directions. The first direction
is the graph coloring homology story built from perfect matchings and TQFTs.
Baldridge constructed a cohomology theory for planar trivalent graphs equipped
with a perfect matching; its graded Euler characteristic is the 2-factor
polynomial \cite{Bal-cohomology}. Baldridge--McCarty then presented an unoriented TQFT interpretation of the Penrose
polynomial in which filtered homology counts face
colorings of ribbon graphs \cite{BaMc-graphcoloring}. The second direction is virtual link
homology. There is a specific ``funny face'' that appears in the cube of resolutions, which can be dealt with in a natural way by introducing nonorientability into
Bar-Natan's cobordism picture; see \eg \, \cite{Man-vKh,Ta-uhqft,Tu-vkh}. Both
perspectives are important for us: perfect matchings give the cube, and
nonorientable cobordisms provide the extra local map. 

The present paper makes use of the above ideas, with the crutial difference being that we use a different local ingredient.

\begin{Remark}
There is also a representation-theoretic reading of the above story. Already
Penrose's original graphical calculus is a tensor calculus: trivalent vertices
represent invariant tensors and edges represent contractions; in the planar
trivalent case, this is closely related to the \(\mathfrak{so}_3\) weight-system
evaluation (or equivalently, \(\mathrm{SO}_3\)-webs), which is called even
Temperley--Lieb or even Verlinde calculus. Using the same language, Tait colorings and
the four color theorem can be reformulated in representation-theoretic terms;
see \eg \,
\cite{BN-lie-4ct,MoPeSn-categories-trivalent-vertex,TeLi-the-tl-paper,Tu-web-reps,Ya-invariant-graphs}, and for categorified versions of this viewpoint using foams and gauge theory, see \eg \, \cite{KhRo-foam,KrMr-tait}.
\end{Remark}

\subsection{This paper's contribution}

We work with trivalent ribbon graphs equipped with a perfect matching.

\begin{Remark}
We will see in \autoref{SS:Graphs} that this is not a restriction.
\end{Remark}

Resolving the matching edges produces a cube of states. As usual, an edge of
the cube can either merge two circles or split one circle into two. In our setting,
there is a third possibility: one circle stays one circle, but the local
surface contribution is nonorientable. Algebraically, the first two moves are
controlled by the multiplication and comultiplication of a Frobenius algebra respectively,
while the third move is controlled by one extra endomorphism.

Before introducing the formal terminology, the algebraic input can therefore be
summarized as follows: take a commutative Frobenius algebra \(V\), with
multiplication \(\mu\) and comultiplication \(\Delta\), together with an
endomorphism \(m\colon V\to V\). The compatibility relations dictate that \(m\)
can be moved through multiplication, and that
\begin{gather}\label{Eq:Rel}
m^2=\mu\circ\Delta,
\end{gather}
which is a key relation required for the new square faces in the cube to
commute. The usual sign assignment then turns the cube into a
chain complex. In slogan form,
\[
\setlength{\fboxsep}{3pt}
\colorbox{blue!6}{
\begin{minipage}{0.78\textwidth}
\centering
\emph{A Penrose-type state sum can be lifted to a homology theory by allowing
one nonorientable local map.}
\end{minipage}}
\]
The corresponding nonorientable TQFTs (called M\"obius TQFTs or MTQFTs), together with the associated Frobenius algebras 
(called M{\"o}bius Frobenius algebras), are introduced later in \autoref{S:TQFT} and \autoref{S:Frob}. The construction itself
is the important point: a cube of resolutions with three local maps,
\(\mu,\Delta,m\).

Our main result is that the resulting homology is an invariant of
perfect matching ribbon graphs. Its graded Euler characteristic is the bracket
defined by the local relations in \autoref{D:mobpoly}, repeated here without further explanation:
\begin{align*}
\Bigg\langle\begin{tikzpicture}[anchorbase,yscale=-1,scale=0.75]
\draw[usual] (0,0) to (0.5,0.25);
\draw[usual] (0.5,0.25) to (1,0);
\draw[match] (0.5,0.25) to (0.5,1);
\draw[usual] (0,1.25) to (0.5,1);
\draw[usual] (1,1.25) to (0.5, 1);
\node at (0.5,0.25)[circle,fill,inner sep=1.55pt]{};
\node at (0.5,1)[circle,fill,inner sep=1.55pt]{};
\end{tikzpicture} \Bigg\rangle
\quad &= \quad
A \, \Bigg\langle\begin{tikzpicture}[anchorbase,yscale=-1,scale=0.75]
\draw[usual] (0,0) to (0,1);
\draw[usual] (0.5,0) to (0.5,1);
\end{tikzpicture} \Bigg\rangle \quad + \quad B \,\Bigg\langle\begin{tikzpicture}[anchorbase,yscale=-1,scale=0.75]
\draw[usual] (0,0) to (0.5,1);
\draw[usual] (0.5,0) to (0,1);
\end{tikzpicture} \Bigg\rangle, \\
\Bigg\langle\begin{tikzpicture}[anchorbase,yscale=-1,scale=0.75]
\filldraw[color=black, fill=white, thick](-1,0) circle (0.5);
\end{tikzpicture}\Bigg\rangle \quad &= \quad C, \\
\langle \Gamma_1 \sqcup \Gamma_2 \rangle \quad &= \quad \langle \Gamma_1 \rangle \cdot \langle \Gamma_2 \rangle,
\end{align*}
with
\[
A=1,\qquad
B=-q_1^{s_1}\cdots q_m^{s_m},\qquad
C=q\mathrm{dim}(V),
\]
where \(q\mathrm{dim}\) is the graded dimension, and the $q_i$ are grading variables. As
usual, the bracket is the
shadow of the homology.

We then construct an explicit family of examples. Given an integer \(n>0\), set
\begin{gather}\label{Eq:Vn}
V_n=\mathbb{Z}[\tfrac{1}{3n}][x,y]/(x^n,y^3-xy).
\end{gather}
A suitable Frobenius trace, together with a suitable choice of endomorphism \(m\), gives the
required algebraic input. The grading is by \(\Z/n\Z\times\Z/2\Z\) for \(n\)
even, and by \(\Z/n\Z\) for \(n\) odd, plus, in both cases, a $\Z$-grading from the homology.
The homology is therefore either doubly- or triply-graded, depending on the parity of $n$.

\begin{Remark}
The algebra
$\C[x,y]/(x^n,y^3-xy)$
has a nice topological model. Suppose that \(x\) and \(y\) have degree \(4\) and \(2\) respectively, and let $\mathbb{H}$ denote quaternions. Let \(\eta\) be the tautological quaternionic line bundle over
\(\mathbb{HP}^{n-1}\), viewed as a complex rank two bundle, and consider the
complex projective bundle
\[
\mathbb{P}(\eta\oplus\mathbb{C})\longrightarrow \mathbb{HP}^{n-1}.
\]
By the projective bundle formula,
\[
H^*\big(\mathbb{P}(\eta\oplus\mathbb{C});\C\big)
\cong \C[x,y]/(x^n,y^3-xy),
\]
up to the harmless sign convention. The algebra used
in this paper can therefore be viewed as a quaternionic analog of the familiar Frobenius
algebra \(\C[x]/(x^n)\cong H^*(\mathbb{CP}^{n-1};\C)\) that is crucial to the construction 
of link homologies.
\end{Remark}

The corresponding Euler characteristic recovers specializations of the Penrose
polynomial \(P(\Gamma,n)\): for \(n\) even, two specializations give
\(P(\Gamma,3n)\) and \(P(\Gamma,n)\), while for \(n\) odd one obtains
\(P(\Gamma,3n)\). In particular, the case \(n=4\) produces a homological
refinement of the Penrose polynomial test for four face colorability.

\begin{Remark}
The \(\Z/2\Z\)-grading is genuinely new as far as we can tell, and it appears
in the case relevant to the four color theorem; this gives a new
graded refinement of the classical Penrose polynomial criterion for
four face colorability: instead of asking only for a nonzero number at the end,
one can ask for a stronger positivity statement before collapsing the grading.
In this sense, the extra grading suggests a new homological route toward the
four color theorem.

It is also not a formal accident, which is remarkable in itself. 
As we will see, the grading is tied to the
nonorientable part of the construction, and it is essentially
controlled by the Dyck surface, and Dyck's theorem, which is (in polygon notation of surfaces):
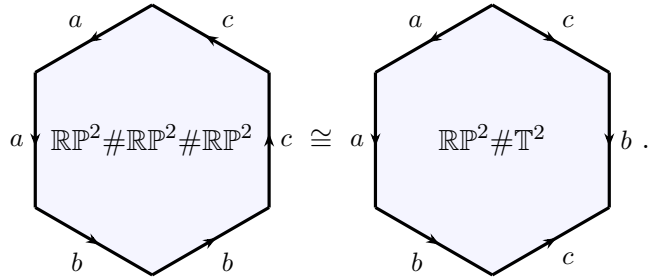
\begin{gather*}
\begin{tikzpicture}[
scale=1.05,
edgearrow/.style={
very thick,
postaction={
decorate,
decoration={
markings,
mark=at position .55 with {\arrow{Stealth[length=5pt]}}
}
}
},
lab/.style={font=\small},anchorbase
]
\coordinate (v1) at (90:1.7);
\coordinate (v2) at (150:1.7);
\coordinate (v3) at (210:1.7);
\coordinate (v4) at (270:1.7);
\coordinate (v5) at (330:1.7);
\coordinate (v6) at (30:1.7);
\fill[blue!4] (v1)--(v2)--(v3)--(v4)--(v5)--(v6)--cycle;
\draw[edgearrow] (v1)--(v2) node[midway, above left, lab] {$a$};
\draw[edgearrow] (v2)--(v3) node[midway, left, lab] {$a$};
\draw[edgearrow] (v3)--(v4) node[midway, below left, lab] {$b$};
\draw[edgearrow] (v4)--(v5) node[midway, below right, lab] {$b$};
\draw[edgearrow] (v5)--(v6) node[midway, right, lab] {$c$};
\draw[edgearrow] (v6)--(v1) node[midway, above right, lab] {$c$};
\node at (0,0) {$\mathbb{RP}^2\#\mathbb{RP}^2\#\mathbb{RP}^2$};
\end{tikzpicture}
\cong
\begin{tikzpicture}[
scale=1.05,
edgearrow/.style={
very thick,
postaction={
decorate,
decoration={
markings,
mark=at position .55 with {\arrow{Stealth[length=5pt]}}
}
}
},
lab/.style={font=\small},anchorbase
]
\coordinate (v1) at (90:1.7);
\coordinate (v2) at (150:1.7);
\coordinate (v3) at (210:1.7);
\coordinate (v4) at (270:1.7);
\coordinate (v5) at (330:1.7);
\coordinate (v6) at (30:1.7);
\fill[blue!4] (v1)--(v2)--(v3)--(v4)--(v5)--(v6)--cycle;
\draw[edgearrow] (v1)--(v2) node[midway, above left, lab] {$a$};
\draw[edgearrow] (v2)--(v3) node[midway, left, lab] {$a$};
\draw[edgearrow] (v3)--(v4) node[midway, below left, lab] {$b$};
\draw[edgearrow] (v4)--(v5) node[midway, below right, lab] {$c$};
\draw[edgearrow] (v6)--(v5) node[midway, right, lab] {$b$};
\draw[edgearrow] (v1)--(v6) node[midway, above right, lab] {$c$};
\node at (0,0) {$\mathbb{RP}^2\#\mathbb{T}^2$};
\end{tikzpicture}
.
\end{gather*}
As we will see, another way of saying this is that we use the monoid structure of 
our MTQFT.
\end{Remark}

The homology is stronger than its Euler characteristic, as expected, and we show
this already on small examples: two nonisomorphic perfect matching graphs can
have the same bracket but different homology; even more surprising, the bracket may vanish while the homology does not, cf. \autoref{S:Exmp}. In fact, for the family \(V_n\)
above, we show that the homology is strictly stronger than the bracket for every positive
integer \(n\). The final section includes the accompanying computer
calculations, which build the cube, insert the local maps, check that \(d^2=0\), and
compute the resulting Poincar\'e polynomial.

\begin{Remark}
Remarkably, there are many categorifications of the Penrose-type polynomials, which are all slightly different in nature.
Let us now describe how the present construction relates to other
categorifications of Penrose-type invariants. There are two nearby strands of
work.

First, Luse--Rong and Kauffman give Penrose-type categorifications in the
general graph-homology tradition \cite{LuRo-penrose-cat,
Kauffman-chrom-dichrom-penrose}. The relation to our construction is as
follows:
\begin{enumerate}

\item Luse--Rong categorify integer evaluations of the Penrose polynomial of a
graph. Their work is close to ours in name and motivation, but their
construction is not a TQFT-valued cube built from a M\"obius local operation.

\item Kauffman's framework is broader: it treats chromatic, dichromatic and
Penrose-type data using enhanced-state constructions. In particular, triply-graded Penrose-related homologies already appear there. Our point is not merely triply gradedness, but the specific TQFT mechanism producing it.

\end{enumerate}
Second, there is Baldridge--McCarty's TQFT approach to graph
coloring \cite{BaMc-graphcoloring}, which is also TQFT-based, uses
unoriented/nonorientable topology, and is designed to encode face colorings of
ribbon graphs. Nevertheless, the present construction is different in several
concrete ways:
\begin{enumerate}[resume]

\item Baldridge--McCarty essentially use the familiar one-variable Frobenius algebra
\(\C[x]/(x^n)\), the cohomology ring of \(\mathbb{CP}^{n-1}\), and its
filtered variant, for their theory. Our explicit family uses instead the two-variable 
algebra \(V_n\) of \autoref{Eq:Vn}, whose cohomological model is a 
projective bundle over quaternionic projective space, which is somewhat richer.

\item Baldridge--McCarty also have a one-circle-to-one-circle local map, but it
is not the M\"obius generator used here. In our construction the third local
map is the M\"obius endomorphism \(m\), and the crucial square relation is \autoref{Eq:Rel}. 

\item In the explicit family constructed below, the MTQFT mechanism
produces doubly- and triply-graded homologies; in particular, the four color
specialization is triply graded. They also use a shifted comultiplication in some cases to ensure that the differential preserves the grading, but we do not have this.

\end{enumerate}
The novelty here is not merely that Penrose-type invariants admit homological
lifts, nor merely that unoriented or nonorientable topology can be used. The
new point is the specific quaternionic-enriched MTQFT lift of the Penrose-type bracket. The
examples in \autoref{S:Exmp} then show that the resulting homology retains
information which is invisible to the bracket itself.
\end{Remark}

Finally, let us stress that our homology is very amenable to calculations, either by hand 
(using the partition-style and Deligne-type categories) or computer 
(code can be found in \cite{CoTu-code}).

\subsection{Structure of the paper}

The paper is organized as follows. In \autoref{S:TQFT}, we recall the relevant
nonorientable two dimensional TQFTs. In \autoref{S:Graphs}, we define the
Penrose-type bracket and construct the chain complex. In \autoref{S:Frob}, we
translate the construction into Frobenius-algebra-language. In
\autoref{S:ncolorchoice}, we introduce the family \(V_n\) and compare the Euler
characteristic with the Penrose polynomial. Finally, in \autoref{S:Exmp} we
compute examples and discuss the code.
\medskip

\noindent\textbf{Acknowledgments.}
DT thanks Scott Baldridge for inspiring the idea behind this paper 
during many enjoyable discussions in Korea, and thanks 
the Korea Institute for Advanced Study for hosting the visit. 
This paper is part of the first author's PhD thesis. DC was supported by the
Commonwealth through an Australian Government Research Training Program
Scholarship [https://doi.org/10.82133/C42F-K220]. DT is supported by ARC Future
Fellowship FT230100489. If the reader finds parts of the construction 
disorienting, then at least the exposition is topologically faithful.


\section{Möbius TQFTs}\label{S:TQFT}


We begin by introducing the notion of a (1+1 or 2D) \emph{Möbius topological quantum field theory} (\emph{MTQFT} for short). We assume the reader has some familiarity with 
TQFT constructions, see e.g. \cite{Ko-tqfts} for more details.

\subsection{The geometric version}

First recall the monoidal category $\mcob{}$, with the nonnegative integers as objects and nonorientable two-dimensional cobordisms as morphisms (see \cite{CoTu-mobius}, which was motivated by many works such as \cite{BaKaMc-unoriented-vkh,BaMc-graphcoloring,Cz-mobius-tqft,Tu-vkh,TuTu-utqft}). $\mcob{}$ has disjoint union as the monoidal product, and is symmetric, pivotal, and has a generator-relation presentation with generators:
\begin{gather}\label{Eq:Gen}
1_1 =\text{id}_1 \colon \begin{tikzpicture}[baseline ={([yshift=-.5ex]current bounding box.center)},scale=0.25]
\pic[
tqft,
incoming boundary components=1,
outgoing boundary components=1,
every lower boundary component/.style={draw},
genus=0,
offset=0,
scale=0.75,
draw,
thick,
boundary separation=40pt,
name=pop
]; 
\end{tikzpicture}, \quad
\mu\colon \begin{tikzpicture}[baseline ={([yshift=-.5ex]current bounding box.center)},scale=0.25]
\pic[
tqft,
incoming boundary components=1,
outgoing boundary components=2,
every lower boundary component/.style={draw},
genus=0,
offset=-0.5,
scale=0.75,
draw,
thick,
boundary separation=40pt,
name=pop
]; 
\end{tikzpicture}, \quad
\eta\colon \begin{tikzpicture}[baseline ={([yshift=-.5ex]current bounding box.center)},scale=0.25]
\pic[
tqft,
incoming boundary components=1,
outgoing boundary components=0,
every lower boundary component/.style={draw},
genus=0,
scale=0.75,
draw,
thick,
name=cup
]; 
\end{tikzpicture}, \quad
\Delta\colon   \begin{tikzpicture}[baseline ={([yshift=-.5ex]current bounding box.center)},scale=0.25]
\pic[
tqft,
incoming boundary components=2,
outgoing boundary components=1,
every lower boundary component/.style={draw},
genus=0,
offset=0.5,
scale=0.75,
draw,
thick,
boundary separation=40pt,
name=pop
]; 
\end{tikzpicture}, \quad
\epsilon\colon \begin{tikzpicture}[baseline ={([yshift=-.5ex]current bounding box.center)},scale=0.25]
\pic[
tqft,
incoming boundary components=0,
outgoing boundary components=1,
every lower boundary component/.style={draw},
genus=0,
scale=0.75,
draw,
thick,
name=cup
]; 
\end{tikzpicture}, \quad
s\colon  \begin{tikzpicture}[baseline ={([yshift=-.5ex]current bounding box.center)},scale=0.25,
tqft/.cd,
cobordism/.style={draw},
every upper boundary component/.style={draw},
every lower boundary component/.style={draw},
]
\pic [tqft/cylinder to next,anchor=incoming boundary 1,name=c, boundary separation=75pt,scale=0.75, draw, thick];
\pic [tqft/cylinder to prior,anchor=incoming boundary 1,
at=(c-outgoing boundary |- c-incoming boundary), boundary separation=75pt,scale=0.75 , draw, thick];
\end{tikzpicture},
\end{gather}
together with
\begin{gather}\label{Eq:MGen}
m=\begin{tikzpicture}[baseline ={([yshift=-.5ex]current bounding box.center)},scale=0.25]
\pic[
tqft,
incoming boundary components=1,
outgoing boundary components=1,
every lower boundary component/.style={draw},
genus=0,
scale=0.75,
draw,
thick,
boundary separation=40pt,
name=pop
]; 
\node at ([xshift=0pt, yshift=80pt]pop-outgoing boundary 1) {$\lightning$};
\end{tikzpicture},
\quad
\lightning
=
\raisebox{-0.1cm}{$\begin{tikzpicture}[scale=0.5, line cap=round, line join=round,anchorbase]
\coordinate (L) at (-2,0);
\coordinate (R) at ( 2,0);
\fill[orchid!8]
(L) .. controls (-0.8, 1.4) and (0.8, 1.4) .. (R)
.. controls (0.8,-1.4) and (-0.8,-1.4) .. cycle;
\draw[line width=1.0pt]
(L) .. controls (-0.8, 1.4) and (0.8, 1.4) .. (R);
\draw[line width=1.0pt]
(L) .. controls (-0.8,-1.4) and (0.8,-1.4) .. (R);
\draw[line width=1.0pt, -{Stealth[length=2.6mm]}]
(-0.15,1.05) -- (0.15,1.05);
\draw[line width=1.0pt, -{Stealth[length=2.6mm]}]
(0.15,-1.05) -- (-0.15,-1.05);
\node at (0, 1.55) {$a$};
\node at (0,-1.55) {$a$};
\fill (L) circle (1.3pt);
\fill (R) circle (1.3pt);
\node (B) at (0,-2.15) {};
\end{tikzpicture}$}
=\mathbb{RP}^2,
\end{gather}
which represents the connected sum of a cylinder and the real projective plane $\mathbb{RP}^2$ i.e. a Möbius strip glued to the identity cobordism \cite[Theorem 3.9]{CoTu-mobius}. The generators in \autoref{Eq:Gen} are called cylinder (or identity), pants, cup, pants-up, cap, and swap, and $m$ in \autoref{Eq:MGen} is called the Möbius generator.

We list all the relations between the generators below for completeness:
\begin{align}
s^2 &= 1_2, \label{startrel}\\
(1_1\otimes s) \circ (s \otimes 1_1) \circ (1_1 \otimes s) & = (s \otimes 1_1) \circ (1_1 \otimes s) \circ (s \otimes 1_1), \label{symm}\\
s \circ (1_1 \otimes \eta) &= \eta \otimes 1_1, \label{hol1}\\
(1_1 \otimes \mu) \circ (s \otimes 1_1)\circ (1_1 \otimes s) & = s \circ (\mu \otimes 1_1), \label{hol2} \\
(1_1 \otimes \epsilon) \circ s & = \epsilon \otimes 1_1, \label{hol3}\\
(1_1 \otimes s) \circ (s \otimes 1_1) \circ (1_1 \otimes \Delta) &= (\Delta \otimes 1_1) \circ s, \label{hol4}\\
\mu \circ (1_1 \otimes \eta) &= 1_1 = \mu \circ (\eta \otimes 1_1), \\
(1_1 \otimes \epsilon) \circ \Delta &= 1_1 = (\epsilon \otimes 1_1) \circ \Delta, \\
(\mu \otimes 1_1) \circ (1_1 \otimes \Delta) &= \Delta \circ \mu = (1_1 \otimes \mu) \circ (\Delta \otimes 1_1), \label{frob}\\
\mu \circ s &= \mu,  \label{endrel} 
\end{align}
together with
\begin{align}
\mu \circ (m \otimes 1_1 ) &= m \circ \mu = \mu \circ (1_1 \otimes m), \label{mobmu} \\
s \circ (m \otimes 1_1) & =  (1_1 \otimes m) \circ s, \label{mobs} \\
m^3 &=h \circ m, \label{mobrel}
\end{align}
where we define the handle as
\begin{gather*}
h\colon \begin{tikzpicture}[baseline ={([yshift=-.5ex]current bounding box.center)},scale=0.25]
\pic[
tqft,
incoming boundary components=1,
outgoing boundary components=1,
every lower boundary component/.style={draw},
genus=1,
scale=0.75,
draw,
thick,
boundary separation=40pt,
name=pop
]; 
\end{tikzpicture} \quad = \quad \begin{tikzpicture}[tqft, baseline ={([yshift=-.5ex]current bounding box.center)},scale=0.25, cobordism/.style={draw},
every upper boundary component/.style={draw},
every lower boundary component/.style={draw}]
\pic[tqft/pair of pants,draw, thick, name=a,scale=0.5];
\pic[tqft/reverse pair of pants,draw, thick,scale=0.5, name=b, anchor = incoming boundary 1, at =(a-outgoing boundary 1)];
\end{tikzpicture} \quad = \quad \mu \circ \Delta.
\end{gather*}

\begin{Remark}\label{rmk2cob}
The relations \autoref{startrel}--\autoref{endrel} with identical generators provide a generator-relation presentation for the (skeletal) category of oriented cobordisms $\cob$; see e.g. \cite{Ko-tqfts}. We would recommend visualizing these before continuing. Helpful are \emph{spine diagrams}:
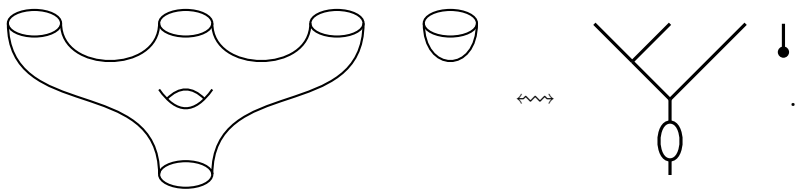
\begin{gather*}
\begin{tikzpicture}[
baseline ={([yshift=-.5ex]current bounding box.center)},anchorbase]
\pic[
tqft,
incoming boundary components=3,
outgoing boundary components=1,
every lower boundary component/.style={draw},
genus=1,
offset = 1,
draw,
thick,
name=ex1
];
\pic[tqft,
incoming boundary components=1,
outgoing boundary components=0,
draw,
thick,
name=cyldoub1,
every lower boundary component/.style={draw},
at={(5.5,0)},
];
\end{tikzpicture} \quad
\leftrightsquigarrow \quad 
\begin{tikzpicture}[anchorbase]
\draw[usual] (0,0) to (1,1);
\draw[usual] (0,0) to (-0.5,0.5);
\draw[usual] (-0.5,0.5) to (-1,1);
\draw[usual] (-0.5,0.5) to (0,1);
\draw[usual] (0,0) to (0,-0.3);
\draw[usual] (0,-1) to (0,-0.8);
\draw[usual] (0,-0.3) to[out=0,in=0] (0,-0.8);
\draw[usual] (0,-0.3) to[out=180,in=180] (0,-0.8);
\draw[usual,dot] (1.5,1) to (1.5,0.6);
\end{tikzpicture}
.
\end{gather*}
(We do not know the original reference for spine diagrams, but this visual simplification is well-known.)
\end{Remark}

When working with spine diagrams, let us use a hollow dot for handles $h$ and a filled dot for a Möbius generator $m$.
The relations \autoref{mobmu} and \autoref{mobs} in terms of spine diagrams are
\begin{gather*}
\begin{tikzpicture}[anchorbase]
\draw[usual,mob=0.25] (0,0) to (0.5,0.5) to (1,0);
\draw[usual] (0.5,0.5) to (0.5,1);
\end{tikzpicture}
=
\begin{tikzpicture}[anchorbase]
\draw[usual] (0,0) to (0.5,0.5) to (1,0);
\draw[usual,mob=0.5] (0.5,0.5) to (0.5,1);
\end{tikzpicture}
=
\begin{tikzpicture}[anchorbase]
\draw[usual,mob=0.75] (0,0) to (0.5,0.5) to (1,0);
\draw[usual] (0.5,0.5) to (0.5,1);
\end{tikzpicture}
\hspace{0.5cm} \text{and} \hspace{0.5cm}
\begin{tikzpicture}[anchorbase]
\draw[usual,mob=0.25] (0,0) to (1,1);
\draw[usual] (1,0) to (0,1);
\end{tikzpicture}
=
\begin{tikzpicture}[anchorbase]
\draw[usual,mob=0.75] (0,0) to (1,1);
\draw[usual] (1,0) to (0,1);
\end{tikzpicture},
\end{gather*}
Moreover, arguably the most exciting relation is \autoref{mobrel}, which, in spine pictures, is
\begin{gather*}
\begin{tikzpicture}[anchorbase]
\draw[usual, mob=0.25, mob=0.5, mob=0.75] (0,0) to (0,1);
\end{tikzpicture}
=
\begin{tikzpicture}[anchorbase]
\draw[usual, hol=0.66, mob=0.33] (0,0) to (0,1);
\end{tikzpicture}\quad
\left(\text{closed version: }
\begin{tikzpicture}[baseline ={([yshift=-.5ex]current bounding box.center)}, draw, thick,scale=0.25]
\draw (0,0) circle (2cm);
\draw (-2,0) arc (180:360:2 and 0.6);
\draw[dashed, thick] (2,0) arc (0:180:2 and 0.6);
\node at (0,-1){$\lightning$};
\node at (-0.66,0.8){$\lightning$};
\node at (0.66,0.8){$\lightning$};
\end{tikzpicture}
=
\begin{tikzpicture}[scale=0.25, draw, thick, baseline ={([yshift=-.5ex]current bounding box.center)}]
\draw (0,0) ellipse (3cm and 2cm);
\draw (-1.5, 0.2) arc (180:360:1.5 and 0.6);
\draw (1,-0.2) arc (0:180:1 and 0.6);
\node at (-2,0){$\lightning$};
\end{tikzpicture}
\right)
.
\end{gather*}
This is a well-known, famous, and surprising relation among surfaces.

\begin{Remark}
As explained in \cite{CoTu-mobius}, we note that the boundaries of all cobordisms in $\mcob{}$ are oriented with fixed orientation, which differs from the notion of unoriented cobordisms presented in e.g. \cite{Cz-mobius-tqft,Tu-vkh,TuTu-utqft}; in particular, we do not consider orientation-reversing diffeomorphisms of the boundaries.
\end{Remark}

The category $\mcob{}$ is symmetric and pivotal, using the swap and the pivotal structure induced by cup/cap (use $\Delta\circ\eta$ and $\epsilon\circ\mu$) as the relevant structures. We will always use these as structures, or obvious analogs, whenever we work with $\mcob{}$ or related categories.

\begin{Example}\label{E:cob}
Consider the following example:
\begin{gather*}
\begin{tikzpicture}[
baseline ={([yshift=-.5ex]current bounding box.center)}]
\pic[
tqft,
incoming boundary components=1,
outgoing boundary components=0,
every lower boundary component/.style={draw},
genus=0,
offset = 0.5,
draw,
thick,
name=ex1
];
\pic[
tqft,
incoming boundary components=1,
outgoing boundary components=2,
every lower boundary component/.style={draw},
genus=1,
offset = -0.5,
draw,
thick,
name=pop,
at={(2,0)},
]; 
\node at ([xshift=29pt, yshift=43pt]pop-outgoing boundary 1) {$\lightning$};
\node at ([yshift=8pt]ex1-incoming boundary 1) {$\Big\uparrow$};
\node at ([yshift=8pt]pop-incoming boundary 1) {$\Big\uparrow$};
\node at ([yshift=-8pt]pop-outgoing boundary 1) {$\Big\uparrow$};
\node at ([yshift=-8pt]pop-outgoing boundary 2) {$\Big\uparrow$};
\end{tikzpicture} \quad \in \quad \Hom_{\mcob{}}(2, 2),
\end{gather*}
where the orientation conventions for the boundary $\mathbb{S}^1$ are indicated by the arrows. As usual, we consider individual cobordisms as representatives for their corresponding diffeomorphism class. 
\end{Example}

\subsection{The linear algebra version}

Let $R$ be a commutative ring. Let $R\textbf{Mod}$ be the symmetric monoidal category of 
finite-dimensional free $R$-modules (finite-dimensional vector spaces if $R$ is a field) with the usual tensor product as the structure.

\begin{Definition}\label{D:Mobiusfrobtqft} 
A \emph{Möbius topological quantum field theory}, or \emph{MTQFT}, is a symmetric monoidal functor $\mathcal{F}\colon\mcob{} \rightarrow R\textbf{Mod}$.
\end{Definition}

\begin{Remark}
We could be more general in \autoref{D:Mobiusfrobtqft} by using finitely-generated projective $R$-modules instead of finite-dimensional free $R$-modules.
\end{Remark}

Note that one can use the pivotality of $\mcob{}$ to put a 
pairing on the $R$-modules in $\mathcal{F}(\mcob{})$.

\begin{Example}\label{E:mobn1}
Fix $R = \mathbb{Z}[\frac{1}{3}]$. Consider the MTQFT defined by $\mathcal{F}(1) = R[y]/(y^3) $ on objects. On morphisms, $\mathcal{F}(\mu)$ is the usual polynomial ring multiplication with unit $\mathcal{F}(\mathcal{\eta})(1) = 1$, and given a basis $\{1, y, y^2 \}$ for $R[y]/(y^3) $, $\mathcal{F}(\epsilon)(y^2)=3$ and zero on the other basis elements. Furthermore, we have
\begin{align*}
\mathcal{F}(\Delta)(1) &=\tfrac{1}{3}(1 \otimes y^2+y\otimes y + y^2 \otimes 1),        &  \mathcal{F}(\Delta)(y) &= \tfrac{1}{3}(y \otimes y^2+y^2\otimes y),             &  \mathcal{F}(\Delta)(y^2)&=\tfrac{1}{3}y^2 \otimes y^2,\\
\mathcal{F}(m)(1)&=y,        &  \mathcal{F}(m)(y)&=y^2,  &  \mathcal{F}(m)(y^2)&=0, 
\end{align*}
and $\mathcal{F}(s)$ is the usual swap map for tensor products of modules.
\end{Example}

Note the appearance of $\frac{1}{3}$, which looks a bit surprising at first glance. However, it is really an artifact of the relation \autoref{mobrel}.

\begin{Remark}\label{rm:nophi}
Let us compare briefly with the standard Turaev--Turner framework for unoriented
two-dimensional TQFTs \cite{TuTu-utqft}. In that setting, the algebraic input is
an extended Frobenius algebra: besides the Frobenius structure, there is an
orientation-reversing involution, often denoted \(\phi\), together with a
M\"obius element. The resulting structure has been used and developed in several
directions; see \eg \,
\cite{CzKeQuWa-extended,GagnonRirieYoung-real-tqft}.

Our convention is slightly different. The boundary circles in our cobordism
category have fixed orientations, and we do not use the twist \(\phi\) itself
as a gluing operation. What remains visible in our diagrammatics is the
M\"obius generator. While we lose the boundary-orientation
bookkeeping present in the full Turaev--Turner setting, which essentially amounts to a 
\(\Z/2\Z\) factor, the advantage, in our opinion, is that the resulting diagrammatics are simpler:
they form a plain diagram-algebra-style calculus, with the M\"obius generator
appearing as an ordinary local diagrammatic operation.
\end{Remark}


\section{Graph invariants}\label{S:Graphs}

We now define graph homologies. The following section is strongly inspired by 
\cite{BaMc-graphcoloring}, but also by \cite{HeGuRo-chromatic} and related works.
We expect that the reader has basic background in graph theory, standard constructions 
in quantum topology, and link homologies. See e.g. \cite{LaZv-graphs-surfaces}, \cite{Tu-qt,TuVi-monoidal-tqft}, and \cite{Kho-jones,BN-cobordisms} for some references. 


\subsection{Preliminaries}\label{SS:Graphs}

We begin by recalling some relevant terminology.


\begin{Notation}
By \emph{graph} we mean an undirected multigraph i.e. $G = (V, E, r)$ where $V$ is a nonempty set of vertices, $E$ a set of edges, and $r: E \rightarrow \big\{\{v_1, v_2\}: v_1, v_2 \in V\big\}$ (using multiset notation, we write $e=\{1,2\}$ for $r(e)=\{1,2\}$ etc.). Also, we require that $|V|, |E| < \infty$ i.e. finite graphs only.
\end{Notation}

\begin{Definition}\label{D:ribbon}
A \emph{ribbon graph of a graph} $G$ is an embedding $i: G \rightarrow \Gamma$, where $\Gamma$ is a surface with boundary that deformation retracts onto $i(G)$. We refer to $G$ as the \emph{underlying graph} of $\Gamma$ and to $\Gamma$ as the \emph{associated surface} of the ribbon graph.
\end{Definition}

We stress that such an embedding is extra structure: every graph can be made into a ribbon graph, potentially in multiple ways.

\begin{Example}
Here is an example:
\begin{gather*}
G=(\{1,2\},\big\{\{1,1\},\{1,2\},\{1,2\}\big\})
,\quad
\Gamma=
\begin{tikzpicture}[anchorbase]
\node[vertex] (v1) at (0,0) {};
\node[vertex] (v2) at (3,0) {};
\draw[ribbon] (v1) -- (v2);
\draw[ribbon] (v1) to[out=60, in=120] (v2);
\draw[ribbon] (v1) to[out=135, in=225, looseness=5] (v1);
\end{tikzpicture}
,
\end{gather*}
where we have thickened up the (images of the) vertices for readability, as we will do throughout.
\end{Example}

An orientation of a ribbon graph $i: G \rightarrow \Gamma$ is given by an orientation of its associated surface $\Gamma$, if it exists. Furthermore, let $\bar{\Gamma}$ denote the closed surface obtained by attaching discs to the boundary of $\Gamma$. Two ribbon graphs $i_1: G_1 \rightarrow \Gamma_1$ and $i_2: G_2 \rightarrow \Gamma_2$ are \emph{equivalent}  ribbon graphs if there is a homeomorphism $f: \bar{\Gamma}_1 \rightarrow \bar{\Gamma}_2$ that induces an isomorphism of graphs from $G_1$ to $G_2$.

\begin{Notation}
It is customary to denote a ribbon graph $i: G \rightarrow \Gamma$ simply by the associated surface $\Gamma$.
\end{Notation}

In fact, one can represent oriented ribbon graphs as graphs in the plane, with vertices drawn as black dots and edges as curves, and at any one point, at most two of such curves can intersect. We call such a diagram a \emph{ribbon diagram}. Every oriented ribbon graph has many associated ribbon diagrams.
Moreover, we recall the following Reidemeister-type theorem for ribbon graphs:

\begin{Lemma}
Two ribbon diagrams represent equivalent oriented ribbon graphs if and only if they are related by a finite sequence of \emph{planar isotopy}, \emph{Reidemeister(-type) moves} and \emph{fork moves}. The planar isotopies are of the form (this is not an exhaustive list):
\begin{gather}\label{eq:plane}
\begin{tikzpicture}[anchorbase,scale=1]
\draw[usual] (0,0.5) to (0,0) to[out=270,in=180] (0.25,-0.25) to[out=0,in=270] (0.5,0);
\draw[usual] (0.5,0) to[out=90,in=180] (0.75,0.25) to[out=0,in=90] (1,0) to (1,-0.5);
\end{tikzpicture}
 \longleftrightarrow 
\begin{tikzpicture}[anchorbase,scale=1]
\draw[usual] (0,-0.5) to (0,0.5);
\end{tikzpicture}
\longleftrightarrow 
\begin{tikzpicture}[anchorbase,scale=1]
\draw[usual] (0,0.5) to (0,0) to[out=270,in=0] (-0.25,-0.25) to[out=180,in=270] (-0.5,0);
\draw[usual] (-0.5,0) to[out=90,in=0] (-0.75,0.25) to[out=180,in=90] (-1,0) to (-1,-0.5);
\end{tikzpicture}
,\quad \quad
\begin{tikzpicture}[anchorbase,scale=0.2]
\draw [usual] (0,-1) to (0,.75);
\draw [usual] (0,.75) to [out=30,in=270] (1,2.5);
\draw [usual] (0,.75) to [out=150,in=270] (-1,2.5); 
\draw [usual] (1,2.5) to [out=90,in=180] (2,3.5) to [out=0,in=90] (3,2.5);
\draw [usual] (-1,2.5) to [out=90,in=180] (2,4.5) to [out=0,in=90] (5,2.5);
\draw [usual] (3,-1) to (3,2.5);
\draw [usual] (5,-1) to (5,2.5);
\end{tikzpicture}
\longleftrightarrow 
\begin{tikzpicture}[anchorbase,scale=0.2]
\draw [usual] (4, .75) to (4,2.5);
\draw [usual] (5,-1) to [out=90,in=330] (4,.75);
\draw [usual] (3,-1) to [out=90,in=210] (4,.75);
\draw [usual] (0, -1) to (0,2.5);
\draw [usual] (0,2.5) to [out=90,in=180] (2,4.5) to [out=0,in=90] (4,2.5);
\end{tikzpicture}
,\quad \quad 
\begin{tikzpicture}[anchorbase,scale=0.2]
\draw [usual] (0,-1) to (0,.75);
\draw [usual] (0,.75) to [out=30,in=270] (1,2.5);
\draw [usual] (0,.75) to [out=150,in=270] (-1,2.5); 
\draw [usual] (1,2.5) to [out=90,in=180] (2,3.5) to [out=0,in=90] (3,2.5);
\draw [usual] (-1,2.5) to [out=90,in=180] (2,4.5) to [out=0,in=90] (5,2.5);
\draw [usual] (3,-1) to (3,2.5);
\draw [usual] (5,-1) to (5,2.5);
\end{tikzpicture}
\longleftrightarrow 
\begin{tikzpicture}[anchorbase,scale=0.2]
\draw [usual] (4, .75) to (4,2.5);
\draw [usual] (5,-1) to [out=90,in=330] (4,.75);
\draw [usual] (3,-1) to [out=90,in=210] (4,.75);
\draw [usual] (0, -1) to (0,2.5);
\draw [usual] (0,2.5) to [out=90,in=180] (2,4.5) to [out=0,in=90] (4,2.5); 
\end{tikzpicture}
.
\end{gather}
The Reidemeister-type moves mirror the three Reidemeister moves for edges away from the vertices, i.e.
\begin{gather}\label{eq:reide}
\begin{tikzpicture}[anchorbase,scale=1]
\draw[usual] (0.5,0.75) to[out=270,in=0] (0.25,0.5) to[out=180,in=270] (0,1.25) to (0,1.5);
\draw[usual] (0,0) to (0,0.25) to[out=90,in=180] (0.25,1) to[out=0,in=90] (0.5,0.75);
\end{tikzpicture}
\quad \longleftrightarrow \quad
\begin{tikzpicture}[anchorbase,scale=1]
\draw[usual] (0,0) to[out=90,in=270] (0,1.5);
\end{tikzpicture}
,\quad \quad
\begin{tikzpicture}[anchorbase,scale=1]
\draw[usual] (0.5,0) to[out=90,in=270] (0,0.75) to[out=90,in=270] (0.5,1.5);
\draw[usual] (0,0) to[out=90,in=270] (0.5,0.75) to[out=90,in=270] (0,1.5);
\end{tikzpicture}
\quad \longleftrightarrow \quad
\begin{tikzpicture}[anchorbase,scale=1]
\draw[usual] (0,0) to[out=90,in=270] (0,1.5);
\draw[usual] (0.5,0) to[out=90,in=270] (0.5,1.5);
\end{tikzpicture}
,\quad \quad
\begin{tikzpicture}[anchorbase,scale=1]
\draw[usual] (1,0) to[out=90,in=270] (0,1.5);
\draw[usual] (0.5,0) to[out=90,in=270] 
(0,0.75) to[out=90,in=270] (0.5,1.5);
\draw[usual] (0,0) to[out=90,in=270] (1,1.5);
\end{tikzpicture}
\quad \longleftrightarrow \quad
\begin{tikzpicture}[anchorbase,scale=1]
\draw[usual] (1,0) to[out=90,in=270] (0,1.5);
\draw[usual] (0.5,0) to[out=90,in=270] 
(1,0.75) to[out=90,in=270] (0.5,1.5);
\draw[usual] (0,0) to[out=90,in=270] (1,1.5);
\end{tikzpicture}
,
\end{gather}
while the fork move at vertices is given by 
\begin{gather}\label{ribbonmove4}
\begin{tikzpicture}[anchorbase,yscale=-1,scale=1.25]
\draw[usual] (0,0) to (0.5,0.25);
\draw[usual] (0.5,0.25) to (1,0);
\draw[usual] (0.5,0.25) to (0.25,0.8);
\draw[usual] (0.5,0.25) to (0.75,0.8);
\draw[usual] (0.95,0.2) arc (0:180:0.45);
\node at (0.5,0.25)[circle,fill,inner sep=1.55pt]{};
\node at (0.5,-0.05){$\cdots$};
\node at (0.5,0.75){$\cdots$};
\end{tikzpicture} \quad \longleftrightarrow \quad \begin{tikzpicture}[anchorbase,yscale=-1,scale=1.25]
\draw[usual] (0,-0.2) to (0.5,0.25);
\draw[usual] (0.5,0.25) to (1,-0.2);
\draw[usual] (0.5,0.25) to (0.25,0.6);
\draw[usual] (0.5,0.25) to (0.75,0.6);
\draw[usual] (0.05,0.3) arc (180:360:0.45);
\node at (0.5,0.25)[circle,fill,inner sep=1.55pt]{};
\node at (0.5,-0.3){$\cdots$};
\node at (0.5,0.6){$\cdots$};
\end{tikzpicture}.
\end{gather}
We call all of these taken together as the \emph{ribbon moves}.
\end{Lemma}

\begin{proof}
This is a nontrivial theorem, see e.g. \cite[Theorem 2.7]{BaKaRu-ribbonvirtual}. The only difference for us is \autoref{ribbonmove4}; \cite{BaKaRu-ribbonvirtual} only considers ribbon graphs whose underlying graphs are trivalent, but the argument is identical in the more general setting (recall that a graph is \emph{trivalent} or \emph{cubic} if every vertex has degree three).
\end{proof}

A ribbon diagram can be used to obtain a ribbon graph by using discs for vertices and bands for edges, as illustrated in \autoref{E:basic} below.

\begin{Example}\label{E:basic}
Given the ribbon diagram
\begin{gather*}
\begin{tikzpicture}[anchorbase,yscale=-1,scale=1.25]
\draw[usual] (0,1) to [ out =180 , in = 180](0,0);
\draw[usual] (0,0) to [ out =0 , in = 0](0,1);
\draw[usual] (0,0) to (0,1);
\node at (0,1)[circle,fill,inner sep=1.55pt]{};
\node at (0,0)[circle,fill,inner sep=1.55pt]{};
\end{tikzpicture}
,
\end{gather*}
we can associate a ribbon graph with underlying graph $G = (V, E,r)$, with vertices $V= \{0,1\}$ and edges $E = \{e, f, g\}$ with $e=f=g = \{0,1\}$, and associated surface
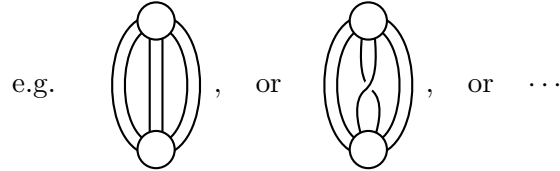
\begin{gather*}
\eg \, \hspace{0.5cm}
\begin{tikzpicture}[anchorbase]
\coordinate (a) at (0,1.7);
\coordinate (b) at (0,0);
\draw[ribbon] (a) to[out=180,in=180] (b);
\draw[ribbon] (a) to[out=270,in=90] (b);
\draw[ribbon] (a) to[out=0,in=0] (b);
\node[vertex] at (a) {};
\node[vertex] at (b) {};
\end{tikzpicture} 
,\quad \text{or} \quad
\begin{tikzpicture}[anchorbase]
\coordinate (a) at (0,1.7);
\coordinate (b) at (0,0);
\draw[ribbon] (a) to[out=180,in=180] (b);
\draw[ribbon] (a) to[out=0,in=0] (b);
\draw[ribbon boundary]
(-0.08,1.55)
.. controls (-0.12,1.15) and (-0.10,0.90) .. (0.05,0.78)
.. controls (0.20,0.65) and (0.15,0.35) .. (0.08,0.15);
\draw[ribbon boundary,preaction={draw=white,line width=4pt}]
(0.08,1.55)
.. controls (0.12,1.15) and (0.10,0.90) .. (-0.05,0.78)
.. controls (-0.20,0.65) and (-0.15,0.35) .. (-0.08,0.15);
\node[vertex] at (a) {};
\node[vertex] at (b) {};
\end{tikzpicture}
, \quad \text{or} \quad \cdots
\end{gather*}
where half twists may be added to the bands resulting in a potentially nonorientable surface \ie \, there may be many choices of associated surface for a given underlying graph.
\end{Example}

\begin{Notation}
We will not use unoriented ribbon graphs in this paper, and we will say ribbon graph instead of oriented ribbon graph for short.
\end{Notation}

\begin{Remark}
In the case of unoriented ribbon graphs, we obtain an almost identical correspondence except that ribbon diagrams are replaced with signed ribbon diagrams. For our purposes though, we only consider oriented ribbon graphs, which we will think of as equivalence classes of ribbon diagrams related by ribbon moves and planar isotopy. 
\end{Remark}

We now come to our main object of study. First, we recall the following:

\begin{Definition}\label{D:perfectgraph}
A \emph{perfect matching} of a graph $G = (V, E, r)$ is a subset of edges $M \subset E$ such that each vertex is incident to exactly one edge in $M$.
\end{Definition}

\begin{Example}
Given the graph $G = (V, E, r)$ from \autoref{E:basic}, the set $M=\{f\}$ gives a perfect matching of $G$; $M = \{e\},$ or $M =\{g\}$ would also work.
\end{Example}

\begin{Definition}\label{D:perfectribbon}
A \emph{perfect matching graph}, denoted $\Gamma_M$, is a ribbon graph $i: G \rightarrow \Gamma$ together with a perfect matching $M$ of the underlying graph $G$. The perfect matching is represented using blue dotted edges on the corresponding ribbon diagrams.
\end{Definition}

\begin{Example}\label{E:permatexm}
Examples of perfect matching graphs are given below:
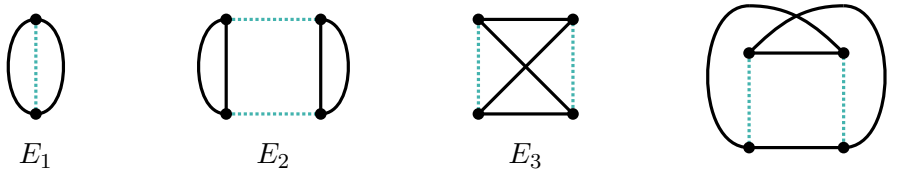
\begin{gather*}
\begin{tikzpicture}[anchorbase,yscale=-1,scale=1.25]
\draw[usual] (0,1) to [ out =180 , in = 180](0,0);
\draw[usual] (0,0) to [ out =0 , in = 0](0,1);
\draw[match] (0,0) to (0,1);
\node at (0,1)[circle,fill,inner sep=1.55pt]{};
\node at (0,0)[circle,fill,inner sep=1.55pt]{};
\node[below=5pt] at (current bounding box.south) {$E_1$};
\end{tikzpicture} \hspace{1.5cm} \begin{tikzpicture}[anchorbase,yscale=-1,scale=1.25]
\draw[usual] (0,1) to [ out =180 , in = 180](0,0);
\draw[usual] (1,1) to [ out =0 , in = 0](1,0);
\draw[match] (0,0) to (1,0);
\draw[match] (0,1) to (1,1);
\draw[usual] (1,1) to (1,0);
\draw[usual] (0,0) to (0,1);
\node at (0,1)[circle,fill,inner sep=1.55pt]{};
\node at (0,0)[circle,fill,inner sep=1.55pt]{};
\node at (1,0)[circle,fill,inner sep=1.55pt]{};
\node at (1,1)[circle,fill,inner sep=1.55pt]{};
\node[below=5pt] at (current bounding box.south) {$E_2$};
\end{tikzpicture} \hspace{1.5cm} \begin{tikzpicture}[anchorbase,yscale=-1,scale=1.25]
\draw[usual] (0,0) to (1,0);
\draw[usual] (0,1) to (1,1);
\draw[match] (1,1) to (1,0);
\draw[match] (0,0) to (0,1);
\draw[usual] (0,0) to (1,1);
\draw[usual] (0,1) to (1,0);
\node at (0,1)[circle,fill,inner sep=1.55pt]{};
\node at (0,0)[circle,fill,inner sep=1.55pt]{};
\node at (1,0)[circle,fill,inner sep=1.55pt]{};
\node at (1,1)[circle,fill,inner sep=1.55pt]{};
\node[below=5pt] at (current bounding box.south) {$E_3$};
\end{tikzpicture} \hspace{1.5cm}  \begin{tikzpicture}[anchorbase,yscale=-1,scale=1.25]
\draw[usual] (0,0) to (1,0);
\draw[usual] (0,1) to (1,1);
\draw[match] (1,1) to (1,0);
\draw[match] (0,0) to (0,1);
\draw[usual] (0,0) to [ out =-45 , in = 180] (1,-0.5);
\draw[usual] (1,-0.5) to [ out =0 , in = 0](1,1);
\draw[usual] (0,1) to [ out =180 , in = 180] (0,-0.5);
\draw[usual] (0,-0.5) to [ out =0 , in = -135](1,0);
\node at (0,1)[circle,fill,inner sep=1.55pt]{};
\node at (0,0)[circle,fill,inner sep=1.55pt]{};
\node at (1,0)[circle,fill,inner sep=1.55pt]{};
\node at (1,1)[circle,fill,inner sep=1.55pt]{};
\node[below=5pt] at (current bounding box.south) {};
\end{tikzpicture}
,
\end{gather*}
where we note that the final two perfect matching graphs are equivalent; we leave it as an exercise to verify this. 
\end{Example} 

\begin{Definition}
Two perfect matching graphs are \emph{equivalent} if they are equivalent as ribbon graphs and the induced isomorphism of their underlying graphs preserves the perfect matching. 
\end{Definition}

\begin{Notation}
From now on, we only consider ribbon graphs with connected underlying graphs $G = (V , E, r)$ with $E$ nonempty; in particular, all vertices have degree at least one.
\end{Notation}

Given any ribbon graph, we can construct a canonical perfect matching graph, whose underlying graph is trivalent, by altering the vertices as follows:

\begin{Definition}\label{D:blowup}
The \emph{blowup} of a ribbon graph is obtained by replacing every vertex with vertices arranged in a circle:
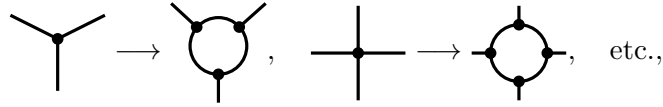
\begin{gather*}
\begin{tikzpicture}[anchorbase,yscale=-1,scale=1.25]
\draw[usual] (0,0) to (0.5,0.25);
\draw[usual] (0.5,0.25) to (1,0);
\draw[usual] (0.5,0.25) to (0.5,0.8);
\node at (0.5,0.25)[circle,fill,inner sep=1.55pt]{};
\end{tikzpicture} 
\longrightarrow
\begin{tikzpicture}[anchorbase,yscale=-1,scale=1.25]
\draw[usual] (0,-0.2) to (0.27,0.05);
\draw[usual] (0.72,0.05) to (1,-0.2);
\draw[usual] (0.5,0.55) to (0.5,0.85);
\node at (0.5,0.55)[circle,fill,inner sep=1.55pt]{};
\node at (0.72,0.05)[circle,fill,inner sep=1.55pt]{};
\node at (0.27,0.05)[circle,fill,inner sep=1.55pt]{};
\draw[usual] (0.5,0.25) circle (0.3cm);
\end{tikzpicture}
,\quad
\begin{tikzpicture}[anchorbase,scale=1.25]
\draw[usual] (0,0.25) to (0.5,0.25);
\draw[usual] (1,0.25) to (0.5,0.25);
\draw[usual] (0.5,-0.25) to (0.5,0.25);
\draw[usual] (0.5,0.75) to (0.5,0.25);
\node at (0.5,0.25)[circle,fill,inner sep=1.55pt]{};
\end{tikzpicture}
\longrightarrow
\begin{tikzpicture}[anchorbase,scale=1.25]
\draw[usual] (0,0.25) to (0.2,0.25);
\draw[usual] (1,0.25) to (0.8,0.25);
\draw[usual] (0.5,-0.25) to (0.5,-0.05);
\draw[usual] (0.5,0.75) to (0.5,0.55);
\draw[usual] (0.5,0.25) circle (0.3cm);
\node at (0.2,0.25)[circle,fill,inner sep=1.55pt]{};
\node at (0.8,0.25)[circle,fill,inner sep=1.55pt]{};
\node at (0.5,-0.05)[circle,fill,inner sep=1.55pt]{};
\node at (0.5,0.55)[circle,fill,inner sep=1.55pt]{};
\end{tikzpicture}
,\quad
\text{etc.},
\end{gather*}
where one uses as many vertices as the degree of the start vertex.
\end{Definition}


\begin{Lemma}\label{L:blowupiso}
The blowup of a ribbon graph is trivalent and admits a perfect matching. Moreover, if two ribbon graphs are equivalent, then so are their blowups.
\end{Lemma}

\begin{proof}
The first statement is immediate by construction, and there is a perfect matching that can be associated with the blowup by using the edges of the original graph, e.g.
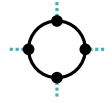
\begin{gather*}
\begin{tikzpicture}[anchorbase,scale=1.25]
\draw[match] (0,0.25) to (0.2,0.25);
\draw[match] (1,0.25) to (0.8,0.25);
\draw[match] (0.5,-0.25) to (0.5,-0.05);
\draw[match] (0.5,0.75) to (0.5,0.55);
\draw[usual] (0.5,0.25) circle (0.3cm);
\node at (0.2,0.25)[circle,fill,inner sep=1.55pt]{};
\node at (0.8,0.25)[circle,fill,inner sep=1.55pt]{};
\node at (0.5,-0.05)[circle,fill,inner sep=1.55pt]{};
\node at (0.5,0.55)[circle,fill,inner sep=1.55pt]{};
\end{tikzpicture}
.
\end{gather*}
For the second statement, let $\Gamma_1$, $\Gamma_2$ be two equivalent ribbon graphs, with $\Gamma_1^\flat$, $\Gamma_2^\flat$ their respective blowups. Away from the vertices, $\Gamma_2$ is obtained from $\Gamma_1$ using a finite sequence of planar isotopy \autoref{eq:plane} and Reidemeister moves \autoref{eq:reide}, and it is easy to see that the same is true of $\Gamma_2^\flat$ and $\Gamma_1^\flat$. At the vertices, we simply note that \eg \,
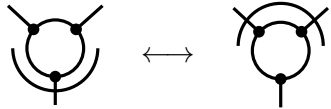
\begin{gather*}
\begin{tikzpicture}[anchorbase,yscale=-1,scale=1.25]
\draw[usual] (0,-0.2) to (0.27,0.05);
\draw[usual] (0.72,0.05) to (1,-0.2);
\draw[usual] (0.5,0.55) to (0.5,0.85);
\node at (0.5,0.55)[circle,fill,inner sep=1.55pt]{};
\node at (0.72,0.05)[circle,fill,inner sep=1.55pt]{};
\node at (0.27,0.05)[circle,fill,inner sep=1.55pt]{};
\draw[usual] (0.5,0.25) circle (0.3cm);
\draw[usual] (0.95,0.25) arc (0:180:0.45);
\end{tikzpicture} \quad \longleftrightarrow \quad \begin{tikzpicture}[anchorbase,yscale=-1,scale=1.25]
\draw[usual] (0,-0.2) to (0.27,0.05);
\draw[usual] (0.72,0.05) to (1,-0.2);
\draw[usual] (0.5,0.55) to (0.5,0.85);
\node at (0.5,0.55)[circle,fill,inner sep=1.55pt]{};
\node at (0.72,0.05)[circle,fill,inner sep=1.55pt]{};
\node at (0.27,0.05)[circle,fill,inner sep=1.55pt]{};
\draw[usual] (0.5,0.25) circle (0.3cm);
\draw[usual] (0.05,0.2) arc (180:360:0.45);
\end{tikzpicture}
\end{gather*}
where we have used the final ribbon move \autoref{ribbonmove4} three times to move the unattached edge across the three vertices on the blowup, and similarly for starting vertices of other degrees.
\end{proof}

Thus, by \autoref{L:blowupiso}, we can and will restrict to ribbon graphs whose underlying graphs are trivalent with at least one perfect matching.


\subsection{Polynomial invariants}\label{S:polyinv} 

We now consider a generic polynomial invariant of ribbon graphs using specified relations.

\begin{Definition}\label{D:mobpoly}
Let $G$ be a trivalent graph with perfect matching $M$, and let $\Gamma_M$ be a perfect matching graph with underlying graph $G$. The element $\langle \Gamma_M \rangle \in \mathbb{Z}[A, B, C]$, called the \emph{bracket} (or \emph{Penrose bracket polynomial}) of $\Gamma_M$, is characterized by:
\begin{align}
\Bigg\langle\begin{tikzpicture}[anchorbase,yscale=-1,scale=0.75]
\draw[usual] (0,0) to (0.5,0.25);
\draw[usual] (0.5,0.25) to (1,0);
\draw[match] (0.5,0.25) to (0.5,1);
\draw[usual] (0,1.25) to (0.5,1);
\draw[usual] (1,1.25) to (0.5, 1);
\node at (0.5,0.25)[circle,fill,inner sep=1.55pt]{};
\node at (0.5,1)[circle,fill,inner sep=1.55pt]{};
\end{tikzpicture} \Bigg\rangle
\quad &= \quad
A \, \Bigg\langle\begin{tikzpicture}[anchorbase,yscale=-1,scale=0.75]
\draw[usual] (0,0) to (0,1);
\draw[usual] (0.5,0) to (0.5,1);
\end{tikzpicture} \Bigg\rangle \quad + \quad B \,\Bigg\langle\begin{tikzpicture}[anchorbase,yscale=-1,scale=0.75]
\draw[usual] (0,0) to (0.5,1);
\draw[usual] (0.5,0) to (0,1);
\end{tikzpicture} \Bigg\rangle, \label{skein1} \\
\Bigg\langle\begin{tikzpicture}[anchorbase,yscale=-1,scale=0.75]
\filldraw[color=black, fill=white, thick](-1,0) circle (0.5);
\end{tikzpicture}\Bigg\rangle \quad &= \quad C, \label{skein2} \\
\langle \Gamma_1 \sqcup \Gamma_2 \rangle \quad &= \quad \langle \Gamma_1 \rangle \cdot \langle \Gamma_2 \rangle. \label{skein3}
\end{align}
Resolving a perfect matching edge by \begin{tikzpicture}[anchorbase,scale=0.25]
\draw[usual] (0,0) to (0,1);
\draw[usual] (0.5,0) to (0.5,1);
\end{tikzpicture} is called a \emph{$0$-smoothing} and by \begin{tikzpicture}[anchorbase,yscale=-1,scale=0.25]
\draw[usual] (0,0) to (0.5,1);
\draw[usual] (0.5,0) to (0,1);
\end{tikzpicture} a \emph{$1$-smoothing}.
\end{Definition}

\autoref{D:mobpoly} is inspired by \cite{Pen-negative}, which is the case with $A=1,B=-1$ in \autoref{skein1} and $C=3$ in \autoref{skein2}. By \cite{Kho-jones}, it is now mandatory to ask 
whether one can categorify \autoref{skein1}--\autoref{skein3}, so we do this below.

\begin{Remark}\label{R:circle}
In order for the l.h.s. of \autoref{skein2} to make sense as a bracket, we should also include the circle and disjoint unions of circles in the plane, up to a finite sequence of Reidemeister moves and planar isotopy, as additional equivalence classes in addition to the ribbon graphs. 
\end{Remark}

\begin{Remark}\label{brackforribb}
The bracket can of course be defined for any ribbon graph by using the canonical perfect matching induced from the blowup and \autoref{L:blowupiso}.
\end{Remark}

\begin{Proposition}\label{P:bracketinv}
The bracket is a ribbon graph invariant.
\end{Proposition}

\begin{proof}
Recall the relations in \autoref{eq:plane}, \autoref{eq:reide} and \autoref{ribbonmove4}; the ribbon moves.
We are required to show that if $\Gamma_M$ and $\Gamma_{M'}'$ are equivalent perfect matching graphs, i.e. there exists a finite sequence of ribbon moves that takes $\Gamma_M$ to $\Gamma_{M'}'$, then $\langle \Gamma_M \rangle = \langle \Gamma_{M'}' \rangle$. Therefore, if we can show that the ribbon moves do not affect the evaluation of the skein relations \autoref{skein1}--\autoref{skein3}, then we are done, since we can of course use them freely away from the evaluation sites.

We first note that \autoref{skein3} is true regardless of what form $\Gamma_1$ and $\Gamma_2$ take. Also, the evaluation of the circle in \autoref{skein2} is not affected by Reidemeister moves or planar isotopy by construction as discussed in \autoref{R:circle}. Finally, it is clear that planar isotopies, and Reidemeister moves away from the perfect matching edge, do not affect the evaluation in \autoref{skein1}. The only part left to prove is then 
\begin{align*}
\Bigg\langle\begin{tikzpicture}[anchorbase,yscale=-1,scale=0.75]
\draw[usual] (0,0) to (0.5,0.25);
\draw[usual] (0.5,0.25) to (1,0);
\draw[usual, dotted] (0.5,0.25) to (0.5,1);
\draw[usual] (0,1.25) to (0.5,1);
\draw[usual] (1,1.25) to (0.5, 1);
\draw[usual, red] (0.95,0.85) arc (0:180:0.45);
\node at (0.5,0.25)[circle,fill,inner sep=1.55pt]{};
\node at (0.5,1)[circle,fill,inner sep=1.55pt]{};
\end{tikzpicture} \Bigg\rangle 
\quad &= \quad
\Bigg\langle\begin{tikzpicture}[anchorbase,yscale=-1,scale=0.75]
\draw[usual] (0,0) to (0.5,0.25);
\draw[usual] (0.5,0.25) to (1,0);
\draw[usual, dotted] (0.5,0.25) to (0.5,1);
\draw[usual] (0,1.25) to (0.5,1);
\draw[usual] (1,1.25) to (0.5, 1);
\draw[usual, red] (0.05,0.3) arc (180:360:0.45);
\node at (0.5,0.25)[circle,fill,inner sep=1.55pt]{};
\node at (0.5,1)[circle,fill,inner sep=1.55pt]{};
\end{tikzpicture} \Bigg\rangle
\end{align*}
where the marked red edge can either be a perfect matching or generic edge. On both sides, the $0$-smoothing term in \autoref{skein1} is clearly unchanged, while for the $1$-smoothing term, equality follows directly from the third Reidemeister move. 
\end{proof}

Up to this point, we have not shown that the bracket is well-defined i.e. the order in which we resolve the perfect matching edges does not affect the result. In order to do so, we make use of the concept of a \emph{hypercube of states}. The setup is as follows: let $G = (V, E, r)$ be a trivalent graph, $M = \{e_1, e_2,\ldots, e_\ell\} \subset E$ be a perfect matching of $G$, and let $\Gamma_M$ be a perfect matching graph with underlying graph $G$. By construction, note that $\ell = |V|/2$. If we resolve each perfect matching edge $e_i$ by either a $0$- or $1$-smoothing, we obtain a set of circles in the plane called a \emph{state} of $\Gamma_M$, and there are $2^\ell$ of them. Given $\alpha = (\alpha_1, \alpha_2, \ldots, \alpha_\ell) \in \{0,1\}^\ell$, we let $\Gamma_\alpha$ denote the state of $\Gamma_M$
where the perfect matching edge $e_i$ has been resolved by an $\alpha_i$-smoothing.

We can arrange the states $\Gamma_\alpha$ as vertices in a hypercube as shown 
in the following example.

\begin{Example}\label{E:hypercube}
Given the perfect matching graph $E_2$ from \autoref{E:permatexm}, which has two perfect matching edges, we have $2^2=4$ states ${\Gamma_{(0,0)}}$, ${\Gamma_{(1,0)}}$, ${\Gamma_{(0,1)}}$, and ${\Gamma_{(1,1)}}$; these states have 2, 1, 1, and 2 circles respectively. It is useful to visualize the hypercube of states as follows:
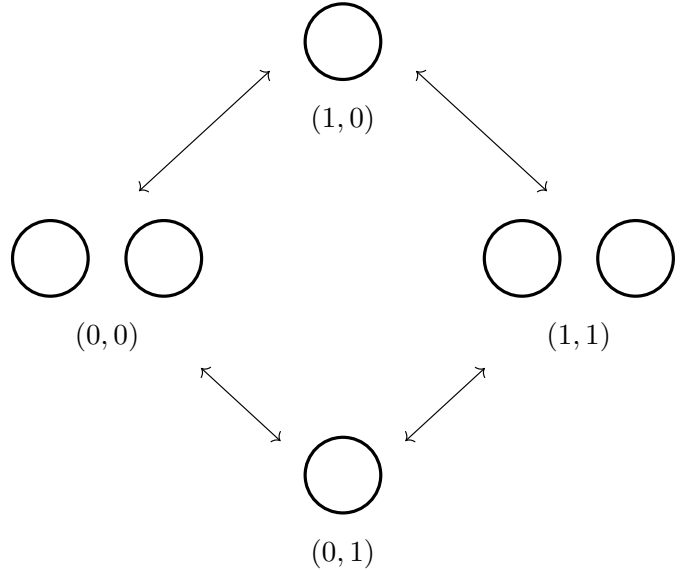
\begin{equation*}
\begin{tikzcd}
& \begin{tikzpicture}
\draw[usual] (0,0) circle (0.5cm);
\node[below=5pt] at (current bounding box.south) {$(1,0)$};
\end{tikzpicture} 
\arrow[dr, leftrightarrow, shift left=1.8ex] & \\
\begin{tikzpicture}
\draw[usual] (0,0) circle (0.5cm);
\draw[usual] (1.5,0) circle (0.5cm);
\node[below=5pt] at (current bounding box.south) {$(0,0)$};
\end{tikzpicture}
\arrow[ur, leftrightarrow, shift left=1.8ex] \arrow[dr, leftrightarrow, shift right=1.8ex]
& & \begin{tikzpicture}
\draw[usual] (0,0) circle (0.5cm);
\draw[usual] (1.5,0) circle (0.5cm);
\node[below=5pt] at (current bounding box.south) {$(1,1)$};
\end{tikzpicture} \\
& \begin{tikzpicture}
\draw[usual] (0,0) circle (0.5cm);
\node[below=5pt] at (current bounding box.south) {$(0,1)$};
\end{tikzpicture} 
\arrow[ur, leftrightarrow,, shift right=1.8ex]&
\end{tikzcd}
.
\end{equation*}
This should come to no surprise for readers familiar with \cite{Kau-state,Kho-jones}.
\end{Example}

In general, we can arrange states with constant $|\alpha|=\alpha_1 + \alpha_2 + \ldots+ \alpha_\ell$ as columns, with $|\alpha|$ increasing from left to right; note that the order in which states appear within a fixed column is arbitrary, and that $0 \leq |\alpha| \leq \ell$. Of course, the use of columns in this manner is simply a useful visualization tool, and in practice, the symmetry of the hypercube is undisturbed. The bidirectional arrows, which correspond to edges of the hypercube, appear between states $\Gamma_\alpha$ and $\Gamma_{\alpha'}$ if and only if $\alpha$ and $\alpha'$ differ on a single component. 

\begin{Lemma}
The hypercube of states is a ribbon graph invariant.
\end{Lemma}

\begin{proof}
From the proof that the skein relation \autoref{skein1} is invariant under ribbon moves and planar isotopy in \autoref{P:bracketinv}, we know that $0$- and $1$-smoothings are invariant separately, and so the result follows.
\end{proof}

Finally, to each vertex corresponding to some state $\Gamma_\alpha$ in the hypercube of states, we associate the quantity $A^{\ell-|\alpha|}B^{|\alpha|}C^{k_\alpha}$, where $k_\alpha$ is the number of circles in $\Gamma_\alpha$, and then sum over all the vertices.

\begin{Lemma}
The result $\sum_{\alpha}A^{\ell-|\alpha|}B^{|\alpha|}C^{k_\alpha}$ is $\langle \Gamma_M \rangle$. Since the order in which the vertices are summed has no effect, $\langle \Gamma_M \rangle$ is therefore well-defined.
\end{Lemma}

\begin{proof}
This is straightforward to see.
\end{proof}

\subsection{Graph homology}\label{graphhomo} 

We can obtain another ribbon graph invariant using homology, from which the bracket emerges for specific choices of $A, B$, and $C$, and which is strictly more powerful; the story is analogous to the emergence of the Jones polynomial from Khovanov homology in knot theory. We do so by enforcing certain conditions on our MTQFTs in order to build an invariant chain complex whose graded Euler characteristic is the bracket.

First suppose that we are given an MTQFT $\mathcal{F}: \mcob{} \rightarrow R\textbf{Mod}$ with $\mathcal{F}(1) = V$. Let $\mathcal{G}$ be a finitely-generated abelian group with generators $g_1, g_2, \ldots, g_m$. Suppose that the finite-dimensional free $R$-module $V$ has a $\mathcal{G}$-grading i.e. $V = \oplus_{g \in \mathcal{G}} V_g$ with $\mathcal{G}$ considered as a set and only finitely many $V_g$ are nonzero. Given a tuple $s= (s_1, s_2, \ldots, s_m)$ of integers and an element $g = g_1^{p_1}g_2^{p_2}\cdots g_m^{p_m} \in \mathcal{G}$ for some integers $p_1, p_2,  \ldots, p_m$, we can always shift the grading by setting $(V\{s\})_g = V_{g'}$ where $g' = g_1^{p_1-s_1}g_2^{p_2-s_2}\cdots g_m^{p_m-s_m}$. The \emph{graded dimension} of $V$ is then a polynomial in the variables $q_1, q_2, \ldots, q_m$ given by 
\begin{gather*}
q\text{dim}(V) = \sum_{g\in \mathcal{G}} q_1^{p_1}q_2^{p_2} \cdots q_m^{p_m} \text{dim}(V_g).
\end{gather*}

\begin{Example}\label{E:Grading}
For $\mathcal{G}=\Z/n\Z\times\Z/2\Z$ this really means we have two grading variables, say $q$ and $b$, that satisfy $q^n=1=b^2$. In this case the grading would be a tuple $(x,y)$ where we read the first entry modulo $n$ and the second modulo $2$. We will use similar conventions for other gradings.
\end{Example}

Further suppose that the following relation holds on $V$:
\begin{gather}\label{specrel}
\mathcal{F}(m^2)=\mathcal{F}(h),
\end{gather}
which is a more restricted relation than \autoref{mobrel} on $\mcob{}$; in particular, $\mathcal{F}$ is not faithful. 

\begin{Remark}
The relation \autoref{specrel} may look a bit mysterious at first. Topologically, the relation is
$m^3=hm$, cf. \autoref{mobrel}, and replacing this by $m^2=h$ would amount to the identity
\begin{gather*}
\begin{tikzpicture}[baseline ={([yshift=-.5ex]current bounding box.center)}, draw, thick,scale=0.25]
\draw (0,0) circle (2cm);
\draw (-2,0) arc (180:360:2 and 0.6);
\draw[dashed, thick] (2,0) arc (0:180:2 and 0.6);
\node at (-0.66,0.8){$\lightning$};
\node at (0.66,0.8){$\lightning$};
\end{tikzpicture}
=
\begin{tikzpicture}[scale=0.25, draw, thick, baseline ={([yshift=-.5ex]current bounding box.center)}]
\draw (0,0) ellipse (3cm and 2cm);
\draw (-1.5, 0.2) arc (180:360:1.5 and 0.6);
\draw (1,-0.2) arc (0:180:1 and 0.6);
\end{tikzpicture}
\quad\text{(not a topological identity).}
\end{gather*}
In other words, it would identify the Klein bottle with the torus, which is, of course, not true.

The reason that \autoref{specrel} can nevertheless hold in our examples is that we are no longer working faithfully with the topological monoid. Roughly speaking, 
and with Green relations/cells \cite{Gr-structure-semigroups} in mind (using the convention of \cite{KhSiTu-monoidal-cryptography}),
we pass to a higher cell of the monoid
$\langle m,h \mid m^3=hm,\; mh=hm\rangle$, where $m$ becomes locally invertible. In that setting one may cancel the remaining $m$, giving the effective relation $m^2=h$. 
Alternatively, the respective monoid algebra is sandwich cellular in the sense of \cite{Tu-sandwich-cellular}, so one can use the standard cell theory to ``cut out'' the relevant parts.
We call this \emph{pseudo-invertibility}.

Concrete solutions are given in \autoref{S:ncolorchoice}.
\end{Remark}

We are now ready to begin constructing our chain complex. Fix a perfect matching graph $\Gamma_M$ with trivalent underlying graph $G = (V, E, r)$ and perfect matching $M = \{e_1, e_2, \ldots, e_\ell\} \subset E$. Form the hypercube of states as shown e.g. in \autoref{E:hypercube}. We now give the edges a fixed direction according to the following rule: there is a directed edge $\Gamma_\alpha \rightarrow  \Gamma_{\alpha'}$ if and only if $|\alpha'| = |\alpha| +1$. To each state $\Gamma_\alpha$, we then associate the $\mathcal{G}$-graded $R$-module $V_\alpha = \otimes^{k_\alpha}V \{s|\alpha|\}$ where $k_\alpha$ is the number of circles in $\Gamma_\alpha$ and $s$ is some constant shift. Define a complex $C^{*, *}(\Gamma_M)$ by 
\begin{gather}\label{compdef}
C^{i, *}(\Gamma_M) = \oplus_{\alpha \in \{0,1\}^\ell, \, |\alpha|=i} V_\alpha,
\end{gather}
and it is trivial outside of $i =1, \ldots, \ell$. $C^{i, *}(\Gamma_M)$ is also $\mathcal{G}$-graded with $C^{i, g}(\Gamma_M)=\big(C^{i, *}(\Gamma_M)\big)_g$.

\begin{Remark}
Given a shift $s = (s_1, s_2, \ldots, s_m)$, we define $s|\alpha| = (s_1|\alpha|, s_2|\alpha|, \ldots, s_m|\alpha|)$ i.e. component-by-component multiplication.
\end{Remark}

\begin{Remark}\label{R:correscomp}
We will also think of the complex $C^{*, *}(\Gamma_M)$ as arranged in a hypercube, with vertices and directed edges in bijective correspondence with those in the hypercube of states, i.e. $\Gamma_\alpha \leftrightarrow V_\alpha$ for vertices and $\Gamma_\alpha \rightarrow \Gamma_{\alpha'} \leftrightarrow V_\alpha \rightarrow V_{\alpha'}$ for directed edges.
\end{Remark}

\begin{Example}\label{E:comp}
Given the hypercube of states from \autoref{E:hypercube}, the corresponding complex can also be visualized as follows:
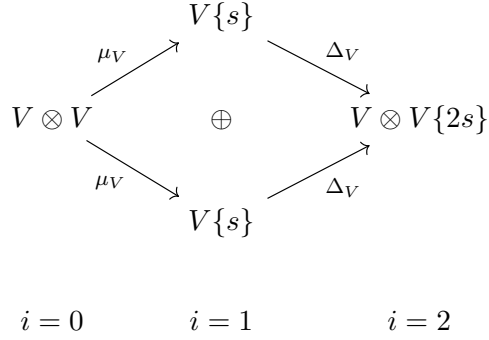
\begin{equation*}
\begin{tikzcd}
& V\{s\}
\arrow[dr, "\Delta_V"] & \\
V \otimes V
\arrow[ur, "\mu_V"] \arrow[dr, "\mu_V"']
& \oplus & V \otimes V \{2s\} \\
& V\{s\}
\arrow[ur, "\Delta_V"']& \\
i = 0 & i =1 & i=2
\end{tikzcd}
\end{equation*}
where the interpretation of directed edges as $R$-linear maps is discussed below.
\end{Example}

\begin{Remark}\label{R:symm}
In forming our complex, the symmetry of the hypercube of states is destroyed; there is a preferred direction that runs from the vertex with $|\alpha| = 0$ to the vertex with $|\alpha| =\ell$. Rotating about the preferred axis gives a residual symmetry for states with fixed $|\alpha|$, which is preserved in the complex up to reordering of the direct sum.
\end{Remark}

We now define the differential
\[
\partial^i \colon C^{i,*}(\Gamma_M) \longrightarrow C^{i+1,*}(\Gamma_M).
\]
Fix a directed edge in the cube of states, say from $\Gamma_\alpha$ to
$\Gamma_{\alpha'}$, where $\alpha'$ is obtained from $\alpha$ by changing one
$0$-smoothing into a $1$-smoothing. Away from this smoothing, the two states
are identical. At the smoothing itself, exactly one of the following three
things happens:
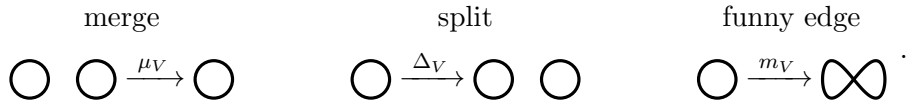
\begin{gather*}
\begin{array}{c@{\qquad\qquad}c@{\qquad\qquad}c}
\text{merge} & \text{split} & \text{funny edge} \\[0.5em]
\begin{tikzpicture}[anchorbase,scale=0.55]
\draw[usual] (-0.8,0) circle (0.45);
\draw[usual] (0.8,0) circle (0.45);
\end{tikzpicture}
\xrightarrow{\;\mu_V\;}
\begin{tikzpicture}[anchorbase,scale=0.55]
\draw[usual] (0,0) circle (0.45);
\end{tikzpicture}
&
\begin{tikzpicture}[anchorbase,scale=0.55]
\draw[usual] (0,0) circle (0.45);
\end{tikzpicture}
\xrightarrow{\;\Delta_V\;}
\begin{tikzpicture}[anchorbase,scale=0.55]
\draw[usual] (-0.8,0) circle (0.45);
\draw[usual] (0.8,0) circle (0.45);
\end{tikzpicture}
&
\begin{tikzpicture}[anchorbase,scale=0.55]
\draw[usual] (0,0) circle (0.45);
\end{tikzpicture}
\xrightarrow{\;m_V\;}
\begin{tikzpicture}[anchorbase,scale=0.55]
\draw[usual,smooth,samples=80,domain=0:360,variable=\t]
plot ({0.75*sin(\t)},{0.45*sin(2*\t)});
\end{tikzpicture}
\end{array}
.
\end{gather*}
Here the pictures only show the components affected by the chosen edge of the
cube; all other components are unchanged and contribute identity tensor factors.
The first two cases are the familiar ones from ordinary link homology: two
circles merge, or one circle splits into two. The third case is the new feature
in the present setting: one circle remains one circle, but acquires a
self-crossing; this edge is represented algebraically by the M\"obius map.

\begin{Remark}
The reader familiar with virtual link homology might recognize the funny edge, which always causes trouble, see \cite{Man-vKh} for an early reference for this.
\end{Remark}

Using the correspondence from \autoref{R:correscomp}, we therefore assign to the
directed edge $\Gamma_\alpha \to \Gamma_{\alpha'}$ the $R$-linear map
\[
\partial_{\alpha\alpha'} \colon V_\alpha \longrightarrow V_{\alpha'}
\]
given by
\[
\partial_{\alpha\alpha'}=
\begin{cases}
\mu_V \otimes \id, & \text{if two circles merge},\\
\Delta_V \otimes \id, & \text{if one circle splits into two},\\
m_V \otimes \id, & \text{if one circle crosses itself},
\end{cases}
\]
where $\mu_V=\mathcal{F}(\mu)$, $\Delta_V=\mathcal{F}(\Delta)$, and
$m_V=\mathcal{F}(m)$, and where the identity factors act on the unchanged
circle components.

\begin{Remark}\label{R:tensororder}
There may be some concern over the ordering of the tensor product of $R$-linear maps along directed edges of the complex. We note that any such choice ultimately corresponds to an ordering of the circles in the hypercube of states; the result is identical up to an overall application of a tensor product swap map.
\end{Remark}

We have the following crucial proposition:

\begin{Proposition}\label{P:commute}
For any perfect matching graph $\Gamma_M$, the complex $C^{*, *}(\Gamma_M)$ together with the maps $\partial_{\alpha \alpha'}$ form a commutative diagram.
\end{Proposition}

\begin{proof}
We first note that if the square faces of our complex commute, then clearly the whole complex commutes by construction. For each face, consider opposing vertices such that the directed edges move away from one vertex in two directions and towards the other vertex from two directions e.g. $V \otimes V$ and $V \otimes V\{2s\}$ form opposing vertices in \autoref{E:comp}. 

Given a face, we begin by noting that if one path between opposing vertices contains a composition of maps with the $\mu_V$, $\Delta_V$, or $m_V$ operating on different copies of $V$ in the tensor product, it is relatively straightforward to see that the other path must contain the same composition of maps with the order swapped; using spine diagrams, this is essentially
\begin{gather*}
\begin{tikzpicture}[anchorbase,scale=0.22,tinynodes]
\draw[spinach!35,fill=spinach!35] (-5.5,2) rectangle (-2,0.5);
\draw[blue!35,fill=blue!35] (-0.5,3) rectangle (3,4.5);
\draw[usual] (-5,2) to (-5,3.75) to (-5,5);
\draw[usual] (-2.5,2) to (-2.5,5);
\draw[usual] (0,0) to (0,1.25)to (0,3);
\draw[usual] (2.5,0) to (2.5,3);
\draw[usual] (-5,0) to (-5,0.25) to (-5,0.5);
\draw[usual] (-2.5,0) to (-2.5,0.5);
\draw[usual] (2.5,4.5) to (2.5,5);
\draw[usual] (0,4.5) to (0,4.75) to (0,5);
\end{tikzpicture}
=
\begin{tikzpicture}[anchorbase,scale=0.22,tinynodes]
\draw[blue!35,fill=blue!35] (5.5,2) rectangle (2,0.5);
\draw[spinach!35,fill=spinach!35] (0.5,3) rectangle (-3,4.5);
\draw[usual] (2.5,2) to (2.5,3.75) to (2.5,5);
\draw[usual] (5,2) to (5,5);
\draw[usual] (0,0) to (0,3);
\draw[usual] (-2.5,0) to (-2.5,1.25) to (-2.5,3);
\draw[usual] (2.5,0) to (2.5,0.25) to (2.5,0.5);
\draw[usual] (5,0) to (5,0.5);
\draw[usual] (-2.5,4.5) to (-2.5,4.75) to (-2.5,5);
\draw[usual] (0,4.5) to (0,5);
\end{tikzpicture}
.
\end{gather*}
By the functorial properties of the tensor product, the face must therefore commute.

It remains to consider the case where both paths between two opposing vertices
in a square face involve maps acting on the same circle components. We suppress
all unaffected components, which only contribute identity tensor factors, and everything that is not relevant for the discussion. By
construction, the tensor powers at the two opposing vertices can differ by
$0$, $1$, or $2$.

If the tensor powers differ by $2$, then both paths are built only from merges,
or dually only from splits. A typical merge--merge face is
\begin{gather*}
\begin{tikzcd}[column sep=large, row sep=large,ampersand replacement=\&]
\& V\otimes V
\arrow[dr, "\mu_V"] \& \\
V\otimes V\otimes V
\arrow[ur, "\mu_V\otimes \mathrm{id}"]
\arrow[dr, "\mathrm{id}\otimes \mu_V"']
\& \& V \\
\& V\otimes V
\arrow[ur, "\mu_V"'] \& 
\end{tikzcd}
\leftrightsquigarrow
\begin{tikzpicture}[anchorbase,scale=1]
\draw[usual] (0,0) to (1,1);
\draw[usual] (1,0) to (0.5,0.5);
\draw[usual] (2,0) to (1,1);
\draw[usual] (1,1) to (1,1.5);
\end{tikzpicture}
=
\begin{tikzpicture}[anchorbase,scale=1]
\draw[usual] (0,0) to (1,1);
\draw[usual] (1,0) to (1.5,0.5);
\draw[usual] (2,0) to (1,1);
\draw[usual] (1,1) to (1,1.5);
\end{tikzpicture}
,
\end{gather*}
which commutes by associativity. The dual split--split face
commutes by coassociativity.

If the tensor powers differ by $1$, then one of the two local changes is a
merge or split, while the other is a M\"obius change. For instance, a typical
merge--M\"obius face is
\begin{gather*}
\begin{tikzcd}[column sep=large, row sep=large,ampersand replacement=\&]
\& V
\arrow[dr, "m_V"] \& \\
V\otimes V
\arrow[ur, "\mu_V"]
\arrow[dr, "m_V\otimes \mathrm{id}"']
\& \& V \\
\& V\otimes V
\arrow[ur, "\mu_V"'] \& 
\end{tikzcd}
\leftrightsquigarrow
\begin{tikzpicture}[anchorbase]
\draw[usual,mob=0.25] (0,0) to (0.5,0.5) to (1,0);
\draw[usual] (0.5,0.5) to (0.5,1);
\end{tikzpicture}
=
\begin{tikzpicture}[anchorbase]
\draw[usual] (0,0) to (0.5,0.5) to (1,0);
\draw[usual,mob=0.5] (0.5,0.5) to (0.5,1);
\end{tikzpicture}
,
\end{gather*}
and this commutes by the indicated relation.
The other faces of this type commute by the analogous relations.

Finally, suppose that the tensor powers at the opposing vertices are equal. These are always the most exciting faces. A prototypical face without $m_V$ is
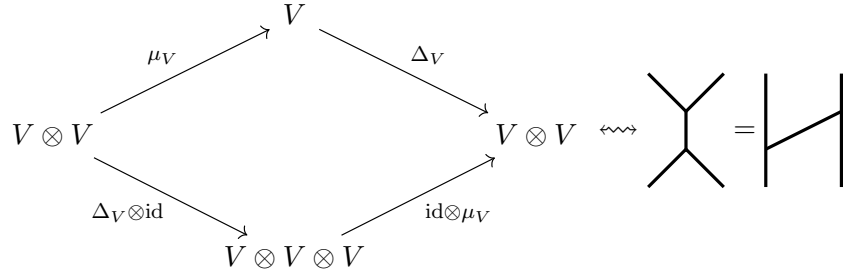
\begin{gather*}
\begin{tikzcd}[column sep=large, row sep=large,ampersand replacement=\&]
\& V
\arrow[dr, "\Delta_V"] \& \\
V\otimes V
\arrow[ur, "\mu_V"]
\arrow[dr, "\Delta_V\otimes \mathrm{id}"']
\&  \& V\otimes V \\
\& V\otimes V \otimes V
\arrow[ur, "\mathrm{id}\otimes\mu_V"'] \& 
\end{tikzcd}
\leftrightsquigarrow
\begin{tikzpicture}[anchorbase]
\draw[usual] (0,0) to (0.5,0.5) to (1,0);
\draw[usual] (0.5,0.5) to (0.5,1);
\draw[usual] (0,1.5) to (0.5,1) to (1,1.5);
\end{tikzpicture}
=
\begin{tikzpicture}[anchorbase]
\draw[usual] (0,0) to (0,1.5);
\draw[usual] (1,0) to (1,1.5);
\draw[usual] (0,0.5) to (1,1);
\end{tikzpicture}
,
\end{gather*}
which commutes by the displayed Frobenius H=I relation.
The only nontrivial new face is
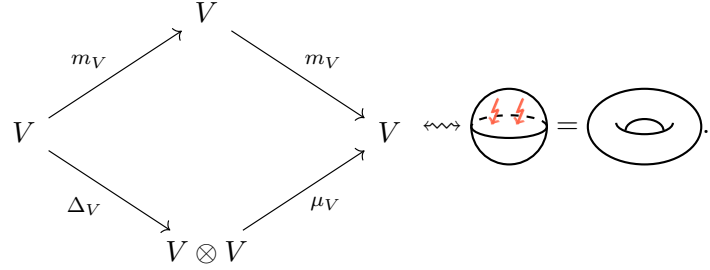
\begin{gather*}
\begin{tikzcd}[column sep=large, row sep=large,ampersand replacement=\&]
\& V
\arrow[dr, "m_V"] \& \\
V
\arrow[ur, "m_V"]
\arrow[dr, "\Delta_V"']
\&  \& V \\
\& V\otimes V
\arrow[ur, "\mu_V"'] \& 
\end{tikzcd}
\leftrightsquigarrow
\begin{tikzpicture}[baseline ={([yshift=-.5ex]current bounding box.center)}, draw, thick,scale=0.25]
\draw (0,0) circle (2cm);
\draw (-2,0) arc (180:360:2 and 0.6);
\draw[dashed, thick] (2,0) arc (0:180:2 and 0.6);
\node at (-0.66,0.8){$\lightning$};
\node at (0.66,0.8){$\lightning$};
\end{tikzpicture}
=
\begin{tikzpicture}[scale=0.25, draw, thick, baseline ={([yshift=-.5ex]current bounding box.center)}]
\draw (0,0) ellipse (3cm and 2cm);
\draw (-1.5, 0.2) arc (180:360:1.5 and 0.6);
\draw (1,-0.2) arc (0:180:1 and 0.6);
\end{tikzpicture}
.
\end{gather*}
The two paths around this face are $m_V^2$ and $\mu_V\circ\Delta_V$, so this
face commutes exactly by pseudo-invertibility \autoref{specrel}, as displayed.
\end{proof}

\begin{Example}
The following example shows that we actually do require the pseudo-invertibility condition \autoref{specrel} in \autoref{P:commute}. Consider the perfect matching graph:
\begin{gather*}
\begin{tikzpicture}[anchorbase,yscale=-1]
\draw[usual] (0,0) to (0.5,0.5);
\draw[usual] (0.5,0.5) to (1,0);
\draw[match] (0.5,0.5) to (0.5,1.25);
\draw[usual] (0,1.75) to (0.5,1.25);
\draw[usual] (1,1.75) to (0.5, 1.25);
\draw[match] (0 ,0) to (1,0);
\draw[match] (0,1.75) to (1,1.75);
\draw[usual] (0,0) to[out =180, in = 180] (0, 1.75);
\draw[usual] (1,0) to[out =0, in = 0] (1, 1.75);
\node at (0.5,0.5)[circle,fill,inner sep=1.55pt]{};
\node at (0.5,1.25)[circle,fill,inner sep=1.55pt]{};
\node at (0,0)[circle,fill,inner sep=1.55pt]{};
\node at (1,0)[circle,fill,inner sep=1.55pt]{};
\node at (0,1.75)[circle,fill,inner sep=1.55pt]{};
\node at (1,1.75)[circle,fill,inner sep=1.55pt]{};
\end{tikzpicture},
\end{gather*}
which indeed contains the following face in its corresponding complex:
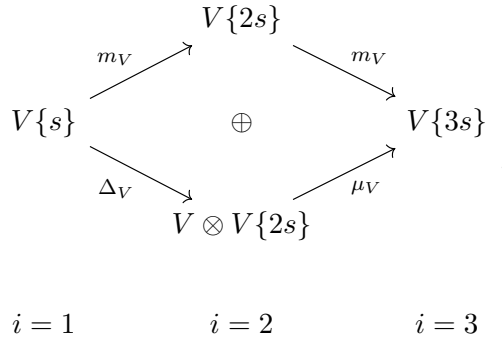
\begin{equation*}
\begin{tikzcd}
& V\{2s\}
\arrow[dr, "m_V"] & \\
V\{s\}
\arrow[ur, "m_V"] \arrow[dr, "\Delta_V"']
& \oplus & V \{3s\} \\
& V\otimes V\{2s\}
\arrow[ur, "\mu_V"']& \\
i = 1 & i =2 & i=3
\end{tikzcd}
,
\end{equation*}
as one can easily check.
\end{Example}

Let us now define the differential $\partial^i$. For $v \in V_\alpha \subset C^{i, *}(\Gamma_M)$, we define
\begin{gather}\label{diffdef}
\partial^i(v) = \sum_{\alpha'} (-1)^{f_{\alpha \alpha'}} \partial_{\alpha \alpha'}(v),
\end{gather}
where the sum is over all $\alpha'$ such that there is a directed edge $\Gamma_\alpha \rightarrow \Gamma_{\alpha'}$ and $f_{\alpha \alpha'}$ counts the number of 1s to the left of the component where $\alpha$ and $\alpha'$ differ. Clearly, $\partial^i(v) \in C^{i+1, *}(\Gamma_M)$.

\begin{Theorem}\label{T:chaincomp}
$(C^{i, *}(\Gamma_M), \partial^i)$ is a chain complex.
\end{Theorem}
\begin{proof}
We already know that $\partial^i: C^{i, *}(\Gamma_M) \rightarrow C^{i+1, *}(\Gamma_M)$ is a $R$-linear map between $R$-modules, and so we only need to show that $\partial^{i+1} \circ \partial^i = 0$. But we know from \autoref{P:commute} that the faces of the complex commute, and so the signs $(-1)^{f_{\alpha \alpha'}}$ chosen in \autoref{diffdef} are such that the faces now anticommute. The result then follows.
\end{proof}

Furthermore, if we also require that the maps $\mu_V$, $\Delta_V$, and $m_V$ preserve the $\mathcal{G}$-grading up to a constant shift $s$, then $\partial$ also does the same, which is then compensated by the shift $s$ in $\mathcal{G}$-grading when moving between $R$-modules $V_\alpha$ and $V_{\alpha'}$ with $|\alpha'|=|\alpha|+1$ in the complex; in other words, $\partial$ has bigrading $(1, 0)$.

We can now define the homology of our complex as follows.

\begin{Definition}\label{D:hom}
Let $G$ be a trivalent graph with perfect matching $M$, and let $\Gamma_M$ be a perfect matching graph with underlying graph $G$. The \emph{Möbius homology} of $\Gamma_M$ is 
\begin{gather*}
MH^{i, g}(\Gamma_M ; R)= \frac{\text{ker}\, \partial: C^{i, g}(\Gamma_M) \rightarrow C^{i+1, g}(\Gamma_M)}{\text{im}\, \partial: C^{i-1, g}(\Gamma_M) \rightarrow C^{i, g}(\Gamma_M)},
\end{gather*}
which is a $\mathcal{G}$-graded $R$-module.
\end{Definition}

The additional (homological) grading variable is denoted by $t$.

\begin{Remark} As in the case of the bracket (also see \autoref{brackforribb}), we can define the Möbius homology for any ribbon graph by using the canonical perfect matching induced from the blowup and \autoref{L:blowupiso}.
\end{Remark}

Finally, we have the following theorem:

\begin{Theorem}\label{T:maintheorem}
The graded (Hilbert-)Poincaré polynomial
\begin{gather}\label{poin}
Mh(\Gamma_M) = \sum_{i} t^i  q\text{dim}\big(MH^{i,*}(\Gamma_M; R)\big)
\end{gather}
associated with the Möbius homology is a ribbon graph invariant. Furthermore, the graded Euler characteristic of the underlying chain complex is the bracket with $A = 1$, $B = -q_1^{s_1}q_2^{s_2} \cdots q_m^{s_m} $, and $C = q\text{dim}(V)$.
\end{Theorem}

\begin{proof}
Firstly, we know that the hypercube of states is a ribbon graph invariant. Furthermore, the preferred direction from the all $0$-smoothing state $|\alpha|=0$ to the all $1$-smoothing state $|\alpha|=\ell$ is independent of the choice of labeling for the perfect matching edges. By \autoref{R:symm}, the residual symmetry, when translated to the complex $C^{*, *}(\Gamma_M)$, amounts to a reordering of the direct sum of $R$-modules $V_\alpha$ with fixed $|\alpha|$. Furthermore, by \autoref{R:tensororder}, we also know that the chain groups and maps are defined up to an overall application of a tensor product swap map. Lastly, any relabeling of the perfect matching edges may also affect the definition of $\partial^i$ in \autoref{diffdef} by a relative minus sign. In any case, when taking the graded dimensions of the homology groups, the result is independent of any such choices.

To prove the second part of the statement, first note that $q\text{dim}(V\{s\}) = q_1^{s_1}q_2^{s_2} \cdots q_m^{s_m} \,q\text{dim}(V)$. Since all our $R$-modules are free of finite rank and $\partial$ preserves the $\mathcal{G}$-grading, the graded Euler characteristic of the chain complex $(C^{i, *}(\Gamma_M), \partial^i)$ is either given by the alternating sum of the graded dimensions of the $MH^{i, *}(\Gamma_M ; R)$ or of the $C^{i, *}(\Gamma_M)$ themselves. Using the latter, we obtain a sum of terms $(-q_1^{s_1}q_2^{s_2} \cdots q_m^{s_m})^{|\alpha|}\big(q\text{dim}(V)\big)^{k_\alpha}$, which is exactly the bracket with the given choice of variables.
\end{proof}

Later in \autoref{S:Exmp}, we will see examples demonstrating that the Möbius homology, through the associated graded Poincaré polynomial, is strictly more powerful than the bracket as a ribbon graph invariant.

\begin{Remark}
It would be interesting to study the Möbius homology using an asymptotic or big data approach, as, for example, in \cite{CoOsTu-growth,DlGuSa-data,KeLaTuVaZh-detection,TubbenhauerZhang-bigdata-quantum-invariants}. In particular, it would be interesting to have a weighting on the statement that the Möbius homology is strictly more powerful than the bracket.
\end{Remark}

\begin{Remark} 
One of the strongest features of link homologies is that they are not only invariants of links, but also behave functorially under link cobordisms; see, for example, \cite{BN-cobordisms,CMW-functoriality,Jac-cobordisms} or, quite generally, \cite{ETW-functoriality}. It is natural to ask for an analogous statement here: one would like suitable cobordisms between perfect matching ribbon graphs to induce maps on the homologies constructed in this paper. We do not pursue this in the present work, but the nonorientable cobordism model suggests that such a functorial extension should exist. 
\end{Remark}

\subsection{Geometric complex}\label{S:geo} 

We briefly mention another way to obtain a ribbon graph invariant using the hypercube of states and nonorientable cobordisms, which is more geometric in nature. The result is similarly a chain complex with an associated homology, but we are able to relax the somewhat artificial condition \autoref{specrel}; the trade-off is that the chain groups have a somewhat more complicated structure with less flexibility.

We begin by considering the $\K$-linearization of $\mcob{}$, for some fixed field $\K$, with closed surfaces evaluated according to 
\begin{gather}
\begin{tikzpicture}[baseline ={([yshift=-.5ex]current bounding box.center)}, draw, thick,scale=0.25]
\draw (0,0) circle (2cm);
\draw (-2,0) arc (180:360:2 and 0.6);
\draw[dashed, thick] (2,0) arc (0:180:2 and 0.6);
\end{tikzpicture}
\quad = \quad
\paraa_0, \hspace{0.75cm} \begin{tikzpicture}[scale=0.25, draw, thick, baseline ={([yshift=-.5ex]current bounding box.center)}]
\draw (0,0) ellipse (3cm and 2cm);
\draw (-1.5, 0.2) arc (180:360:1.5 and 0.6);
\draw (1,-0.2) arc (0:180:1 and 0.6);
\end{tikzpicture}
\quad = \quad
\paraa_1, \hspace{0.75cm}  \begin{tikzpicture}[scale=0.25, baseline ={([yshift=-.5ex]current bounding box.center)}, draw, thick]
\draw (0,0) ellipse (4.8cm and 2cm);
\draw (-3.5, 0.2) arc (180:360:1.5 and 0.6);
\draw (-1,-0.2) arc (0:180:1 and 0.6);
\draw (0.5, 0.2) arc (180:360:1.5 and 0.6);
\draw (3,-0.2) arc (0:180:1 and 0.6);
\end{tikzpicture}
\quad = \quad
\paraa_2 \label{series1}, \hspace{0.75cm} \ldots\,, \\
\begin{tikzpicture}[baseline ={([yshift=-.5ex]current bounding box.center)}, draw, thick,scale=0.25]
\draw (0,0) circle (2cm);
\draw (-2,0) arc (180:360:2 and 0.6);
\draw[dashed, thick] (2,0) arc (0:180:2 and 0.6);
\node at (0,-1){$\lightning$};
\end{tikzpicture}
\quad = \quad
\parab_0, \hspace{0.75cm} \begin{tikzpicture}[scale=0.25, draw, thick, baseline ={([yshift=-.5ex]current bounding box.center)}]
\draw (0,0) ellipse (3cm and 2cm);
\draw (-1.5, 0.2) arc (180:360:1.5 and 0.6);
\draw (1,-0.2) arc (0:180:1 and 0.6);
\node at (0,-1){$\lightning$};
\end{tikzpicture}
\quad = \quad
\parab_1, \hspace{0.75cm} \begin{tikzpicture}[scale=0.25, baseline ={([yshift=-.5ex]current bounding box.center)}, draw, thick]
\draw (0,0) ellipse (4.8cm and 2cm);
\draw (-3.5, 0.2) arc (180:360:1.5 and 0.6);
\draw (-1,-0.2) arc (0:180:1 and 0.6);
\draw (0.5, 0.2) arc (180:360:1.5 and 0.6);
\draw (3,-0.2) arc (0:180:1 and 0.6);
\node at (0,-1){$\lightning$};
\end{tikzpicture}
\quad = \quad
\parab_2, \hspace{0.75cm} \ldots\,, \label{series2} \\
\begin{tikzpicture}[baseline ={([yshift=-.5ex]current bounding box.center)}, draw, thick,scale=0.25]
\draw (0,0) circle (2cm);
\draw (-2,0) arc (180:360:2 and 0.6);
\draw[dashed, thick] (2,0) arc (0:180:2 and 0.6);
\node at (0,-1){$\lightning$};
\node at (0,1){$\lightning$};
\end{tikzpicture}
\quad = \quad
\parac_0, \hspace{0.75cm} \begin{tikzpicture}[scale=0.25, draw, thick, baseline ={([yshift=-.5ex]current bounding box.center)}]
\draw (0,0) ellipse (3cm and 2cm);
\draw (-1.5, 0.2) arc (180:360:1.5 and 0.6);
\draw (1,-0.2) arc (0:180:1 and 0.6);
\node at (0,-1){$\lightning$};
\node at (0,1.2){$\lightning$};
\end{tikzpicture}
\quad = \quad
\parac_1, \hspace{0.75cm} \begin{tikzpicture}[scale=0.25, baseline ={([yshift=-.5ex]current bounding box.center)}, draw, thick]
\draw (0,0) ellipse (4.8cm and 2cm);
\draw (-3.5, 0.2) arc (180:360:1.5 and 0.6);
\draw (-1,-0.2) arc (0:180:1 and 0.6);
\draw (0.5, 0.2) arc (180:360:1.5 and 0.6);
\draw (3,-0.2) arc (0:180:1 and 0.6);
\node at (0,-1){$\lightning$};
\node at (0,1){$\lightning$};
\end{tikzpicture}
\quad = \quad
\parac_2, \hspace{0.75cm} \ldots\, , \label{series3}
\end{gather}
and the empty cobordism $\emptyset$ is evaluated to $1\in\K$. The coefficients in \autoref{series1}--\autoref{series3} are obtained from the following generating functions:
\begin{gather*}
Z_{\paraa} = \frac{p_\paraa (T)}{q(T)} = \sum_{k \geq 0} \paraa_kT^k, \hspace{0.5cm} Z_{\parab} = \frac{p_\parab (T)}{q(T)} = \sum_{k \geq 0} \parab_kT^k, \hspace{0.5cm} Z_{\parac} = \frac{p_\parac (T)}{q(T)} = \sum_{k \geq 0} \parac_kT^k,
\end{gather*}
where $q(T) = 1 - a_1 T +a_2T^2+ \ldots + (-1)^Ma_MT^M \in \K(T)$ and furthermore, $p_\paraa(T), p_\parab(T), p_\parac(T)\in \K[T]$ satisfy
\begin{gather}
\deg\big(p_\parab(T)\big), \deg\big(p_\parac(T)\big) < K = \max(\deg\big(p_\paraa(T)\big)+1, M). \label{conds}
\end{gather}

Next, we make use of the \emph{universal construction} of \cite{BlHaMaVo-tqft-kauffman-bracket} with the given closed surface evaluations (the surfaces listed are the only closed surfaces by \autoref{mobrel}) to construct 
the category $\mucob{\paraa}{\parab}{\parac}$, introduced in \cite{CoTu-mobius}. From \cite[Proposition 3.11]{CoTu-mobius}, we know that $\mucob{\paraa}{\parab}{\parac}$ is symmetric and pivotal, and has a generator-relation presentation with generators \autoref{Eq:Gen} and \autoref{Eq:MGen}, and relations \autoref{startrel}--\autoref{endrel} and \autoref{mobmu}--\autoref{mobrel}, together with the \emph{handle relation}
\begin{gather*}
\sigma = h^K + \sum_{i=1}^M (-1)^ia_ih^{K-i}=0\leftrightsquigarrow
\begin{tikzpicture}[anchorbase]
\draw[usual,hol=0.5] (0,0) to (0,0.5)node[left]{$K$} to (0,1);
\end{tikzpicture}
-a_1\cdot 
\begin{tikzpicture}[anchorbase]
\draw[usual,hol=0.5] (0,0) to (0,0.5)node[left]{$(K-1)$} to (0,1);
\end{tikzpicture}
\pm\dots+(-1)^{\deg q}a_{\deg q}\cdot
\begin{tikzpicture}[anchorbase]
\draw[usual,hol=0.5] (0,0) to (0,0.5)node[left]{$(K-\deg q)$} to (0,1);
\end{tikzpicture}
=0.
\end{gather*}
(The number next to the dot denotes its multiplicity.)
In particular, the hom spaces of $\mucob{\paraa}{\parab}{\parac}$ are finite dimensional vector spaces whose elements contain a maximum of $K-1$ handles.
For more details on the construction of $\mucob{\paraa}{\parab}{\parac}$, see \cite{CoTu-mobius}.

\begin{Remark}
Here and in \autoref{S:IGrad}, we use a field instead of a commutative ring 
only because \cite{CoTu-mobius} did. This requirement can be relaxed if needed.
\end{Remark}

Let us now take $n \in \mathbb{N}_{>0}$. Suppose that we choose $p_{\paraa}(T)$ such that 
\begin{gather*}
\deg\big(p_\paraa(T)\big) =\begin{cases*}
n  & $n$ even,\\
n-1 & $n$ odd,
\end{cases*}
\end{gather*}
and that $q(T) = 1 - T$. The handle relation then reads $h^{n+1} = h^{n}$ for $n$ even, and $h^{n}=h^{n-1}$ for $n$ odd i.e. $K = n+1, n$ for $n$ even, odd, respectively. We then choose $p_\parab(T), p_\parac(T)$ such that \autoref{conds} is satisfied. We then have a family of categories $\mucob{\paraa}{\parab}{\parac} ^n$ for each $n$.

We now build a chain complex for each $n$. Given a perfect matching graph $\Gamma_M$ with trivalent underlying graph $G = (V, E, r)$ together with perfect matching $M = \{e_1, e_2, \ldots, e_\ell\} \subset E$, we again consider the hypercube of states with directed edges as before. To each state with $k_\alpha$ circles, we associate the finite-dimensional vector space $V_{\alpha,n} =\hom_{\mucob{\paraa}{\parab}{\parac} ^n}(0, k_\alpha)$. We then define a complex $C^{*, *}_n(\Gamma_M)$ as in \autoref{compdef} (we will come to the grading later); note that \autoref{R:correscomp} and \autoref{R:symm} also apply in this setting.

The differential $\partial^{i}_n: C^{i, *}_n(\Gamma_M) \rightarrow C^{i+1, *}_n(\Gamma_M)$ can also be constructed in the same manner as before: recall that there were three scenarios describing the change of state moving from one vertex to another: $i)$ two circles merge, $ii)$ one circle splits in two, or $iii)$ one circle crosses itself. We need to define linear maps representing these three scenarios. To begin, let
\begin{gather*}
\zeta_n =\begin{cases*}
h^{n/2}\circ m & $n$ even,\\
h^{(n-1)/2}\circ m & $n$ odd,
\end{cases*}
\end{gather*}
be cobordisms i.e. $\zeta_n \in \End_{\mucob{\paraa}{\parab}{\parac} ^n}(1)$ for each $n$. Given elements $X_n\in \hom_{\mucob{\paraa}{\parab}{\parac} ^n}(0, 1)$ and $Y_n\in \hom_{\mucob{\paraa}{\parab}{\parac} ^n}(0, 2)$, we define the following linear maps:
\begin{align*}
\mu_n: \hom_{\mucob{\paraa}{\parab}{\parac} ^n}(0, 2) \rightarrow \hom_{\mucob{\paraa}{\parab}{\parac} ^n}(0, 1), \hspace{0.5cm} Y_n &\mapsto  \zeta_n \circ \mu \circ Y_n,  \\
\Delta_n: \hom_{\mucob{\paraa}{\parab}{\parac} ^n}(0, 1) \rightarrow \hom_{\mucob{\paraa}{\parab}{\parac} ^n}(0, 2), \hspace{0.5cm} X_n &\mapsto \Delta \circ \zeta_n \circ X_n, 
\\
m_n: \hom_{\mucob{\paraa}{\parab}{\parac} ^n}(0, 1) \rightarrow \hom_{\mucob{\paraa}{\parab}{\parac} ^n}(0, 1), \hspace{0.5cm} X_n &\mapsto \zeta_n \circ X_n,
\end{align*}
and associate them with the directed edges corresponding to scenarios $i)$, $ii)$, and $iii)$ respectively, tensored by the appropriate number of identity maps (\autoref{R:tensororder} also applies here).

\begin{Example}\label{E:basicvec}
For $n=1$, $h_1 = 1_1$ i.e. there are no handles, and we have three basis vectors for $\hom_{\mucob{\paraa}{\parab}{\parac} ^1}(0, 1)$:
\begin{gather*} X^1_{1} :
\begin{tikzpicture}[baseline ={([yshift=-.5ex]current bounding box.center)},scale=0.25]
\pic[
tqft,
incoming boundary components=1,
outgoing boundary components=0,
every lower boundary component/.style={draw},
genus=0,
scale=1.5,
draw,
thick,
name=cup
]; 
\end{tikzpicture}, \hspace{1cm} X^2_{1} : \begin{tikzpicture}[baseline ={([yshift=-.5ex]current bounding box.center)},scale=0.25]
\pic[
tqft,
incoming boundary components=1,
outgoing boundary components=0,
every lower boundary component/.style={draw},
genus=0,
scale=1.5,
draw,
thick,
name=cup
]; 
\node at ([xshift=0pt, yshift=-10pt]current bounding box.center) {$\lightning$};
\end{tikzpicture},
\hspace{1cm} X^3_{1} : \begin{tikzpicture}[baseline ={([yshift=-.5ex]current bounding box.center)},scale=0.25]
\pic[
tqft,
incoming boundary components=1,
outgoing boundary components=0,
every lower boundary component/.style={draw},
genus=0,
scale=1.5,
draw,
thick,
name=cups
]; 
\node at ([xshift=-15pt, yshift=-10pt]current bounding box.center)  {$\lightning$};
\node at ([xshift=20pt, yshift=-10pt]current bounding box.center)  {$\lightning$};
\end{tikzpicture},
\end{gather*}
and we have 
\begin{gather*}
m_1\colon \hspace{0.5cm} X^1_{1} :\begin{tikzpicture}[baseline ={([yshift=-.5ex]current bounding box.center)},scale=0.25]
\pic[
tqft,
incoming boundary components=1,
outgoing boundary components=0,
every lower boundary component/.style={draw},
genus=0,
scale=1.5,
draw,
thick,
boundary separation=40pt,
name=pop
]; 
\end{tikzpicture} \quad \mapsto \quad \begin{tikzpicture}[tqft, baseline ={([yshift=-.5ex]current bounding box.center)},scale=0.25, cobordism/.style={draw},
every upper boundary component/.style={draw},
every lower boundary component/.style={draw}]
\pic[tqft/cylinder, thick, name=a,scale=0.75];
\pic[tqft/cup,draw, thick,scale=0.75, name=b, anchor = incoming boundary 1, at =(a-outgoing boundary 1)];
\node at ([xshift=0pt, yshift=10pt]current bounding box.center)  {$\lightning$};
\end{tikzpicture} \quad = \quad X^2_{1} :\begin{tikzpicture}[baseline ={([yshift=-.5ex]current bounding box.center)},scale=0.25]
\pic[
tqft,
incoming boundary components=1,
outgoing boundary components=0,
every lower boundary component/.style={draw},
genus=0,
scale=1.5,
draw,
thick,
name=cups
]; 
\node at ([xshift=0pt, yshift=-10pt]current bounding box.center)  {$\lightning$};
\end{tikzpicture},
\end{gather*}
and similarly for the other basis vectors and maps.
\end{Example}

Recall from \autoref{P:commute} that if we could get all the faces of our complex to commute, then the complex itself forms a commutative diagram; we would like to do the same here. Also recall from the corresponding proof that opposing vertices can have tensor powers differing by either $i)$ 0, $ii)$ 1, or $iii)$ 2. Instead of tensor powers, in the present case we have the positive integer $k_\alpha$ appearing in $\hom_{\mucob{\paraa}{\parab}{\parac} ^n}(0, k_\alpha)$, which gives the number of circles in the corresponding state $\Gamma_\alpha$ in the hypercube of states. 

Using identical reasoning as in the proof of \autoref{P:commute}, we only need to consider faces whose paths contain a composition of maps with $\mu_n$, $\Delta_n$, or $m_n$ operating on the same circles. Faces corresponding to scenario $iii)$ again must either trivially commute, or commute by associativity or coassociativity. Also, since handles and Möbius strip insertions move freely around cobordisms, faces corresponding to scenario $ii)$ also commute. For scenario $i)$, we again have that the only new nontrivial relation that must be satisfied is $m_n^2 = \mu_n \circ \Delta_n$,
which holds automatically: both sides contain two Möbius strip insertions, and the r.h.s. contains an extra handle with respect to the l.h.s. that is then eliminated by the handle relation. By the classification of surfaces, we therefore conclude that they represent the same cobordism.

If we then define a differential $\partial^{i}_n$ the same as in \autoref{diffdef}, we can construct a chain complex by the same reasoning as in the proof of \autoref{T:chaincomp}. Let us now define a grading by setting $\mathcal{G}= \Z/2\Z$. Given a positive integer $k$, the grading of an element $Z_n \in \hom_{\mucob{\paraa}{\parab}{\parac} ^n}(0, k)$ is given by $0(1)$ for an even(odd) number of Möbius strip insertions in total; note that the parity of Möbius strip insertions is respected by \autoref{mobrel}, and so the grading is well-defined.

\begin{Example}
For $n=1$ following \autoref{E:basicvec}, we have $\hom_{\mucob{\paraa}{\parab}{\parac} ^1}(0, 1) = W_0 \oplus W_1$ with $X^1_{1}, X^3_{1} \in W_0$ and $X^2_{1} \in W_1$.
\end{Example}

The point is that the maps $\mu_n$, $\Delta_n$, and $m_n$ all contain a single Möbius strip insertion; hence, the differential $\partial^{i}_n$ preserves the $\Z/2\Z$-grading up to a shift of size 1. By shifting the $\Z/2\Z$-grading of each of the vector spaces $V_{\alpha, n}$ by the appropriate amount for each fixed $|\alpha|$, we then have that $\partial_n$ has bigrading $(1, 0)$. Finally, we can define the homology as in \autoref{D:hom} and associated graded Poincaré polynomial, which is a ribbon graph invariant for the same reasons given in the proof of the first part of \autoref{T:maintheorem}. A more thorough analysis of the resulting invariant is left for future work. 

\subsection{An integer graded version}\label{S:IGrad}

The fact that the handle relation $h^{n+1} = h^n$ for $n$ even, and $h^n = h^{n-1}$ for $n$ odd, is monoid-like prevents us from obtaining a more interesting grading than $\mathcal{G}= \Z/2\Z$. As an alternative, let us return to working with $\mcob{}$ instead of $\mucob{\paraa}{\parab}{\parac}$. In place of $\hom_{\mucob{\paraa}{\parab}{\parac} ^n}(0, k_\alpha)$, define the vector space $V_\alpha'$ as $\K$-linear combinations of cobordisms in $\hom_{\mcob{}}(0, k_\alpha)$ whose connected components are all nonorientable. In place of $
\mu_n, \Delta_n$, and $m_n$, define the maps
\begin{align*}
\mu: \hom_{\mcob{}}(0, 2) \rightarrow \hom_{\mcob{}}(0, 1), \hspace{0.5cm} Y &\mapsto   \mu \circ Y,  \\
\Delta: \hom_{\mcob{}}(0, 1) \rightarrow \hom_{\mcob{}}(0, 2), \hspace{0.5cm} X &\mapsto \Delta \circ X, 
\\
m: \hom_{\mcob{}}(0, 1) \rightarrow \hom_{\mcob{}}(0, 1), \hspace{0.5cm} X &\mapsto m \circ X,
\end{align*}
extended $\K$-linearly. Given any element $X \in \hom_{\mcob{}}(0, 1)$ whose connected components are all nonorientable, we must have an element $X' \in \hom_{\mcob{}}(0, 1)$ with $X = m \circ X'$. Now notice that we have 
\begin{gather*}
(\mu \circ \Delta)(X) = \mu \circ \Delta \circ X = \mu \circ \Delta \circ m \circ X' = m^3 \circ X' = m^2 \circ X = m^2(X),
\end{gather*}
where we have used the relation \autoref{mobrel} in the third equality. Therefore, the problematic face commutes (the other face types commute trivially or using the relations in $\mcob{}$ as previously) and we can again build a chain complex as before. 

Furthermore, for the grading, set $\mathcal{G} =\Z$, and for each of the $\circ$-$\otimes$-generators of $\mcob{}$, assign the following integers $1_1, s:0$, $\eta, \epsilon: -1$, and $\mu, \Delta, m : 1$. Define the grading of a cobordism in $V_\alpha'$ by decomposing it into a $\circ$-$\otimes$-product of generators, and summing over all the integers assigned to each generator appearing in the decomposition. Since the $\mcob{}$ relations respect the given assignments (the reader is encouraged to check this), the grading is well-defined. Clearly, the differentials built from $\mu, \Delta, m$ preserve the grading up to a shift of size 1. 

Finally, we note that the $V_\alpha'$ are infinite dimensional as the genus of cobordisms may be unbounded. To produce a finite-dimensional space, we can form a subspace of each $V_\alpha'$ by restricting to cobordisms with a certain fixed genus or less. In addition, to ensure that the images of the maps comprising the differential lie in the relevant codomains, we must increase the maximum genus every second step in the chain to account for appearances of $h=\mu \circ \Delta$. We are then left with a choice of maximum genus $n \in \mathbb{N}$ for the first nonzero chain group. 

Altogether, after taking the homology and forming the associated graded Poincaré polynomial, we have another family of ribbon graph invariants, one for each $n \in \mathbb{N}$, though this invariant appears to be much more difficult to compute in practice. 


\section{Möbius Frobenius algebras}\label{S:Frob}


We assume some familiarity with \cite{Ko-tqfts} or similar references.

\subsection{Definition}

In analogy with the ordinary Frobenius algebra and TQFT correspondence, we now establish a bijective correspondence that allows us to work with more compact objects, namely Frobenius algebras with some extra structure, instead of MTQFTs. We begin with the following definition.

\begin{Definition}\label{D:Mobiusfrobalg} 
Let $\mathcal{C}$ be a braided monoidal category. A \emph{Möbius Frobenius algebra} in $\mathcal{C}$ is a 6-tuple $(A, \mu_A, \eta_A, \Delta_A, \epsilon_A, m_A)$, where $(A, \mu_A, \eta_A, \Delta_A, \epsilon_A)$ is a commutative Frobenius algebra in $\mathcal{C}$ and $m_A: A \rightarrow A$ is a morphism such that the diagrams
\begin{equation} \label{mobfrob1}
\begin{tikzcd}
A \otimes A  \arrow[d, "\mu_A"']
& A \otimes A \arrow[d,
"\mu_A" ] \arrow[l,
"m_A \otimes \text{id}_A"' ] \arrow[r,
"\text{id}_A \otimes m_A" ]& A\otimes A \arrow[d, "\mu_A"]\\
A 
&  A \arrow[l,
"m_A"] \arrow[r, "m_A"'] & A
\end{tikzcd}
\end{equation}
and
\begin{equation} \label{mobfrob2}
\begin{tikzcd}
A   \arrow[r,
"\Delta_A" ]
& A \otimes A \arrow[d,
"\mu_A" ] \\
A \arrow[r,
"m_A^3"'] \arrow[u, "m_A"] 
&  A 
\end{tikzcd}
\end{equation}
commute. 

Two Möbius Frobenius algebras $(A, \mu_A, \eta_A, \Delta_A, \epsilon_A, m_A)$ and $(A', \mu_{A'}, \eta_{A'}, \Delta_{A'}, \epsilon_{A'}, m_{A'})$ are said to be \emph{isomorphic} if there exists an isomorphism of commutative Frobenius algebras $f: A \rightarrow A'$ such that $m_{A'} \circ f = f \circ m_A$.
\end{Definition}

\begin{Remark}
The morphism $m_A: A \rightarrow A$ is not required to be a morphism of Frobenius algebras in the sense of \cite{Ko-tqfts}; in particular, it need not preserve the unit or counit. 
\end{Remark}

\begin{Example}
Take $n \in \mathbb{N}_{>0}$. Consider the subcategory of $\K\textbf{Vec}=\K\textbf{Mod}$ whose objects are given by $\hom_{\mucob{\paraa}{\parab}{\parac} ^n}(0, k)$ with $k$ a positive integer (see \autoref{S:geo}). The morphisms in $\hom_{\mucob{\paraa}{\parab}{\parac} ^n}(0, k) \rightarrow \hom_{\mucob{\paraa}{\parab}{\parac} ^n}(0, k')$ are given by $\hom_{\mucob{\paraa}{\parab}{\parac} ^n}(k, k')$ using composition in the same manner e.g. \autoref{E:basicvec}; call the resulting category $\mathcal{C}$. 

Now define a monoidal product in $\mathcal{C}$ given on objects as $$\hom_{\mucob{\paraa}{\parab}{\parac} ^n}(0, k) \, \square \, \hom_{\mucob{\paraa}{\parab}{\parac} ^n}(0, k') = \hom_{\mucob{\paraa}{\parab}{\parac} ^n}(0, k+ k')$$ with unit object $\hom_{\mucob{\paraa}{\parab}{\parac} ^n}(0, 0) \cong \K$ and braiding given by appropriate insertions of $s$ from \autoref{Eq:Gen}. On morphisms, $\square$ is disjoint union of cobordisms. Given $V = \hom_{\mucob{\paraa}{\parab}{\parac} ^n}(0, 1)$, we have that $\mathcal{C}$ is generated by the Möbius Frobenius algebra $(V, \mu, \eta, \Delta, \epsilon, m)$.
\end{Example}

\subsection{Classification}

We have the following important theorem:

\begin{Theorem}\label{T:Mobiusfrobtqftbij} 
Isomorphism classes of MTQFTs are in bijective correspondence with isomorphism classes of Möbius Frobenius algebras in $R\textbf{Mod}$.
\end{Theorem}

\begin{proof}
First suppose that we are given a Möbius TQFT $\mathcal{F}: \mcob{} \rightarrow R\textbf{Mod}$. If we define $V= \mathcal{F}(1)$, then it is clear that we can define $R$-linear maps $\mu_V=\mathcal{F}(\mu):V \otimes V \rightarrow V$, $\eta_V=\mathcal{F}(\eta): R \rightarrow V$, $\Delta_V=\mathcal{F}(\Delta):V \rightarrow V \otimes V$, $\epsilon_V = \mathcal{F}(\epsilon): V \rightarrow R$, and $m_V=\mathcal{F}(m): V \rightarrow V$. From \autoref{rmk2cob}, it follows that $(V, \mu_V, \eta_V, \Delta_V, \epsilon_V)$ is a commutative Frobenius algebra in $R\textbf{Mod}$ using the usual correspondence between oriented TQFTs and commutative Frobenius algebras. Furthermore, relations \autoref{mobmu} and \autoref{mobrel} ensure that $m_V$ satisfies \autoref{mobfrob1} and \autoref{mobfrob2}, and hence $(V, \mu_V, \eta_V, \Delta_V, \epsilon_V, m_V)$ is a Möbius Frobenius algebra in $R\textbf{Mod}$. Finally, it is clear that given two Möbius TQFTs $\mathcal{F}$, $\mathcal{F}'$ and a symmetric monoidal natural isomorphism $\alpha: \mathcal{F} \rightarrow \mathcal{F}'$, the $R$-linear map $\alpha_{\textbf{1}}:\mathcal{F}(1) \rightarrow \mathcal{F}'(1)$ is an isomorphism of Möbius Frobenius algebras in $R\textbf{Mod}$.

Now suppose that we are given a Möbius Frobenius algebra $(V, \mu_V, \eta_V, \Delta_V, \epsilon_V, m_V)$ in $R\textbf{Mod}$. Define a functor $\mathcal{F}:\mcob{} \rightarrow R\textbf{Mod}$ by setting $\mathcal{F}(n)=V^{\otimes n}$ with $V^{\otimes 0} = R$ on objects. To define $\mathcal{F}$ on morphisms, we refer to any given cobordism (or rather the corresponding isomorphism class of said cobordism) as a triple $(S, B_I, B_F)$ with $S$ a compact surface whose boundary is a disjoint union of two closed 1-manifolds $B_I = \bigsqcup_{i=1}^n \mathbb{S}^1$ and $B_F = \bigsqcup_{i=1}^m \mathbb{S}^1$ with fixed orientation. If $S$ is orientable, then we can use the correspondence between oriented TQFTs and commutative Frobenius algebras to build a well-defined $R$-linear map $\mathcal{F}(S, B_I, B_F)$ using the corresponding generators $\mu_V, \eta_V, \Delta_V, \epsilon_V$ in place of $\mu, \eta, \Delta, \epsilon$ and relations \autoref{startrel}--\autoref{endrel}.

Next consider the case that $S$ is nonorientable i.e. there must be at least one $m$ in any $\circ$-$\otimes$ decomposition by generators. We then build $\mathcal{F}(S, B_I, B_F)$ exactly as in the oriented case, except with insertions of $m_V$ whenever $m$ appears in $S$. The resulting map is well-defined as usual, since $m_V$ satisfies the same relations as $m$, together with their other respective generators.

Now we show that $\mathcal{F}$ is indeed a functor, and that it is symmetric monoidal. The latter is clear. It is also clear that the identity cobordism is sent to the identity $R$-linear map on $V$. Furthermore, given two cobordisms $\big(S_1, (B_I)_1, (B_F)_1\big)$ and  $\big(S_2, (B_I)_2, (B_F)_2\big)$ with $(B_F)_1  = (B_I)_2$, we can simply form the composite $\big(S_3, (B_I)_3, (B_F)_3\big) = \big(S_2 \circ S_1, (B_F)_2, (B_I)_1\big)$ much like the oriented case, since the boundaries have fixed orientations and are all copies of $\mathbb{S}^1$ i.e. all homeomorphisms between the boundaries must be trivial. By construction, it then follows that $$\mathcal{F}\big(S_3, (B_I)_3, (B_F)_3\big) = \mathcal{F}\big(S_2, (B_I)_2, (B_F)_2\big) \circ \mathcal{F}\big(S_1, (B_I)_1, (B_F)_1\big)$$ as required.

In addition, it is a straightforward exercise to prove that if there exists an isomorphism of Möbius Frobenius algebras $f: V \rightarrow V'$, then there is a symmetric monoidal natural isomorphism $\alpha: \mathcal{F} \rightarrow \mathcal{F}'$.

Finally, it is clear that the two constructions above are inverses of one another.
\end{proof}

\begin{Remark}
The proof of the above theorem is much simpler than the proof of a similar result for unoriented cobordisms e.g. in \cite{TuTu-utqft} primarily because of the fact that in our case the boundaries are all copies of $\mathbb{S}^1$ with fixed orientations, and so there are no nontrivial homeomorphisms between them.
\end{Remark}

\begin{Example}
Fix $R = \mathbb{Z}[\frac{1}{3}]$, and consider the MTQFT $\mathcal{F}$ defined in \autoref{E:mobn1}. We define the corresponding Möbius Frobenius algebra $(V, \mu_V, \eta_V, \Delta_V, \epsilon_V, m_V)$ by $V= \mathcal{F}(1)  $, $\mu_V = \mathcal{F}(\mu)$, $\eta_V = \mathcal{F}(\eta)$, $\Delta_V = \mathcal{F}(\Delta)$, $\epsilon_V = \mathcal{F}(\epsilon)$, and $m_V = \mathcal{F}(m)$ as shown there.
\end{Example}

We can translate the conditions on an MTQFT $\mathcal{F}: \mcob{} \rightarrow R\textbf{Mod}$ used to construct a chain complex and the Möbius homology of ribbon graphs from the previous section to the corresponding Möbius Frobenius algebra $(V, \mu_V, \eta_V, \Delta_V, \epsilon_V, m_V)$ as follows: we require that $V$ is a finite-dimensional free $\mathcal{G}$-graded $R$-module whose $R$-linear maps $\mu_V$, $\Delta_V$, and $m_V$ preserve the $\mathcal{G}$-grading up to the same constant shift, and that $m_V^2 = \mu_V \circ \Delta_V$. We can then build an identical chain complex with the $R$-module $V$ together with the maps $\mu_V$, $\Delta_V$, and $m_V$. We will therefore work exclusively with Möbius Frobenius algebras going forward.


\section{A choice related to $n$-coloring}\label{S:ncolorchoice}


We now explicitly construct a family of Möbius Frobenius algebras from which we define a ribbon graph invariant using the procedure outlined in \autoref{graphhomo}. We then compare the result to another known ribbon graph invariant, the \emph{Penrose polynomial}, which has been the subject of much research (see \eg \, \cite{Aig-penrose,ElMo-penrose}
and, of course, \cite{Pen-negative}) and whose evaluation at certain integral values relates to definitive properties of graphs \eg \, the number of $3$-edge colorings (or the number of $4$-face colorings); we will see more examples at the end of this section.

\subsection{Möbius Frobenius algebras for the Penrose polynomial}

Take $n \in \mathbb{N}_{>0}$, and set $R =\mathbb{Z}[\frac{1}{3n}]$. Consider the $R$-module
\begin{gather*}
V = R[x,y]/(x^n, y^3 -xy).
\end{gather*}
We can immediately write down a basis $B = \{1, y, y^2, x, xy, xy^2, \ldots, x^{n-1}, x^{n-1}y, x^{n-1}y^2\}$ of size $|B|=3n$. 

Now let $\mu_V: V \otimes V \rightarrow V$ be the usual multiplication map on the polynomial ring $V$, with $\eta_V: R \rightarrow V$, $r \mapsto r$ as the unit. Taking as the counit $\epsilon_V: V \rightarrow R$ with $\epsilon(x^{n-1}y^2) = 3n$ and zero on the other basis vectors, we have as the comultiplication $\Delta_V: V \rightarrow V \otimes V$:
\begin{gather}\label{comultdef}
\begin{aligned}
\Delta_V(v) =& -\frac{1}{3n}  \sum\limits_{\substack{i+j=n \\ i, j >0}} (v \cdot x^i) \otimes x^j \\
&+ \frac{1}{3n}\sum\limits_{\substack{i+j=n-1}} (v \cdot x^i y^2) \otimes x^j + (v \cdot x^iy) \otimes x^jy + (v \cdot x^i) \otimes x^jy^2,
\end{aligned}
\end{gather}
for all $v \in V$. In other words, $(V, \mu_V, \eta_V, \Delta_V, \epsilon_V)$ is a commutative Frobenius algebra in $R\textbf{Mod}$. Furthermore, if we set 
\begin{gather}\label{mdef}
m_V(v) =\begin{cases*}
v \cdot x^{(n-2)/2} y^2 & $n$ even,\\
v \cdot x^{(n-1)/2} y & $n$ odd,
\end{cases*}
\end{gather}
for all $v \in V$, then \autoref{mobfrob1} is automatically satisfied. We also have that
\begin{gather*}
(\mu_V \circ \Delta_V)(v) = v \cdot x^{n-1}y^2,
\end{gather*}
while
\begin{gather*}
m_V^2(v) =\begin{cases*}
v \cdot x^{(n-2)/2}\cdot m_V( y^2) & $n$ even,\\
v \cdot x^{(n-1)/2} \cdot m_V( y)& $n$ odd
\end{cases*} = v \cdot x^{n-1}y^2 = (\mu_V \circ \Delta_V)(v),
\end{gather*}
which implies \autoref{mobfrob2}. Therefore, we get the following.

\begin{Lemma}
The above choices
$(V, \mu_V, \eta_V, \Delta_V, \epsilon_V, m_V)$ give a Möbius Frobenius algebra in $R\textbf{Mod}$.
\end{Lemma}

\begin{proof}
An easy calculation.
\end{proof}

\begin{Example}\label{n2}
We have already seen the simplest case $n=1$ in \autoref{E:mobn1}, so let us look at the simplest even case $n =2$: firstly, we set $R= \mathbb{Z}[\frac{1}{6}]$. The base $R$-module is $ V = R[x,y]/(x^2, y^3 -xy)$, which with respect to the basis $B = \{1, y, y^2, x, xy, xy^2\}$ has multiplication table
\begin{center}
\begin{tabular}{ c |c c c c c c } 
& $1$ & $y$ & $y^2$ & $x$ & $xy$ & $xy^2$ \\ 
\hline
$1$ & $1$ & $y$ & $y^2$ & $x$ & $xy$ & $xy^2$ \\ 
$y$ & $y$ & $y^2$ & $xy$  & $xy $& $xy^2$ & 0\\ 
$y^2$ & $y^2$ & $xy$ & $xy^2$  & $xy^2 $& 0 & 0\\ 
$x$ & $x$ & $xy$ &  $xy^2$ & 0&  0& 0\\ 
$xy$ & $xy$ & $xy^2$ &  0 & 0&  0& 0\\ 
$xy^2$ & $xy^2$ & 0 &  0 & 0 & 0 & 0 \\ 
\end{tabular}.
\end{center}
Given counit $\epsilon_V(xy^2)=3\cdot2=6$ and zero otherwise, the pairing matrix is given by
\begin{gather*}
\begin{psmallmatrix}
0 & 0 & 0 & 0 & 0 & 6\\
0 & 0 & 0 & 0 & 6 & 0 \\
0 & 0 & 6 & 6 & 0 & 0 \\
0 & 0 & 6 & 0 & 0 & 0 \\
0 & 6 & 0 & 0 & 0 & 0 \\
6 & 0 & 0 & 0 & 0 & 0 \\
\end{psmallmatrix},
\end{gather*}
which has inverse 
\begin{gather*}
\frac{1}{6}\begin{psmallmatrix}
0 & 0 & 0 & 0 & 0 & 1\\
0 & 0 & 0 & 0 & 1 & 0 \\
0 & 0 & 0 & 1 & 0 & 0 \\
0 & 0 & 1 & -1 & 0 & 0 \\
0 & 1 & 0 & 0 & 0 & 0 \\
1 & 0 & 0 & 0 & 0 & 0 \\
\end{psmallmatrix},
\end{gather*}
and so $\Delta_V(1)=-\frac{1}{6} x \otimes x + \frac{1}{6}(xy^2 \otimes 1 + xy\otimes y + x \otimes y^2 + y^2\otimes x + y \otimes xy + 1 \otimes xy^2)$. Furthermore, we have $m_V(1) = y^2$, and so $m^2_V(1) = y^2m_V(1)=y^4 = xy^2=\mu_V\big(\Delta_V(1)\big)$, where the first equality follows since $m_V(v) = v \cdot m_V(1)$ by construction for all $v \in V$. Therefore, for all $v \in V$, we have $m_V^2(v) = v \cdot m_V^2(1)$, and in addition, also by construction, $\mu_V\big(\Delta_V(v)\big)= v \cdot \mu_V\big(\Delta_V(1)\big)$, and so we explicitly see that $m_V^2(v)  = \mu_V\big(\Delta_V(v)\big) = (\mu_V \circ \Delta_V)(v)$.
\end{Example}

At the end of \autoref{S:Frob}, we outlined the conditions that our Möbius Frobenius algebra given by $(V, \mu_V, \eta_V, \Delta_V, \epsilon_V, m_V)$ must satisfy in order to construct the Möbius homology of a given ribbon graph from \autoref{graphhomo}. We already have that $V$ is free and finite dimensional, and that $m_V^2 = \mu_V \circ \Delta_V$. We now need to specify a grading; for this, let
\begin{gather*}
\mathcal{G} =\begin{cases*}
\Z/n\Z \times \Z/2\Z & $n$ even,\\
\Z/n\Z & $n$ odd,
\end{cases*}
\end{gather*}
and set $V = \oplus_{g \in \mathcal{G}}W_g$ such that
\begin{gather}\label{graddef}
x^iy^j \in \begin{cases*}
W_{(2i+j, j)} & $n$ even,\\
W_{2i+j} & $n$ odd,
\end{cases*}
\end{gather}
where $0 \leq i \leq n-1$ and $0 \leq j \leq2$.

\begin{Proposition}\label{gradwelldef}
The $\mathcal{G}$-grading is well-defined for any element $x^iy^j$ in $V$ with indices $i,j$ defined over all nonnegative integers.
\end{Proposition}

\begin{proof}
We would like to establish that the $\mathcal{G}$-grading respects, in particular, the relations $x^n = 0$ and $y^3 = xy$ defining $V$. Let us begin by expressing any nonnegative integers $i, j$ as $i = q n + r$ and $j = 3q' + r'$ with $q, q'$ nonnegative integers, and $0 \leq r \leq n-1$, $0 \leq r' \leq 2$. We then have 
\begin{align}
W_{(2i+j, j)} &= W_{(2r+3q'+r', q'+r')} \hspace{0.5cm} \text{for $n$ even} \label{grad1} \\
W_{2i+j} &= W_{2r+3q'+r'} \hspace{0.5cm} \text{for $n$ odd} \label{grad2}
\end{align}
Note that, if $q > 0$, then $x^{i}y^{j} = 0$, which is in every graded subspace; in particular, the $\mathcal{G}$-grading respects the relation $x^n=0$.

Furthermore, if $q'>0$ and $q=0$, we have $x^{i}y^{j} = x^{r}y^{3q'+r'} = x^{r+q'}y^{q'+r'}$ in $V$. We have 
\begin{align*}
x^{r+q'}y^{q'+r'} \in \begin{cases*}
W_{(2r + 2q'+q'+r', q'+r')} = W_{(2r + 3q'+r', q'+r')}= W_{(2i + j, j)}  & $n$ even,\\
W_{2r + 2q'+q'+r'} = W_{2r + 3q'+r'}= W_{2i + j}  & $n$ odd,
\end{cases*}
\end{align*}
where the final equality in both cases uses \autoref{grad1} and \autoref{grad2} respectively. The $\mathcal{G}$-grading therefore respects the relation $y^3=xy$.

The remaining case i.e. $q=q'=0$ is simply the base case covered in \autoref{graddef}.
\end{proof}

\begin{Example}\label{gradedexmp}
For $n=1$, the grading is trivial. For $n=2$, we have the following decomposition $V = W_{(0,0)} \oplus W_{(1,0)}\oplus W_{(0,1)} \oplus W_{(1,1)}$ with $1, y^2, x, xy^2 \in W_{(0,0)}$ and $y, xy \in W_{(1,1)}$, with $W_{(1,0)}=W_{(0,1)}=\{0\}$. 
\end{Example}

\begin{Proposition}
The maps $\mu_V$, $\Delta_V$, and $m_V$ preserve the $\mathcal{G}$-grading with zero shift.
\end{Proposition}

\begin{proof}
First, observe that $x^iy^j \cdot x^{i'} y^{j'} = x^{i+i'}y^{j+j'}$ with 
\begin{gather*}
x^{i+i'}y^{j+j'} \in \begin{cases*}
W_{(2i + 2i'+j+j', j+j')} = W_{(2i + j, j)+(2i'+j', j')}  & $n$ even,\\
W_{2i + 2i'+j+j'} = W_{(2i + j)+(2i'+j')} & $n$ odd,
\end{cases*}
\end{gather*}
where we have used \autoref{gradwelldef} to define the $\mathcal{G}$-grading for any nonnegative $i, i', j, j'$, which shows that $\mu_V$ preserves the $\mathcal{G}$-grading with zero shift.

Now suppose that we are given a $v\in V$ such that $v \in W_g$ for some $g \in \mathcal{G}$. From the definition \autoref{comultdef} of comultiplication, the first summation in $\Delta_V(v)$ contains terms with $\mathcal{G}$-grading $g+(2i,0)+(2j,0)=g+(2(i+j),0)=g+(2n, 0)=g+ (0,0)=g$ for $n$ even, and an analogous computation for $n$ odd. Similarly, the second summation contains terms with $\mathcal{G}$-grading $g+(2(i+j)+2,2)=g+(2(n-1)+2, 0)=g+ (2n,0)=g+ (0,0)=g$ for $n$ even, and an analogous computation for $n$ odd. Therefore, $\Delta_V$ preserves the $\mathcal{G}$-grading with zero shift.

Finally, from the definition \autoref{mdef}, the element $m_V(v)$ has $\mathcal{G}$-grading $g+ (2(n-2)/2+2,2)=g+ (n,0) = g+ (0,0) =g$ for $n$ even, and $g + 2(n-1)/2 + 1=g+n=g$ for $n$ odd. Therefore, $m_V$ also preserves the $\mathcal{G}$-grading with zero shift, and the proposition is proven.
\end{proof}

We can now build a chain complex and Möbius homology associated with the specified Möbius Frobenius algebra $(V, \mu_V, \eta_V, \Delta_V, \epsilon_V, m_V)$ of any ribbon graph. From here, it is straightforward to determine the graded Euler characteristic by first calculating the graded dimension of $V$. As before in \autoref{E:Grading}, we use the grading variables $q$ and $b$.

\begin{Proposition}\label{qdimhom}
The graded dimension of $V$ is given by:
\begin{gather*}
q\text{dim}(V) =\begin{cases*}
2(2 + 2q^2+ \ldots + 2q^{n-2} + qb + \ldots + q^{n-1}b) & $n$ even,\\
3(1 + q + \ldots + q^{n-1})& $n$ odd.
\end{cases*}
\end{gather*}
\end{Proposition}

\begin{proof}
For $n$ even, $\mathcal{G} = \Z/n\Z \times \Z/2\Z$ with two associated generators, and so the graded dimension is a polynomial in two variables, say $q, b$. As the integer $i$ runs from $0$ to $n-1$, it follows that $2i \,\text{mod}\,n$ coming from the grading of $x^i$, as well as $2i+2 \,\text{mod}\,n $ coming from the grading of $x^iy^2$, hits all the even integers from $0$ to $n-2$ twice. On the other hand, $2i + 1\,\text{mod}\,n$, coming from the grading of $x^iy$, hits all the odd integers from $1$ to $n-1$ twice. Therefore, $W_{(2i,0)}$ are four dimensional while $W_{(2i+1,1)}$ are two dimensional given $i = 0, \ldots, n/2-1$.

For $n$ odd, $\mathcal{G} = \Z/n\Z $ with one associated generator, and so the graded dimension is a polynomial in one variable, say $q$. As the integer $i$ runs from $0$ to $n-1$, it is straightforward to see that $2i \,\text{mod}\,n$ coming from the grading of $x^i$, $2i+2 \,\text{mod}\,n $ coming from the grading of $x^iy^2$, and $2i + 1\,\text{mod}\,n$ coming from the grading of $x^iy$, each hit all integers from $0$ to $n-1$. Therefore, each $W_{g}$, $g \in \Z/n\Z$, is three dimensional.
\end{proof}

\begin{Example}
Following from \autoref{E:mobn1} and \autoref{gradedexmp}, we know that for $n=1$ the grading is trivial and there are three basis vectors, and so $q\text{dim}(V)=3$, as expected. For $n=2$, from \autoref{gradedexmp} we know that there are four basis vectors with grading $(0,0)$ and the remaining two basis vectors have grading $(1,1)$, giving $q\text{dim}(V) = 4+2qb$, also as expected.
\end{Example}

Given a perfect matching graph $\Gamma_M$ with trivalent underlying graph $G$, which is trivalent with perfect matching $M$, we can then determine the graded Euler characteristic using \autoref{T:maintheorem} together with \autoref{qdimhom}: we simply calculate the bracket $\langle \Gamma_M \rangle_n$ of $\Gamma_M$ characterized by the following skein relations:
\begin{align}
\Bigg\langle\begin{tikzpicture}[anchorbase,yscale=-1,scale=0.75]
\draw[usual] (0,0) to (0.5,0.25);
\draw[usual] (0.5,0.25) to (1,0);
\draw[match] (0.5,0.25) to (0.5,1);
\draw[usual] (0,1.25) to (0.5,1);
\draw[usual] (1,1.25) to (0.5, 1);
\node at (0.5,0.25)[circle,fill,inner sep=1.55pt]{};
\node at (0.5,1)[circle,fill,inner sep=1.55pt]{};
\end{tikzpicture} \Bigg\rangle_n
\quad &= \quad
\, \Bigg\langle\begin{tikzpicture}[anchorbase,yscale=-1,scale=0.75]
\draw[usual] (0,0) to (0,1);
\draw[usual] (0.5,0) to (0.5,1);
\end{tikzpicture} \Bigg\rangle_n \quad - \quad  \,\Bigg\langle\begin{tikzpicture}[anchorbase,yscale=-1,scale=0.75]
\draw[usual] (0,0) to (0.5,1);
\draw[usual] (0.5,0) to (0,1);
\end{tikzpicture} \Bigg\rangle_n, \label{skein12} \\
\Bigg\langle\begin{tikzpicture}[anchorbase,yscale=-1,scale=0.75]
\filldraw[color=black, fill=white, thick](-1,0) circle (0.5);
\end{tikzpicture}\Bigg\rangle_n \quad &= \quad \begin{cases*}
2(2 + 2q^2+ \ldots + 2q^{n-2} + qb + \ldots + q^{n-1}b) & $n$ even,\\
3(1 + q + \ldots + q^{n-1})& $n$ odd,
\end{cases*} \label{skein22} \\
\langle \Gamma_1 \sqcup \Gamma_2 \rangle_n \quad &= \quad \langle \Gamma_1 \rangle_n \cdot \langle \Gamma_2 \rangle_n. \label{skein32}
\end{align}
We think of \autoref{skein12}--\autoref{skein32} as a refinement of 
\autoref{skein1}--\autoref{skein3}.

We can also write down the associated graded Poincaré polynomial, defined in \autoref{poin}, and we will see explicit examples of this in \autoref{S:Exmp}. 

\begin{Remark}\label{R:poin2eul}
Note that setting $t=-1$ in the graded Poincaré polynomial also gives the graded Euler characteristic (see the proof of the second part of \autoref{T:maintheorem}).
\end{Remark}

\subsection{Penrose polynomial} Suppose that we are given a ribbon graph $\Gamma$ and associated blowup $\Gamma^\flat$. Now use the canonical perfect matching of $\Gamma^\flat$, with the edges of the underlying graph of $\Gamma$ as the perfect matching, to form
the perfect matching graph $\Gamma^\flat_E$. The Penrose polynomial $P(\Gamma, n)$ of $\Gamma$ is then the bracket $\langle \Gamma^\flat_E \rangle$ from \autoref{D:mobpoly} with $A = 1$, $B =-1$, and $C = n$. Using \autoref{brackforribb}, \autoref{P:bracketinv} and the rest of the discussion in \autoref{S:polyinv}, we can conclude that $P(\Gamma, n)$ is indeed a well-defined ribbon graph invariant.

Let $\langle \Gamma_E^\flat \rangle_n$ denote the bracket associated with the Möbius homology characterized by \autoref{skein12}--\autoref{skein32} above. We have the following proposition.

\begin{Proposition}\label{pencorr}
Recall that as a polynomial, $\langle \Gamma_E^\flat \rangle_n = \langle \Gamma_E^\flat\rangle_n(q, b)$ for $n$ even and for $n$ odd, $\langle \Gamma_E^\flat \rangle_n = \langle \Gamma_E^\flat \rangle_n(q)$. We have that:
\begin{gather*}
\langle \Gamma_E^\flat\rangle_n(q, b) =\begin{cases*}
P(\Gamma, 3n) & when $q=b=1$,\\
P(\Gamma, n) & when $q =1$, $b=-1$,
\end{cases*}
\end{gather*}
for $n$ even, and $$\langle \Gamma_E^\flat \rangle_n(1) = P(\Gamma, 3n)$$ for $n$ odd.
\end{Proposition}

\begin{proof}
In \autoref{skein22}, for $n$ even, if we set $q=b=1$, 
\begin{gather}\label{topen}
2(2 + 2q^2+ \ldots + 2q^{n-2} + qb + \ldots + q^{n-1}b)  = \begin{cases*}
2(2n/2 + n/2)=3n & $q=b=1$,\\
2(2n/2 - n/2)=n & $q=1$, $b=-1$,
\end{cases*} 
\end{gather}
and since $n$ is even, $(qb)^n=1$ and $b$ only appears with odd powers of $q$ i.e. setting $q=1$, $b=-1$ before or after taking powers of the l.h.s. of \autoref{topen} does not affect the computation; of course this is also trivially true for $q=b=1$.

For $n$ odd, setting $q =1$ gives
\begin{gather}\label{topen2}
3(1 + q + \ldots + q^{n-1}) = 3n,
\end{gather}
and setting $q=1$ before or after taking powers of the l.h.s. of \autoref{topen2} again does not affect the computation.

Altogether, we have shown that the skein relations characterizing the Penrose polynomial have been recovered for the respective cases.
\end{proof}

We conclude this section by listing some of the explicit ribbon graph invariants that can be obtained using the Penrose polynomial; for more information and proofs, see \cite{Aig-penrose,ElMo-penrose,Pen-negative}.

For the M\"obius Frobenius algebra from \autoref{S:ncolorchoice}, write
\[
MH_n^{i,g}(\Gamma_M;R)
\text{ and }
Mh_n(\Gamma_M)
\]
for the corresponding M\"obius homology and graded Poincar\'e polynomial, respectively. For $n$ even, write $g=(a,\epsilon)\in\Z/n\Z\times\Z/2\Z$.

\begin{Definition}
For $n$ even, define the \emph{positive and negative parts of the M\"obius homology} and their $q$-graded ranks by
\begin{gather*}
MH_n^{+}(\Gamma_M;R)
=
\bigoplus_{\substack{i,a,\epsilon\\ i+\epsilon\equiv 0 \bmod 2}}
MH_n^{i,(a,\epsilon)}(\Gamma_M;R),\quad
Q_n^{+}(\Gamma_M)
=q\text{dim}MH_n^{+}(\Gamma_M;R)|_{b=1},
\\
MH_n^{-}(\Gamma_M;R)
=
\bigoplus_{\substack{i,a,\epsilon\\ i+\epsilon\equiv 1 \bmod 2}}
MH_n^{i,(a,\epsilon)}(\Gamma_M;R)\quad
Q_n^{-}(\Gamma_M)
=q\text{dim}MH_n^{-}(\Gamma_M;R)|_{b=1},
\end{gather*}
where we set $b=1$. Moreover, for $f,g\in\Z[q]/(q^4-1)$, write $f>_{\mathrm{cw}}g$ if, after representing $f-g$ as
$c_0+c_1q+c_2q^2+c_3q^3$, all coefficients $c_i$ are nonnegative and at least one is positive.
\end{Definition}

For simplicity of notation, in the final point of \autoref{penfac}, the numbers $n'$ and $x$ are determined by \autoref{pencorr}.

\begin{Theorem}\label{penfac}
Let $\Gamma$ be a ribbon graph whose underlying graph $G=(V,E,r)$ is connected and planar. We have the following:
\begin{enumerate}
\item $\langle\Gamma^\flat_E\rangle_2(1,-1)=2^{|V|}$ if $\Gamma$ is Eulerian and zero otherwise.
\item $\langle\Gamma^\flat_E\rangle_1(1)=\#\{3\text{-edge colorings of $\Gamma$}\}$ when $\Gamma$ is trivalent.
\item $\Gamma$ is $4$-face colorable if and only if
\[
\big(Q_4^{+}(\Gamma^\flat_E)-Q_4^{-}(\Gamma^\flat_E)\big)\big|_{q=1}>0.
\]
In particular, the $q$-graded homological condition
\[
Q_4^{+}(\Gamma^\flat_E)>_{\mathrm{cw}} Q_4^{-}(\Gamma^\flat_E)
\]
implies that $\Gamma$ is $4$-face colorable. Thus proving this coefficientwise positivity condition for every connected planar ribbon graph $\Gamma$ would imply the four color theorem.
\item Given a nonnegative integer $n$ with $n\in2\Z$ or $n\in3\Z$, $\langle\Gamma^\flat_E\rangle_n(x)\geq\chi(\Gamma^*,n')$, where $\chi(\Gamma^*,n')$ is the chromatic polynomial of the geometric dual of $\Gamma$.
\end{enumerate}
\end{Theorem}

\begin{proof}
The polynomial statements follow from \autoref{pencorr} and certain facts about the Penrose polynomial, cf. \cite{Aig-penrose,ElMo-penrose,Pen-negative}. For part~(c), \autoref{T:maintheorem} gives
\[
Q_4^{+}(\Gamma^\flat_E)-Q_4^{-}(\Gamma^\flat_E)
=
Mh_4(\Gamma^\flat_E)\big|_{t=-1,b=-1}.
\]
Therefore,
\[
\big(Q_4^{+}(\Gamma^\flat_E)-Q_4^{-}(\Gamma^\flat_E)\big)\big|_{q=1}
=
Mh_4(\Gamma^\flat_E)\big|_{t=-1,q=1,b=-1}
=
\langle\Gamma^\flat_E\rangle_4(1,-1).
\]
By \autoref{pencorr}, the final expression is $P(\Gamma,4)$, and the Penrose polynomial statement says that $P(\Gamma,4)>0$ if and only if $\Gamma$ is $4$-face colorable. This proves the equivalence.
Finally, if $Q_4^{+}(\Gamma^\flat_E)>_{\mathrm{cw}}Q_4^{-}(\Gamma^\flat_E)$, then evaluating at $q=1$ gives a positive integer, so $\Gamma$ is $4$-face colorable.
\end{proof}

\begin{Remark}
There are other results for $n$ negative, however these ranges of $n$ cannot be reached in our construction so we omit them. In the case $n =0$, though not explicitly included, we can simply define $\langle \Gamma_E^\flat\rangle_0(1, -1)=0=P(\Gamma, 0)$.
\end{Remark}

\begin{Remark}
We point out that \autoref{penfac}.(c) is in the same general
spirit as several approaches to the four color theorem via graph and web
homologies. 

Firstly, Baldridge~\cite{Bal-cohomology} constructed a cohomology theory
for planar trivalent graphs with perfect matchings whose graded Euler
characteristic detects even perfect matchings, and therefore gives a homological
route to four-face colorability. Baldridge--McCarty~\cite{BaMc-graphcoloring}
later constructed TQFT-type homologies whose dimensions encode face colorings
of cell decompositions of surfaces, leading to a constructive homological
approach to four-face coloring.

There is also a closely related web-and-foam viewpoint.
Kronheimer--Mrowka~\cite{KrMr-tait} introduced an instanton homology for webs
and foams, with a nonvanishing theorem motivated by a possible new proof of
the four color theorem. Khovanov--Robert~\cite{KhRo-foam} then constructed
combinatorial foam-evaluation analogs of these Kronheimer--Mrowka theories
for planar trivalent graphs.
\end{Remark}


\section{Examples and computer talk}\label{S:Exmp}


We now give a few calculation examples.

\subsection{By hand calculations}

We begin this section by computing the graded Poincaré polynomial associated with the Möbius homology outlined in the previous sections for three explicit examples of perfect matching graphs: $E_1$, $E_2$, and $E_3$ shown in \autoref{E:permatexm}. We will concentrate on the cases $n=1,2$, the simplest examples of $n$ odd and even respectively. 

\begin{Example}\label{E:e1}
The hypercube of states corresponding to the perfect matching graph 
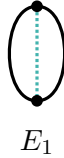
\begin{gather*}
\begin{tikzpicture}[anchorbase,yscale=-1,scale=1.25]
\draw[usual] (0,1) to [ out =180 , in = 180](0,0);
\draw[usual] (0,0) to [ out =0 , in = 0](0,1);
\draw[match] (0,0) to (0,1);
\node at (0,1)[circle,fill,inner sep=1.55pt]{};
\node at (0,0)[circle,fill,inner sep=1.55pt]{};
\node[below=5pt] at (current bounding box.south) {$E_1$};
\end{tikzpicture} 
\end{gather*}
is given by 
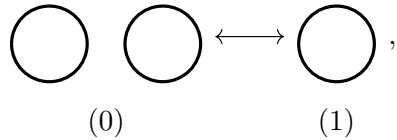
\begin{equation*}
\begin{tikzcd}
\begin{tikzpicture}
\draw[usual] (0,0) circle (0.5cm);
\draw[usual] (1.5,0) circle (0.5cm);
\node[below=5pt] at (current bounding box.south) {$(0)$};
\end{tikzpicture}
\arrow[r, leftrightarrow] 
&  \begin{tikzpicture}
\draw[usual] (0,0) circle (0.5cm);
\node[below=5pt] at (current bounding box.south) {$(1)$};
\end{tikzpicture}
\end{tikzcd},
\end{equation*}
and so we can write our complex generically as 
\begin{equation*}
\begin{tikzcd}
V \otimes V
\arrow[r, "\mu_V"] 
& V  \\
i = 0 & i =1 
\end{tikzcd}
,
\end{equation*}
with $V = R[x,y]/(x^n, y^3 -xy)$ and $\mu_V$ given by multiplication on the polynomial ring $V$ as discussed in the previous section.

\begin{table}
\renewcommand{\arraystretch}{1}
\centering
\begin{tabular}{| c || c | w{c}{.75cm} | }
\hline 
$0$ & \makecell{$<1 \otimes y - y\otimes 1,$ \\$y \otimes y^2, y^2 \otimes y, y^2 \otimes y^2$, \\ $y \otimes y - 1 \otimes y^2,$ \\$y^2 \otimes 1 - 1 \otimes y^2>$}  &   0  \\ \hline \hline
\diagbox[dir=SW]{$j$}{$i$} & $0$& $1$  \\ \hline
\end{tabular} 
\caption{Möbius homology for $E_1$, $n=1$.}
\label{table:homology}
\end{table}
\begin{table}
\renewcommand{\arraystretch}{1}
\centering
\begin{tabular}{| c || c | w{c}{.75cm} | }
\hline
$(1,1)$ & \makecell{$<x \otimes xy, xy \otimes x, 1 \otimes xy - xy \otimes 1,$ \\$x \otimes y - y \otimes x,$ \\ $xy \otimes xy^2, xy^2 \otimes xy, y^2 \otimes xy,$ \\$xy \otimes y^2,  1\otimes y - y \otimes 1, $ \\ $y \otimes xy^2, xy^2 \otimes y, y \otimes y^2 -y^2 \otimes y,$ \\ $x \otimes y - y\otimes y^2, x \otimes y - 1 \otimes xy >$} &   0  \\ \hline
$(0,0)$ & \makecell{$<xy \otimes xy, xy^2 \otimes xy^2, x \otimes x y^2, xy^2 \otimes x$, \\ $x \otimes x, xy^2 \otimes y^2, y^2 \otimes xy^2, 1\otimes x - x \otimes 1, $ \\ $1 \otimes y^2 - y^2 \otimes 1, 1 \otimes xy^2 - xy^2 \otimes 1,$ \\$x \otimes y^2 -1 \otimes xy^2,$ \\ $y^2\otimes x - 1 \otimes xy^2, y \otimes xy - 1\otimes xy^2,$ \\ $xy \otimes y - 1 \otimes xy^2,$ \\ $y \otimes y - 1 \otimes y^2, y^2 \otimes y^2 - 1 \otimes xy^2>$} &    0  \\ \hline \hline
\diagbox[dir=SW]{$j$}{$i$} & $0$& $1$  \\ \hline
\end{tabular} 
\caption{Möbius homology for $E_1$, $n=2$.}
\label{table:homology2}
\end{table} 

For $n=1$, the grading is trivial, and since $\text{rank}(\mu_V)=\text{dim}(V) = 3$, $\text{null}(\mu_V) = 9-3=6$. The resulting homology is show in \autoref{table:homology}; therefore, we have that $Mh(E_1) = 6$. 

For $n=2$, we have $\mathcal{G} = \Z/2\Z \times \Z/2\Z$ and all elements of $V$ are spread over the gradings $(0,0)$ and $(1,1)$. Since $\text{rank}(\mu_V)=\text{dim}(V) = 6$, $\text{null}(\mu_V) = 36-6=30$ spread over gradings $(0,0)$ and $(1,1)$. Enumerating over all possibilities gives the homology shown in \autoref{table:homology2}, and so we have that $Mh(E_1) = 16+14qb$. Furthermore, notice that $E_1$ is the blowup of a ribbon graph with one vertex and one edge with the canonical perfect matching; using \autoref{pencorr} and \autoref{penfac}, we see that $16+14(1)(-1)=2=2^1$, as expected.
\end{Example}

\begin{Example}\label{E:e2}
The hypercube of states corresponding to the perfect matching graph
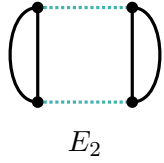
\begin{gather*}
\begin{tikzpicture}[anchorbase,yscale=-1,scale=1.25]
\draw[usual] (0,1) to [ out =180 , in = 180](0,0);
\draw[usual] (1,1) to [ out =0 , in = 0](1,0);
\draw[match] (0,0) to (1,0);
\draw[match] (0,1) to (1,1);
\draw[usual] (1,1) to (1,0);
\draw[usual] (0,0) to (0,1);
\node at (0,1)[circle,fill,inner sep=1.55pt]{};
\node at (0,0)[circle,fill,inner sep=1.55pt]{};
\node at (1,0)[circle,fill,inner sep=1.55pt]{};
\node at (1,1)[circle,fill,inner sep=1.55pt]{};
\node[below=5pt] at (current bounding box.south) {$E_2$};
\end{tikzpicture}
\end{gather*}
is given by
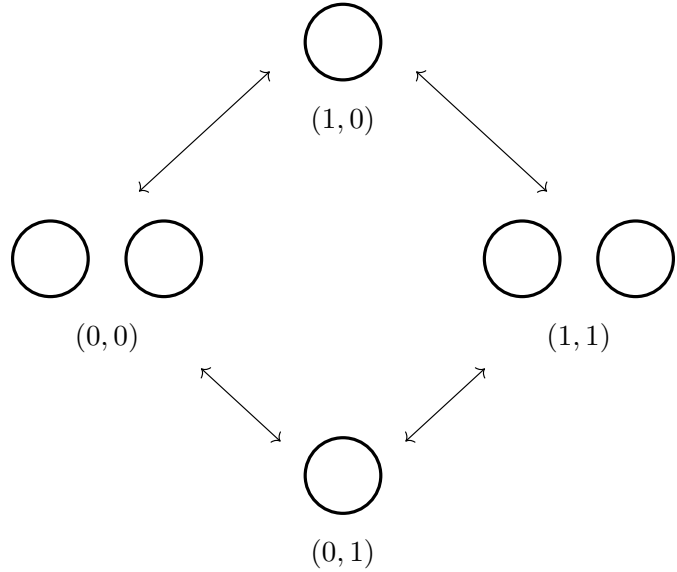
\begin{equation*}
\begin{tikzcd}
& \begin{tikzpicture}
\draw[usual] (0,0) circle (0.5cm);
\node[below=5pt] at (current bounding box.south) {$(1,0)$};
\end{tikzpicture} 
\arrow[dr, leftrightarrow, shift left=1.8ex] & \\
\begin{tikzpicture}
\draw[usual] (0,0) circle (0.5cm);
\draw[usual] (1.5,0) circle (0.5cm);
\node[below=5pt] at (current bounding box.south) {$(0,0)$};
\end{tikzpicture}
\arrow[ur, leftrightarrow, shift left=1.8ex] \arrow[dr, leftrightarrow, shift right=1.8ex]
& & \begin{tikzpicture}
\draw[usual] (0,0) circle (0.5cm);
\draw[usual] (1.5,0) circle (0.5cm);
\node[below=5pt] at (current bounding box.south) {$(1,1)$};
\end{tikzpicture} \\
& \begin{tikzpicture}
\draw[usual] (0,0) circle (0.5cm);
\node[below=5pt] at (current bounding box.south) {$(0,1)$};
\end{tikzpicture} 
\arrow[ur, leftrightarrow,, shift right=1.8ex]&
\end{tikzcd}
\end{equation*}
with corresponding complex 
\begin{equation*}
\begin{tikzcd}
& V
\arrow[dr, "-\Delta_V"] & \\
V \otimes V
\arrow[ur, "\mu_V"] \arrow[dr, "\mu_V"']
& \oplus & V \otimes V  \\
& V
\arrow[ur, "\Delta_V"']& \\
i = 0 & i =1 & i=2
\end{tikzcd}
\end{equation*}
as shown in \autoref{E:hypercube} and \autoref{E:comp}, except that we have included signs for the maps in the complex as required in the definition of the differential. Aside from the multiplication map $\mu_V$, we see the comultiplication $\Delta_V$, which is defined by \autoref{comultdef}.

\begin{table}
\renewcommand{\arraystretch}{1}
\centering
\begin{tabular}{| c || c | w{c}{.75cm} | c | }
\hline
$0$ & \makecell{$<y \otimes y^2, y^2 \otimes y, y^2 \otimes y^2$, \\ $1 \otimes y - y\otimes 1, y \otimes y - 1 \otimes y^2,$ \\$y^2 \otimes 1 - 1 \otimes y^2>$} &   0  & \makecell{$<1 \otimes 1, 1 \otimes y^2, y \otimes y$, \\ $y^2 \otimes 1, y\otimes y^2,$ \\$y^2 \otimes y>$} \\ \hline \hline
\diagbox[dir=SW]{$j$}{$i$} & $0$& $1$ & $2$ \\ \hline
\end{tabular} 
\caption{Möbius homology for $E_2$, $n=1$.}
\label{table:homology3}
\end{table}
\begin{table}
\renewcommand{\arraystretch}{1}
\centering
\begin{tabular}{| c || c | w{c}{.75cm} |c| }
\hline
$(1,1)$ & \makecell{$<x \otimes xy, xy \otimes x, 1 \otimes xy - xy \otimes 1, x \otimes y - y \otimes x,$ \\ $xy \otimes xy^2, xy^2 \otimes xy, y^2 \otimes xy, xy \otimes y^2,  1\otimes y - y \otimes 1, $ \\ $y \otimes xy^2, xy^2 \otimes y, y \otimes y^2 -y^2 \otimes y,$ \\ $x \otimes y - y\otimes y^2, x \otimes y - 1 \otimes xy >$} &   0 & \makecell{$<1 \otimes y,1 \otimes xy, x \otimes y,$ \\ $x \otimes xy, y \otimes 1, y \otimes x, $ \\ $y \otimes y^2, y^2 \otimes y, xy \otimes 1,$ \\ $xy \otimes x, y \otimes xy^2, y^2 \otimes xy$ \\ $  xy\otimes y^2, xy \otimes xy^2>$}  \\ \hline
$(0,0)$ & \makecell{$<xy \otimes xy, xy^2 \otimes xy^2, x \otimes x y^2, xy^2 \otimes x$, \\ $x \otimes x, xy^2 \otimes y^2, y^2 \otimes xy^2, 1\otimes x - x \otimes 1, $ \\ $1 \otimes y^2 - y^2 \otimes 1, 1 \otimes xy^2 - xy^2 \otimes 1, x \otimes y^2 -1 \otimes xy^2,$ \\ $y^2\otimes x - 1 \otimes xy^2, y \otimes xy - 1\otimes xy^2, xy \otimes y - 1 \otimes xy^2,$ \\ $y \otimes y - 1 \otimes y^2, y^2 \otimes y^2 - 1 \otimes xy^2>$} &    0  & \makecell{$<1 \otimes 1, 1 \otimes x, 1 \otimes y^2,$ \\ $x \otimes 1, y \otimes y, y^2 \otimes 1 $ \\ $y^2 \otimes y^2 , 1 \otimes xy^2, x \otimes x$ \\ $x \otimes y^2, y \otimes xy, y^2 \otimes x, $ \\ $xy \otimes y, x \otimes xy^2$ \\$ xy \otimes xy, y^2 \otimes xy^2>$} \\ \hline \hline
\diagbox[dir=SW]{$j$}{$i$} & $0$& $1$ & 2  \\ \hline
\end{tabular} 
\caption{Möbius homology for $E_2$, $n=2$.}
\label{table:homology4}
\end{table} 

For $n=1$, in particular for comultiplication, the explicit maps are given in \autoref{E:mobn1}. The $i=0$ homology is already given in \autoref{table:homology}. Given $\partial^1(v, v'): V \oplus V \rightarrow V \otimes V$, $(v, v') \mapsto -\Delta_V(v)+\Delta_V(v')$ and since $\text{rank}(\Delta_V)=3$, $\text{null}(\partial^1)=3$ and so the $i=1$ homology is $0$. For $i=2$, we have $\text{rank}(\partial^1)=6-3=3$ and since the nullity of the final zero map is $3\cdot 3 = 9$, the dimension of the $i=2$ homology is $9-3=6$ with basis given in \autoref{table:homology3}. Therefore, $Mh(E_2) = 6+6t^2$.

For $n=2$, the explicit maps (evaluated on $1$) are given in \autoref{n2}. The $i=0$ homology is again already given in \autoref{table:homology2}. By a similar reasoning to the $n=1$ case, with now $\text{rank}(\Delta_V)=6$, the $i=1$ homology is $0$, and the dimension of the $i=2$ homology is $6\cdot 6 -6=36-6=30$ with basis, spread over $(0,0)$ and $(1,1)$, given in \autoref{table:homology4}. Therefore, $Mh(E_2) = 16 + 14qb + 16t^2+14qbt^2$. Furthermore, notice that $E_2$ is the blowup of a ribbon graph with two vertices and two edges between the vertices, with the canonical perfect matching; using \autoref{pencorr} and \autoref{penfac}, we see that $16+14(1)(-1) +16(-1)^2+14(1)(-1)(-1)^2=4=2^2$, as expected.
\end{Example}

\begin{Example}\label{E:e3}
The hypercube of states corresponding to the perfect matching graph
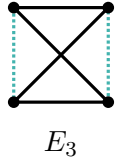
\begin{gather*}
\begin{tikzpicture}[anchorbase,yscale=-1,scale=1.25]
\draw[usual] (0,0) to (1,0);
\draw[usual] (0,1) to (1,1);
\draw[match] (1,1) to (1,0);
\draw[match] (0,0) to (0,1);
\draw[usual] (0,0) to (1,1);
\draw[usual] (0,1) to (1,0);
\node at (0,1)[circle,fill,inner sep=1.55pt]{};
\node at (0,0)[circle,fill,inner sep=1.55pt]{};
\node at (1,0)[circle,fill,inner sep=1.55pt]{};
\node at (1,1)[circle,fill,inner sep=1.55pt]{};
\node[below=5pt] at (current bounding box.south) {$E_3$};
\end{tikzpicture}
\end{gather*}
is given by
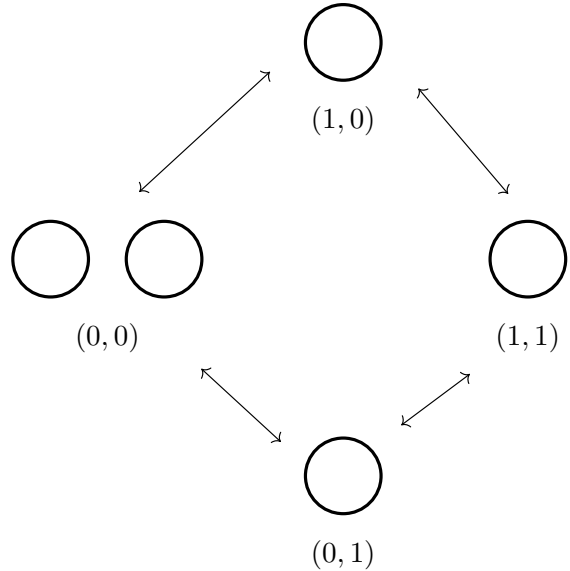
\begin{equation*}
\begin{tikzcd}
& \begin{tikzpicture}
\draw[usual] (0,0) circle (0.5cm);
\node[below=5pt] at (current bounding box.south) {$(1,0)$};
\end{tikzpicture} 
\arrow[dr, leftrightarrow, shift left=1.8ex] & \\
\begin{tikzpicture}
\draw[usual] (0,0) circle (0.5cm);
\draw[usual] (1.5,0) circle (0.5cm);
\node[below=5pt] at (current bounding box.south) {$(0,0)$};
\end{tikzpicture}
\arrow[ur, leftrightarrow, shift left=1.8ex] \arrow[dr, leftrightarrow, shift right=1.8ex]
& & \begin{tikzpicture}
\draw[usual] (0,0) circle (0.5cm);
\node[below=5pt] at (current bounding box.south) {$(1,1)$};
\end{tikzpicture} \\
& \begin{tikzpicture}
\draw[usual] (0,0) circle (0.5cm);
\node[below=5pt] at (current bounding box.south) {$(0,1)$};
\end{tikzpicture} 
\arrow[ur, leftrightarrow, start anchor =north east, end anchor=south west]&
\end{tikzcd}
\end{equation*}
with corresponding complex 
\begin{equation*}
\begin{tikzcd}
& V
\arrow[dr, "-m_V"] & \\
V \otimes V
\arrow[ur, "\mu_V"] \arrow[dr, "\mu_V"']
& \oplus & V  \\
& V
\arrow[ur, "m_V"']& \\
i = 0 & i =1 & i=2
\end{tikzcd}
\end{equation*}
Aside from the multiplication map $\mu_V$, the map $m_V$ is defined in \autoref{mdef}.

\begin{table}
\renewcommand{\arraystretch}{1}
\centering
\begin{tabular}{| c || c | c | c | }
\hline
$0$ & \makecell{$<y \otimes y^2, y^2 \otimes y, y^2 \otimes y^2$, \\ $1 \otimes y - y\otimes 1, y \otimes y - 1 \otimes y^2,$ \\$y^2 \otimes 1 - 1 \otimes y^2>$} &   $<(y^2,0)>$  & \makecell{$<1 >$} \\ \hline \hline
\diagbox[dir=SW]{$j$}{$i$} & $0$& $1$ & $2$ \\ \hline
\end{tabular} 
\caption{Möbius homology for $E_3$, $n=1$.}
\label{table:homology5}
\end{table}
\begin{table}
\renewcommand{\arraystretch}{1}
\centering
\begin{tabular}{| c || c | c |c| }
\hline
$(1,1)$ & \makecell{$<x \otimes xy, xy \otimes x, 1 \otimes xy - xy \otimes 1, x \otimes y - y \otimes x,$ \\ $xy \otimes xy^2, xy^2 \otimes xy, y^2 \otimes xy, xy \otimes y^2,  1\otimes y - y \otimes 1, $ \\ $y \otimes xy^2, xy^2 \otimes y, y \otimes y^2 -y^2 \otimes y,$ \\ $x \otimes y - y\otimes y^2, x \otimes y - 1 \otimes xy >$} &   \makecell{$<(xy,0)>$}  & \makecell{$<y>$}  \\ \hline
$(0,0)$ & \makecell{$<xy \otimes xy, xy^2 \otimes xy^2, x \otimes x y^2, xy^2 \otimes x$, \\ $x \otimes x, xy^2 \otimes y^2, y^2 \otimes xy^2, 1\otimes x - x \otimes 1, $ \\ $1 \otimes y^2 - y^2 \otimes 1, 1 \otimes xy^2 - xy^2 \otimes 1, x \otimes y^2 -1 \otimes xy^2,$ \\ $y^2\otimes x - 1 \otimes xy^2, y \otimes xy - 1\otimes xy^2, xy \otimes y - 1 \otimes xy^2,$ \\ $y \otimes y - 1 \otimes y^2, y^2 \otimes y^2 - 1 \otimes xy^2>$} &    \makecell{$<(xy^2,0), (x, y^2)>$}   & \makecell{$<1, x>$} \\ \hline \hline
\diagbox[dir=SW]{$j$}{$i$} & $0$& $1$ & $2$  \\ \hline
\end{tabular} 
\caption{Möbius homology for $E_3$, $n=2$.}
\label{table:homology6}
\end{table} 

For $n=1$, the map $m_V$ is given in \autoref{E:mobn1}. The $i=0$ homology is again already given in \autoref{table:homology}. Given $\partial^1(v, v'): V \oplus V \rightarrow V $, $(v, v') \mapsto -m_V(v)+m_V(v')$ and since $\text{rank}(m_V)=2$, $\text{null}(\partial^1)=6-2=4$ and so the $i=1$ homology has dimension $4-3=1$ with basis given in \autoref{table:homology5}. For $i=2$, we have $\text{rank}(\partial^1)=2$ and since the nullity of the final zero map is $3$, the dimension of the $i=2$ homology is $3-2=1$ with basis given in \autoref{table:homology5}. Therefore, $Mh(E_3) = 6+t+t^2$. Notice that the graded Euler characteristic is $6$ (see \autoref{R:poin2eul}), which is the same as $E_1$, though they are clearly not isomorphic; the Möbius homology is therefore strictly more powerful than the bracket characterized by \autoref{skein12}--\autoref{skein32} for $n=1$.

For $n=2$, the map $m_V$ (evaluated on $1$) is given in \autoref{n2}. The $i=0$ homology is once again given in \autoref{table:homology2}. Again by a similar reasoning to the $n=1$ case, with now $\text{rank}(m_V)=3$, $\text{null}(\partial^1)=12-3=9$ and so the dimension of the $i=1$ homology is $9-6=3$ with basis given in \autoref{table:homology6}. For $i=2$, we have $\text{rank}(\partial^1)=3$ and since the nullity of the final zero map is $6$, the dimension of the $i=2$ homology is $6-3=3$ with basis given in \autoref{table:homology6}. Therefore, $Mh(E_3) = 16 + 14qb + 2t+ qbt + 2t^2 + qbt^2$, and just like the $n=1$ case, we see that the Möbius homology is strictly more powerful than the bracket for $n=2$.
\end{Example}

In fact, we can use \autoref{E:e1} and \autoref{E:e3} above to prove the following.

\begin{Proposition}
The Möbius homology associated with the Möbius Frobenius algebras from \autoref{S:ncolorchoice} is strictly more powerful than the bracket characterized by \autoref{skein12}--\autoref{skein32} for all positive integers $n$.
\end{Proposition}

\begin{proof}
We have already noted that the perfect matching graphs $E_1$ and $E_3$ from \autoref{E:permatexm} are not isomorphic. The bracket characterized by \autoref{skein12}--\autoref{skein32}, or equivalently, the graded Euler characteristic associated with the Möbius homology, for $E_1$ is  $q\text{dim}(V)^2-q\text{dim}(V)$, which is the same as
\begin{gather*}
q\text{dim}(V)^2-2q\text{dim}(V) + q\text{dim}(V) = q\text{dim}(V)^2-q\text{dim}(V)
\end{gather*}
for $E_2$, where $V= R[x,y]/(x^n, y^3-xy)$ is the Möbius Frobenius algebra from \autoref{S:ncolorchoice} with $n$ a positive integer.

On the other hand, since the multiplication map $\mu_V$ is obviously full rank, the homology group of homological degree $i=1$ for $E_1$ is $0$ for all positive integers $n$ i.e. the graded Poincaré polynomial for $E_1$ given by \autoref{poin} contains no nontrivial powers of $t$. For $E_3$, however, the degree $i=2$ homology group contains at least $<1>$ for all positive integers $n$ since the map $m_V$, given in \autoref{mdef}, is not full rank; in particular, neither $y$ nor $y^2$ is invertible in $V$. Therefore, the graded Poincaré polynomial for $E_3$ contains a nonzero coefficient of $t^2$, and the result is proved.
\end{proof}

\subsection{Computer talk}

For the remainder of this section, we briefly describe some Python code for producing the Poincaré polynomial as shown in \autoref{E:e1}, \autoref{E:e2}, and \autoref{E:e3} above, but now for general perfect matching graphs. For further details and code files, we direct the reader to \cite{CoTu-code}.

Our input is a perfect matching graph $\Gamma_M$, whose underlying graph $G=(V, E, r)$ is trivalent with perfect matching $M$. The vertices are encoded as points $\{0, 1, \ldots, |V|-1\}$, and perfect matching edges are specified as (unordered) pairs of points. To specify the remaining edges, for each vertex $i=0, 1, \ldots, |V|-1$ in order, associate an ordered pair of vertices $(v_i, w_i)$ for which there are edges $L_i$ and $R_i$, away from the perfect matching edge, with $L_i=\{i, v_i\}$ and $R_i=\{i, w_i\}$. By convention, the ordering $(v_i, w_i)$ is determined by forming an equivalent perfect matching graph where the perfect matchings locally have the following fixed configuration: 
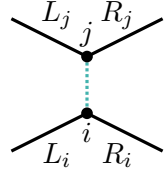
\begin{gather*}
\begin{tikzpicture}[anchorbase,yscale=-1,scale=1]
\draw[usual] (-0.5,0) to (0.5,0.5);
\draw[usual] (0.5,0.5) to (1.5,0);
\draw[match] (0.5,0.5) to (0.5,1.25);
\draw[usual] (-0.5,1.75) to (0.5,1.25);
\draw[usual] (1.5,1.75) to (0.5, 1.25);
\node at (0.5,1.25)[circle,fill,inner sep=1.55pt]{};
\node at (0.5,0.5)[circle,fill,inner sep=1.55pt]{};
\node at (0.5, 1.5){$i$};
\node at (0.1, 1.80){$L_i$};
\node at (0.9, 1.80){$R_i$};
\node at (0.5, 0.2){$j$};
\node at (0.1, -0.05){$L_j$};
\node at (0.9, -0.05){$R_j$};
\end{tikzpicture}
\end{gather*}
with $j > i$; such a configuration is always possible at every vertex since $G$ is trivalent and every vertex is adjacent to a perfect matching, by definition of a perfect matching. 

For example, we can input the perfect matching graph $E_3$ from \autoref{E:permatexm} as follows, together with the accompanying illustration:
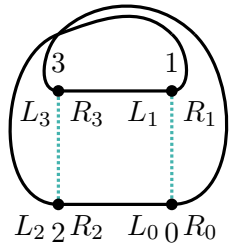
\begin{gather*}
\begin{tikzpicture}[anchorbase,yscale=-1,scale=1.5]
\draw[usual] (0,0) to (1,0);
\draw[usual] (0,1) to (1,1);
\draw[match] (1,1) to (1,0);
\draw[match] (0,0) to (0,1);
\draw[usual] (0,0) to [ out =-180 , in = 135] (0,-0.5);
\draw[usual] (0,-0.5) to [ out =-45 , in = -135] (1.2,-0.5);
\draw[usual] (1.2,-0.5) to [ out =45 , in = 0](1,1);
\draw[usual] (0,1) to [ out =180 , in = 160] (0,-0.5);
\draw[usual] (1,-0.5) to [ out =40 , in = 0](1,0);
\draw[usual] (0,-0.5) to [ out =-20 , in = -135](1,-0.5);
\node at (0,1)[circle,fill,inner sep=1.55pt]{};
\node at (0,0)[circle,fill,inner sep=1.55pt]{};
\node at (1,0)[circle,fill,inner sep=1.55pt]{};
\node at (1,1)[circle,fill,inner sep=1.55pt]{};
\node[below=5pt] at (current bounding box.south) {};
\node at (1,1.25) {$0$};
\node at (1,-0.25) {$1$};
\node at (0,1.25) {$2$};
\node at (0,-0.25) {$3$};
\node at (0.75,1.2) {$L_0$};
\node at (1.25,1.2) {$R_0$};
\node at (0.75,0.2) {$L_1$};
\node at (1.25,0.2) {$R_1$};
\node at (0.25, 1.2) {$R_2$};
\node at (-0.25, 1.2) {$L_2$};
\node at (0.25, 0.2) {$R_3$};
\node at (-0.2, 0.2) {$L_3$};
\end{tikzpicture}
\end{gather*}

\begin{lstlisting}[language=Python]
pairs = [[2, 3], [3, 2], [1, 0], [0, 1]]
matching = [[0, 1], [2, 3]]
\end{lstlisting}

The final input is the positive integer $n$, which fixes the choice of Möbius Frobenius algebra $V = R[x,y]/(x^n, y^3-xy)$ specified in \autoref{S:ncolorchoice}. For purposes of the code, we set $R = \mathbb{Q}$. For example, we could look at the simplest case $n=1$ i.e. $V=\mathbb{Q}[y]/(y^3)$:

\begin{lstlisting}[language=Python]
n = 1
\end{lstlisting}

We can now run the code, for example:

\begin{lstlisting}[language=Python]
MH = MobiusHomology(pairs, matching, n=n)
\end{lstlisting}

The essential functions performed are as follows:
\begin{enumerate}[label=\emph{\upshape(\roman*)}]
\item Construction of the hypercube of states returned as a list of lists giving the number of circles in each state of fixed $|\alpha|$, which is later interpreted as the homological degree. The code also keeps track of which circles change when moving along a directed edge in order to later assign the $\mu_V$, $\Delta_V$, and $m_V$ maps to the correct tensor power in the chain complex.

In our example, we can run
\begin{lstlisting}[language=Python]
print('circle counts by homological degree:', MH.state_circle_counts())
\end{lstlisting}
The output is
\begin{lstlisting}[language=Python]
circle counts by homological degree: [[2], [1, 1], [1]]
\end{lstlisting}
as expected.
\item Computation of the differential defined in \autoref{diffdef} at each homological degree, which can be viewed as a matrix if desired by running, for example at degree $0$:
\begin{lstlisting}[language=Python]
degree = 0
D = MH.differential_matrix(degree)
print('shape:', D.shape)
D
\end{lstlisting}
which outputs both the shape of the matrix and the matrix itself.
\item Extraction of the dimension of the homology groups at each homological degree by standard rank-nullity linear algebra. For the total dimensions, we can run:
\begin{lstlisting}[language=Python]
print('homology dimensions by homological degree:', MH.homology_dimensions())
\end{lstlisting}
which returns
\begin{lstlisting}[language=Python]
homology dimensions by homological degree: [6, 1, 1]
\end{lstlisting}
as expected from \autoref{table:homology5}. For the dimensions listed by grading, we can run:
\begin{lstlisting}[language=Python]
print('homology dimensions by internal degree:', MH.homology_dimensions_by_internal_degree())
\end{lstlisting}
which returns 
\begin{lstlisting}[language=Python]
homology dimensions by internal degree: {0: [6, 1, 1]}
\end{lstlisting}
also as expected; recall that for $n=1$, the grading is trivial.
\item Formation of the graded Poincaré polynomial by running:
\begin{lstlisting}[language=Python]
print('Poincare polynomial:', MH.poincare_polynomial())
\end{lstlisting}
which returns 
\begin{lstlisting}[language=Python]
Poincare polynomial: t**2 + t + 6
\end{lstlisting}
as expected for our example.
\item A final check that $\partial^2=0$ by running:
\begin{lstlisting}[language=Python]
print('d^2 = 0:', MH.check_d_squared())
\end{lstlisting}
which returns a Boolean 
\begin{lstlisting}[language=Python]
d^2 = 0: True
\end{lstlisting}
as required.
\end{enumerate}

\begin{Remark}
For $n$ even, we have chosen to ignore the $\Z/2\Z$ component in the grading $\mathcal{G}=(\Z/n\Z, \Z/2\Z)$ in the code for convenience: only a single $q$ variable for the grading appears in the graded Poincaré polynomial. Of course, by replacing $q^i$ with $q^ib$ for $i$ odd in the final result, the $\Z/2\Z$ component can be restored.
\end{Remark}

We conclude this section with a few more examples.

\begin{Example}
We run the code again on $E_3$, but for higher $n$. For $n=3$, the output is
\begin{lstlisting}[language=Python]
circle counts by homological degree: [[2], [1, 1], [1]]
d^2 = 0: True
homology dimensions by homological degree: [72, 5, 5]
homology dimensions by internal degree: {0: [24, 1, 1], 1: [24, 2, 2], 2: [24, 2, 2]}
Poincare polynomial: 2*q**2*t**2 + 2*q**2*t + 24*q**2 + 2*q*t**2 + 2*q*t + 24*q + t**2 + t + 24
\end{lstlisting}
and for $n=4$,
\begin{lstlisting}[language=Python]
circle counts by homological degree: [[2], [1, 1], [1]]
d^2 = 0: True
homology dimensions by homological degree: [132, 7, 7]
homology dimensions by internal degree: {0: [36, 2, 2], 1: [30, 1, 1], 2: [36, 3, 3], 3: [30, 1, 1]}
Poincare polynomial: q**3*t**2 + q**3*t + 30*q**3 + 3*q**2*t**2 + 3*q**2*t + 36*q**2 + q*t**2 + q*t + 30*q + 2*t**2 + 2*t + 36
\end{lstlisting}
Notice that the total homology dimensions appear to follow the pattern $(3n)^2-3n$, $2n-1$, and $2n-1$ for $i=0$, $1$, and $2$ respectively. In general, the result holds by standard rank-nullity analysis for the maps $\mu_V$ and $m_V$, which we will not do here.
\end{Example}

\begin{Example}
Consider the following perfect matching graph:
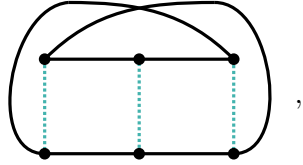
\begin{gather*}
  \begin{tikzpicture}[anchorbase,yscale=-1,scale=1.25]
\draw[usual] (0,0) to (1,0);
\draw[usual] (1,0) to (2,0);
\draw[usual] (0,1) to (1,1);
\draw[usual] (1,1) to (2,1);
\draw[match] (1,1) to (1,0);
\draw[match] (0,0) to (0,1);
\draw[match] (2,0) to (2,1);
\draw[usual] (0,0) to [ out =-45 , in = 180] (1.8,-0.6);
\draw[usual] (1.8,-0.6) to [ out =0 , in = 0](2,1);
\draw[usual] (0,1) to [ out =180 , in = 175] (0.25,-0.6);
\draw[usual] (0.25,-0.6) to [ out =0 , in = -135](2,0);
\node at (0,1)[circle,fill,inner sep=1.55pt]{};
\node at (2,0)[circle,fill,inner sep=1.55pt]{};
\node at (0,0)[circle,fill,inner sep=1.55pt]{};
\node at (1,0)[circle,fill,inner sep=1.55pt]{};
\node at (1,1)[circle,fill,inner sep=1.55pt]{};
\node at (2,1)[circle,fill,inner sep=1.55pt]{};
\node[below=5pt] at (current bounding box.south) {};
\end{tikzpicture}
,
\end{gather*}
also known as $K_{3,3}$, with the additional perfect matching specified. We can input $K_{3,3}$ as
\begin{lstlisting}[language=Python]
pairs = [[5, 3], [2, 4], [3, 1], [0, 2], [1, 5], [4, 0]]
matching = [[0, 1], [5, 2], [4, 3]]
\end{lstlisting}
which outputs
\begin{lstlisting}[language=Python]
circle counts by homological degree: [[3], [2, 2, 2], [1, 1, 1], [2]]
d^2 = 0: True
homology dimensions by homological degree: [8, 2, 0, 6]
homology dimensions by internal degree: {0: [8, 2, 0, 6]}
Poincare polynomial: 6*t**3 + 2*t + 8
\end{lstlisting}
for $n =1$, and for $n=2$:
\begin{lstlisting}[language=Python]
circle counts by homological degree: [[3], [2, 2, 2], [1, 1, 1], [2]]
d^2 = 0: True
homology dimensions by homological degree: [126, 6, 0, 30]
homology dimensions by internal degree: {0: [63, 3, 0, 16], 1: [63, 3, 0, 14]}
Poincare polynomial: 14*q*t**3 + 3*q*t + 63*q + 16*t**3 + 3*t + 63
\end{lstlisting}
Note that we could take $K_{3, 3}$ without the specified perfect matching and consider the blowup with the canonical perfect matching instead, though the resulting perfect matching graph would of course be much more complicated.
\end{Example}

\begin{Example}
Finally, consider the Petersen graph with following perfect matching:
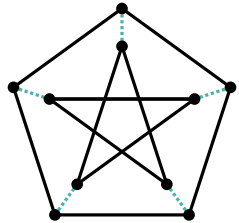
\begin{align*}
 \begin{tikzpicture}
\node[regular polygon, draw=white, regular polygon sides=5, 
          minimum size=3 cm, 
          inner sep=0pt, name=1]    
          at (0, 0) {};
\node[regular polygon, draw=white, regular polygon sides=5, 
          minimum size=2 cm, 
          inner sep=0pt, name=2]    
          at (0, 0) {};
\draw[match] (1.north) to (2.north);
\draw[match] (1.corner 2) to (2.corner 2);
\draw[match] (1.corner 3) to (2.corner 3);
\draw[match] (1.corner 4) to (2.corner 4);
\draw[match] (1.corner 5) to (2.corner 5);
\draw[usual] (1.north) to (1.corner 2);
\draw[usual] (1.corner 2) to (1.corner 3);
\draw[usual] (1.corner 3) to (1.corner 4);
\draw[usual] (1.corner 4) to (1.corner 5);
\draw[usual] (1.corner 5) to (1.north);
\draw[usual] (2.north) to (2.corner 3);
\draw[usual] (2.north) to (2.corner 4);
\draw[usual] (2.corner 2) to (2.corner 4);
\draw[usual] (2.corner 2) to (2.corner 5);
\draw[usual] (2.corner 5) to (2.corner 2);
\draw[usual] (2.corner 5) to (2.corner 3);
\node at (1.north)[circle,fill,inner sep=1.55pt]{};
\node at (1.corner 2)[circle,fill,inner sep=1.55pt]{};
\node at (1.corner 3)[circle,fill,inner sep=1.55pt]{};
\node at (1.corner 4)[circle,fill,inner sep=1.55pt]{};
\node at (1.corner 5)[circle,fill,inner sep=1.55pt]{};
\node at (2.north)[circle,fill,inner sep=1.55pt]{};
\node at (2.corner 2)[circle,fill,inner sep=1.55pt]{};
\node at (2.corner 3)[circle,fill,inner sep=1.55pt]{};
\node at (2.corner 4)[circle,fill,inner sep=1.55pt]{};
\node at (2.corner 5)[circle,fill,inner sep=1.55pt]{};
    \end{tikzpicture}
\end{align*}
where the input may be given as:
\begin{lstlisting}[language=Python]
pairs = [[4, 1], [0, 2], [1, 3], [2, 4], [3, 0], [8, 7], [9, 8], [5, 9], [6, 5], [7, 6]]
matching = [[0, 5], [1, 6], [2, 7], [3, 8], [4, 9]]
\end{lstlisting}
The output is
\begin{lstlisting}[language=Python]
circle counts by homological degree: [[1], [1, 1, 1, 1, 1], [2, 1, 1, 2, 2, 1, 1, 2, 1, 2], [1, 2, 1, 2, 2, 1, 1, 2, 2, 1], [1, 1, 1, 1, 1], [1]]
d^2 = 0: True
homology dimensions by homological degree: [1, 1, 8, 8, 1, 1]
homology dimensions by internal degree: {0: [1, 1, 8, 8, 1, 1]}
Poincare polynomial: t**5 + t**4 + 8*t**3 + 8*t**2 + t + 1
\end{lstlisting}
\begin{lstlisting}[language=Python]
circle counts by homological degree: [[1], [1, 1, 1, 1, 1], [2, 1, 1, 2, 2, 1, 1, 2, 1, 2], [1, 2, 1, 2, 2, 1, 1, 2, 2, 1], [1, 1, 1, 1, 1], [1]]
d^2 = 0: True
homology dimensions by homological degree: [3, 3, 66, 66, 3, 3]
homology dimensions by internal degree: {0: [2, 2, 34, 34, 2, 2], 1: [1, 1, 32, 32, 1, 1]}
Poincare polynomial: q*t**5 + q*t**4 + 32*q*t**3 + 32*q*t**2 + q*t + q + 2*t**5 + 2*t**4 + 34*t**3 + 34*t**2 + 2*t + 2
\end{lstlisting}
for $n=1, 2$ respectively. Note that the homology here is much stronger at its decatergorification at $t=-1$, since the latter vanishes.
\end{Example}


\end{document}